\documentclass[11pt]{article}
\usepackage{float}
\usepackage{amsfonts,amsmath,amssymb,graphicx,color}
\usepackage{amsthm}
\usepackage{algorithm}
\usepackage{algpseudocode}
\usepackage{psfrag}
\usepackage{pdflscape}
\usepackage{multirow} 
\usepackage{enumitem}
\usepackage{framed}
\usepackage{sidecap}
\usepackage{tikz}
\usetikzlibrary{arrows}
\usetikzlibrary{patterns,arrows,decorations.pathreplacing}
\graphicspath{ {./figures/} }
    \usepackage{soul}
    \usepackage{mdframed}
\usepackage[normalem]{ulem}
\usepackage{url}
\usepackage[numbers, sort&compress ]{natbib}  

\numberwithin{figure}{section}
\numberwithin{table}{section}

\def\E{\mathbb{E}}

\newcommand{\bequ}{\begin{equation}}     \newcommand{\eequ}{\end{equation}}
\newcommand{\benn}{\begin{equation*}}    \newcommand{\eenn}{\end{equation*}}
\newcommand{\bbma}{\begin{bmatrix}}      \newcommand{\ebma}{\end{bmatrix}}

\newcommand{\var}{\mbox{Var}}

\newcommand{\R}{\mathbb{R}}

\newcommand{\bsub}{\begin{subequations}}
\newcommand{\esub}{\end{subequations}}

\newtheorem{thm}{Theorem}[section]

\newtheorem{lem}[thm]{Lemma}

\numberwithin{equation}{section}

\newcommand{\comment}[1]{}

\newcommand{\be}{\begin{equation}}
\newcommand{\ee}{\end{equation}}

\newcommand{\bea}{\begin{eqnarray}}
\newcommand{\eea}{\end{eqnarray}}
\newcommand{\beqa}{\begin{eqnarray}}
\newcommand{\eeqa}{\end{eqnarray}}
\newcommand{\beann}{\begin{eqnarray*}}
\newcommand{\eeann}{\end{eqnarray*}}

\newcommand{\bmat}{\left[ \begin{array}}
\newcommand{\emat}{\end{array} \right]}

\newcommand{\beq}{\begin{equation}}
\newcommand{\eeq}{\end{equation}}

\newcommand{\bproof}{\begin{description} \item[{\it Proof}.] ~ }
\newcommand{\eproof}{\hspace*{\fill}$\Box$\medskip \end{description}}

\newcommand{\defeq}{\stackrel{\rm def}{=}}

\newcounter{algo}[section]

\newcounter{prog}[section]

\title{Adaptive Sampling Strategies for Stochastic Optimization}
\author{       
        Raghu Bollapragada\thanks{Department of Industrial Engineering and Management Sciences, Northwestern University, 
       Evanston, IL, USA.  This author was supported by the Office of Naval Research grant N00014-14-1-0313 P00003.}
       \and     
       Richard Byrd \thanks{Department of Computer Science, University of Colorado,
        Boulder, CO, USA.  This author was supported by National Science Foundation Grant DMS-1620070.}
        \and
       Jorge Nocedal \thanks{Department of Industrial Engineering and Management Sciences, Northwestern University, 
       Evanston, IL, USA.  This author was supported by National Science Foundation grant DMS-0810213, and by Department 
       of Energy grant DE-FG02-87ER25047.} 
      }
\date{\today}
\begin{document}

\maketitle

\maketitle

\begin{abstract}
	In this paper, we propose a stochastic optimization  method that adaptively controls the sample size used in the computation of gradient approximations.  Unlike other variance reduction techniques that either require additional storage or the regular computation of full gradients, the proposed method reduces variance by increasing the sample size as needed. The decision to increase the sample size is governed by an \emph{inner product test} that
ensures that search directions are descent directions with high probability.
	 We show that the inner product test improves upon the well known \emph{norm test}, and can be used as a basis for an algorithm that is globally convergent on nonconvex functions and enjoys a global linear rate of convergence on strongly convex functions. Numerical experiments on logistic regression problems illustrate the performance of the  algorithm. 
\end{abstract}

\newpage
%\tableofcontents
\newpage
\section{Introduction}
\label{intro}
\setcounter{equation}{0}

This paper presents a first-order stochastic optimization method 
that progressively changes the size of the sample used in the gradient approximation with the aim of achieving overall efficiency. The algorithm starts by choosing a small sample, and increases it as needed so that the gradient approximation is  accurate enough to yield  a  linear rate of convergence for strongly convex functions. Adaptive sampling methods of this type are appealing because they enjoy optimal complexity properties \cite{byrd2012sample,friedlander2012hybrid} and have the potential of being effective on a wide range of applications. Theoretical guidelines for controlling the sample size have been established in the literature \cite{byrd2012sample,friedlander2012hybrid,2014pasglyetal}, but the design of practical implementations has proven to be difficult. For example, the mechanism studied in \cite{byrd2012sample,2014hasghopasWSC,cartis2015global}, although intuitively appealing,  is often inefficient in practice for reasons discussed below.

The problem of interest is 
\[
\min_{x \in \R^d} \ \mathbb{E}[\phi(x;\xi)],
\]
where $\phi : \R^d \rightarrow R$ is a smooth function and $\xi$ is a random variable. A particular instance of this problem arises in machine learning, where it takes the form 
\begin{equation}\label{prob}
\min_{x \in \R^{d}}   F(x) = \int f(x; z,y) dP( z,y).
\end{equation}
In this setting,  $f$ is the composition of a prediction function (parametrized by a vector $x $) and a smooth loss function, and $(z,y)$ are random input-output pairs with probability distribution $P(z,y)$.  We call $F$  the \emph{expected risk}.

Often, problem \eqref{prob} cannot be tackled directly because  the joint probability distribution $P(z,y)$ is unknown. In this case, one draws a data set  $\{(z^i,y^i)\}$, $i=1,\ldots, N$,   from the distribution $P(z,y)$, and minimizes  the  \emph{empirical risk}  
\[
R(x) = \frac{1}{N} \sum_{i =1}^N f(x;z^i,y^i) .
\]
We define 
$
F_i(x) \defeq f(x;z^i,y^i),
$
so that the empirical risk can be written conveniently as
\begin{equation}\label{probrr}
R(x) = \frac{1}{N} \sum_{i =1}^N F_i(x).
\end{equation}     
One may view the optimization algorithm as being applied directly to the expected risk $F$ or to the empirical risk $R$. We state our algorithm and establish a convergence result in terms of the minimization of $F$. Later on, in Section~\ref{practical}, we discuss a practical implementation designed to minimize $R$.

An approximation to the gradient of $F$ can be obtained by sampling. At the iterate $x_k$, we define 
\begin{equation}   \label{subs}
\nabla F_{S_k}(x_k)=\frac{1}{|S_k|}\sum_{i \in S_k} \nabla F_i(x_k),
\end{equation}
where the set $ S_k \subset\{1, 2, \ldots \}$ indexes certain data points $ (z^i,y^i)$. A first-order method based on this gradient approximation is then given by
\begin{equation}   \label{searchd}
x_{k+1}=x_k - \alpha_k \nabla F_{S_k}(x_k), \quad\ \alpha_k >0.
\end{equation}
In our approach, the sample  $S_k$ changes at every iteration, and its size $|S_k|$ is determined by a mechanism described in the next section. It is based on an \emph{inner product test} that ensures that the search direction in \eqref{searchd} is a descent direction with high probability.   
In contrast to the test studied in [6-8, 14], which we call \emph{the norm test}, and which  controls both the direction
 and length of the  gradient approximation and promotes  search directions that are close to the true gradient, the inner product test places more emphasis on generating  descent directions and allows more freedom in their length. 

The numerical results presented in Section~\ref{numerical} suggest that the inner product test is efficient in practice, but in order to establish a Q-linear convergence rate for strongly convex functions, we must reinforce it with an additional mechanism that prevents search directions from becoming nearly orthogonal to the true gradient $\nabla F(x_k)$. More precisely, we introduce an  \emph{orthogonality test} that ensures that the variance of sampled gradients along the direction orthogonal to $\nabla F(x_k)$ is properly controlled. The orthogonality test is invoked infrequently in practice and should be regarded as a safeguard against rare difficult cases. 

An important component of Algorithm \eqref{searchd} is the selection of the steplength  $\alpha_k$.  One option   is to use a fixed value $\alpha$ that is selected for each problem after careful experimentation. An alternative that we explore in more depth is a backtracking line search that imposes sufficient decrease in the sampled function
 \begin{equation}    \label{sfunc}
F_{S_k}(x)=\frac{1}{|S_k|}\sum_{i \in S_k}  F_i(x),
\end{equation}
and that is controlled by an adaptive estimate of the Lipschitz constant $L$ of the gradient. %$\nabla F$. 
Similar strategies have been considered for deterministic problems (see e.g. \cite{beck2009fast}), but the stochastic setting provides some challenges and  opportunities that we explore in our line search procedure.

This paper is organized into five sections. A literature review and a summary of our notation are presented in the rest of this section. In Section~\ref{sec:ip}, we describe the inner product test, and in Section~\ref{sec:or} we introduce the orthogonality test and establish convergence analysis of an adaptive sampling algorithm that employs both tests. In Section~\ref{practical}, we discuss some practical implementation issues, and present a full description of the algorithm.  Numerical results are presented in Section~\ref{numerical}, and in Section~\ref{sec:final} we make some concluding remarks.

\subsection{Literature Review}
Optimization methods that progressively increase  sample sizes have been studied in \cite{homem2003variable,royset2013optimal,friedlander2012hybrid,byrd2012sample,2014pasglyetal,roosta2016sub,roosta2016sub1,de2017automated}.  Friedlander and Schmidt \cite{friedlander2012hybrid} consider the finite sum problem \eqref{probrr} and show linear convergence  by increasing  $|S_k|$ at a geometric rate. They also experiment with a quasi-Newton version of their algorithm. Byrd et al. \cite{byrd2012sample} study the minimization of expected risk \eqref{prob} and show linear convergence when the sample size grows geometrically, and provide computational complexity bounds. They propose the \emph{norm test} as a practical procedure for controlling the sample size.
Pasupathy et al. \cite{2014pasglyetal} study more generally the effect of sampling rates on the convergence and complexity of various optimization methods. Hashemi et al. \cite{2014hasghopasWSC} consider a test that is similar to the norm test, which is reinforced by a back up mechanism that ensures a geometric increase in the sample size. They motivate this approach from a stochastic simulation perspective and using variance-bias ratios. Cartis and Scheinberg  \cite{cartis2015global} relax the norm test by allowing it to be violated with a probability less than 0.5, and this ensures that the search directions are successful descent directions  more than $50 \%$ of the time. They use techniques from stochastic processes and analyze algorithms that perform a line search using the true function values $F(x_k)$. Bollapragada et al. \cite{bollapragada2016exact} study methods that sample the gradient and Hessian, and establish conditions for global linear convergence. They also provide a superlinear convergence result in the case when the gradient samples are increased at rates faster than geometric and Hessian samples are increased without bound (at any rate).

The adaptive sampling methods studied here can be regarded as variance reducing methods; see the survey \cite{bottou2016optimization}.  Other noise reducing methods include  stochastic aggregated gradient methods, such as SAG \cite{schmidt2013minimizing}, SAGA  \cite{defazio2014saga}, and SVRG \cite{johnson2013accelerating}. These methods either compute the full gradient at regular intervals, as in  SVRG,  or require storage of the component gradients, as in SAG or SAGA. These methods have gained much popularity in recent years, as they are able to achieve a linear rate of convergence for the finite sum problem, with a very low iteration cost.

%%%
\subsection{Notation}
We denote the variables of the optimization problem by $x \in \R^d$, and a minimizer of the objective $F$ as $x^*$. Throughout the paper,   $\| \cdot \|$ denotes the $\ell_2$ vector norm. The notation $A \preceq B$ means that $B-A$ is a symmetric and positive semi-definite matrix. 

\section{The Inner Product Test}  \label{sec:ip}
Let us consider how to select the sample size in the first-order stochastic optimization method
\begin{equation} \label{iter}
x_{k+1} = x_k - \alpha_k \nabla F_{S_k}(x_k) .
\end{equation}
Here $\alpha_k >0$ is the steplength parameter and  the sampled gradient $\nabla F_{S_k}(x_k)$ is defined in \eqref{subs}. We propose to determine the sample size $|S_k|$ at every iteration through the \emph{inner product test}  described below,  which  aims to ensure that the algorithm generates descent directions sufficiently often. 
We recall that the search direction  of algorithm \eqref{iter} is a descent direction for $F$ if
\[
\nabla F_{S_k}(x_k)^T \nabla F(x_k) > 0.
\] 
This condition will not hold at every iteration of our algorithm, but if $S_k$ is chosen uniformly at random from $\{1, 2, \ldots \}$, it will hold in expectation,  i.e.,
\begin{equation}  \label{recall}
\E\left[\nabla F_{S_k}(x_k)^T\nabla F(x_k)\right] = \|\nabla F(x_k)\|^2 > 0 .
\end{equation}
We must in addition control the variance of the term on the left hand side to guarantee that the iteration \eqref{iter} is convergent. We do so by requiring that the sample size $|S_k|$ be large enough so that the following  condition is satisfied
\begin{equation} \label{variancet}
\E \left[\left( \nabla F_{S_k}(x_k)^T\nabla F(x_k) - \|\nabla F(x_k)\|^2 \right)^2\right] \leq \theta^2 \|\nabla F(x_k)\|^4, \quad \mbox{for some} \quad \theta > 0 . 
\end{equation} 
The  left hand side is difficult to compute but can be bounded by the true variance of individual gradients, i.e.,
\begin{equation} \label{inner}
\E \left[\left( \nabla F_{S_k}(x_k)^T\nabla F(x_k) - \|\nabla F(x_k)\|^2 \right)^2\right] \leq \frac{\E \left[\left( \nabla F_{i}
	(x_k)^T\nabla F(x_k) - \|\nabla F(x_k)\|^2 \right)^2\right]}{|S_k|}.
\end{equation} 
Therefore, the following condition ensures   \eqref{variancet} 
\begin{equation} \label{inner-i}
\frac{\E \left[\left( \nabla F_{i}(x_k)^T\nabla F(x_k) - \|\nabla F(x_k)\|^2 \right)^2\right]}{|S_k|} \leq \theta^2 \|\nabla F(x_k)\|^4.
\end{equation}	
We refer to \eqref{inner-i} as the \emph{(exact variance) inner product test.}
In large-scale applications, the computation of  $\nabla F(x_k)$ can be prohibitively expensive, but we can approximate the variance on the left side of \eqref{inner-i} with the sample variance and the  gradient $\nabla F(x_k)$ on the right side with a sampled gradient, to obtain
	\begin{equation} \label{final}
  \frac{{\rm Var}_{i \in S_k}(\nabla F_i(x_k)^T\nabla F_{S_k}(x_k))}{|S_k|} \leq \theta^2 \|\nabla F_{S_k}(x_k)\|^4 ,
	\end{equation}
where
\[
{\rm Var}_{i \in S_k}\left(\nabla F_i(x_k)^T\nabla F_{S_k}(x_k)\right) = \frac{1}{|S_k| -1}\sum_{i \in S_k}\left(\nabla F_i(x_k)^T\nabla F_{S_k}(x_k) - \|\nabla F_{S_k}(x_k)\|^2\right)^2.
\]
Condition \eqref{final} will be called the \emph{(approximate) inner product test}. Whenever it is not satisfied, we increase the sample size $|S_k|$ to one that we predict will satisfy \eqref{final}. An outline of this approach is given in Algorithm~\ref{alg:Adasample}.

%%%%%%%%%%%% ALGORITHM %%%%%%%%%%%%%%%%
\begin{algorithm}[htp] 
	\caption{Basic Version}
	\label{alg:Adasample}
	{\bf Input:} Initial iterate $x_0$, initial sample  $S_0$, and a constant $\theta > 0$.\\
	Set $k \leftarrow 0$\\
	{\bf Repeat} until a convergence test is satisfied:
	\begin{algorithmic}[1]
		\State Compute $d_k = -\nabla F_{S_k}(x_k)$
		\State Choose a steplength $\alpha_k > 0$
		\State Compute new iterate: $x_{k+1} = x_k + \alpha_k d_k$
		\State Set $k \leftarrow k + 1$
		\State Choose a new sample $S_k$ such that the condition \eqref{final} is satisfied
		%\State Compute the sample variance defined in \eqref{final}.
		%\State If condition \eqref{scvr-final-condition} is not satisfied, Choose $|S_k|$ using formula \eqref{heuristic-samplesize}. 
	\end{algorithmic}
\end{algorithm}

\medskip\noindent
In Section~4, we discuss how to implement the inner product test in practice,  how to choose the parameter $\theta$ and the stepsize $\alpha_k$, as well as the strategy for increasing the size of a new sample $S_k$, when the algorithm calls for it.  

It is illuminating to compare the inner product test\footnote{We use the term \emph{inner product test} to refer to \eqref{inner-i} or \eqref{final} when the distinction is not important in the discussion.} with a related rule studied in the literature  
\cite{byrd2012sample,2014hasghopasWSC,cartis2015global} that we call the \emph{norm test.}  The  comparison can be most simply and clearly seen in  the deterministic setting with the gradient based method  $x_{k+1}= x_k - \alpha_k g_k$, where $g_k$ is some approximation to the gradient $\nabla F(x_k)$. In this context,  the deterministic analog of \eqref{variancet} is
\begin{equation} \label{varianced}
\left( g_k^T\nabla F(x_k) - \|\nabla F(x_k)\|^2 \right)^2 \leq \theta^2 \|\nabla F(x_k)\|^4 .
\end{equation} 
In contrast, the norm test corresponds to 
\begin{equation} \label {carter}
\|g_k - \nabla F(x_k)\|^2 \leq \theta^2 \|\nabla F(x_k)\|^2, \qquad \mbox{for some} \quad \theta \in [0,1).
\end{equation}
This rule was studied by Carter \cite{carter} in the context of trust region methods with inaccurate gradients.
It is easy to see that \eqref{carter} ensures that  $g_k$ is a descent direction, but it is not a necessary condition;  in fact \eqref{carter} is more restrictive than  \eqref{varianced} because it requires approximate gradients  to lie in a  ball centered at the true gradient $\nabla F(x_k)$, whereas \eqref{varianced} allows gradients that are within an infinite band  around the true gradient, as illustrated in Figure~\ref{fig:Angle tests}. 

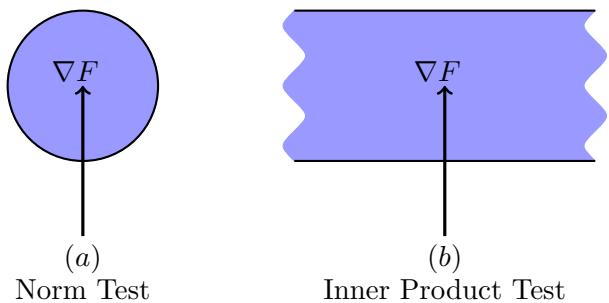
\begin{figure}[H]
	\begin{tikzpicture}[
	scale=2,
	axis/.style={very thick, ->, >=stealth'},
	important line/.style={thick},
	dashed line/.style={dashed, thin},
	pile/.style={thick, ->, >=stealth', shorten <=2pt, shorten
		>=2pt},
	every node/.style={color=black}
	]  
	\hspace{3.5cm}
	\filldraw[fill=blue!40!white, draw=blue!40!white](0,1) circle (0.5cm);
	\draw[thick] (0,1) circle (0.5cm);
	\draw[very thick,->] (0,0) coordinate (A) -- (0,1)
	coordinate (B) node[right, text width=5em] {};
	\fill[white] (-0.05,1.1) circle (0.0000003pt) node[] [black]{$ \nabla F$};
	\fill[white] (0,-0.15) circle (0.0000003pt) node[] [black]{$(a)$};
	\fill[white] (0,-0.35) circle (0.0000003pt) node[] [black]{Norm Test};
	\begin{scope}[xshift=40]
	\filldraw[fill=blue!40!white, draw=blue!40!white] (2,0.5) -- (0,0.5) .. controls (-0.1, 0.6) .. (0.0, 0.7) .. controls  (0.1, 0.8) .. (0.0, 0.9) .. controls  (-0.1, 1.0) .. (0.0, 1.1) .. controls  (0.1, 1.2) .. (0.0, 1.3) .. controls  (-0.1, 1.4) .. (0.0, 1.5) -- (2, 1.5) .. controls (1.9, 1.4) .. (2, 1.3) .. controls (2.1, 1.2) .. (2, 1.1) .. controls (1.9, 1.0) .. (2, 0.9) .. controls (2.1, 0.8) .. (2, 0.7) .. controls (1.9, 0.6) .. (2, 0.5);
	\draw[thick] (0,0.5) coordinate (A) -- (2,0.5)
	coordinate (B) node[right, text width=5em] {};
	\draw[thick] (0,1.5) coordinate (A) -- (2,1.5)
	coordinate (B) node[right, text width=5em] {};
	\draw[very thick,->] (1,0) coordinate (A) -- (1,1)
	coordinate (B) node[right, text width=5em] {};
	\fill[white] (0.95,1.1) circle (0.0000003pt) node[] [black]{$ \nabla F$};    
	\fill[white] (1,-0.15) circle (0.0000003pt) node[] [black]{$(b)$};
	\fill[white] (1,-0.35) circle (0.0000003pt) node[] [black]{Inner Product Test};
	\end{scope}
	\end{tikzpicture}
	\vspace{0.1 cm}
	\caption{Deterministic setting. Given a gradient $\nabla F$, the shaded areas denote the set of vectors $g$ satisfying (a): the norm condition \eqref{carter}; (b) the deterministic inner product condition \eqref{varianced}.}
	\label{fig:Angle tests}
\end{figure}

In the stochastic setting, the norm condition \eqref{carter} becomes
\begin{equation} \label{normt}
  {\E[\|\nabla F_{S_k}(x_k) - \nabla F(x_k)\|^2]} \leq \theta^2 \|\nabla F(x_k)\|^2.
\end{equation}
Following the same reasoning as in \eqref{inner},  this condition will be satisfied if we impose instead
\begin{equation} \label{dss-sample-test}
\frac{\E[\|\nabla F_{i}(x_k) - \nabla F(x_k)\|^2]}{|S_k|} \leq \theta^2 \|\nabla F(x_k)\|^2.
\end{equation}
This \emph{norm test} is used in \cite{byrd2012sample} to control the sample size: if \eqref{dss-sample-test} is not satisfied, then sample size is increased.

Numerical experience indicates that the norm test can be unduly restrictive, often leading to a very  fast increase in the sample size, negating the benefits of adaptive sampling. 
An indication that the inner product test increases the sample size more slowly than the norm test can be see through the following argument. Let $|S_{i}|$, $|S_{n}|$ represent the minimum number of samples required to satisfy the inner product test \eqref{inner-i} and the norm test \eqref{dss-sample-test}, respectively,  at any given iterate $x$, using the same value of $\theta$. A simple computation (see Appendix~A) shows that
\begin{equation}   \label{beta1}
	\frac{|S_{i}|}{|S_{n}|} = \beta(x) \leq 1, 
\end{equation}
	where 
\begin{equation} \label{beta2}
	\beta(x) =\frac{\E[\|\nabla F_{i}(x)\|^2 \cos^2(\chi_i)]- \|\nabla F(x)\|^2}{\E[\|\nabla F_{i}(x)\|^2] - \|\nabla F(x)\|^2}
\end{equation}
	and $\chi_i$ is the angle between $\nabla F_i(x_k)$ and $\nabla F(x_k)$. The quantity $\beta(x)$ is the ratio of the variance  the individual gradients along the true gradient direction and the total variance of the individual gradients.
The numerical results presented in Section~\ref{numerical} are consistent with this observation and show that $\beta(x_k)$ is often much less than~1.

\section{ Analysis}  \label{sec:or}
In order to establish linear convergence  for this method, it is necessary to introduce an additional condition that has only a slight effect on the algorithm in practice, but guarantees the quality of the search direction in  difficult cases. In this section, we first describe this test, and in the second part we establish some results on convergence rates.

\subsection{Orthogonality Test}
Establishing a convergence rate usually involves showing that the step direction is bounded away from orthogonality to $\nabla F(x_k)$. However,  iteration \eqref{iter} with a sample $S_k$ satisfying the inner product condition \eqref{inner-i} does not necessarily enjoy this property. The possible near orthogonality corresponds to the case where the ratio $\beta(x)$ defined above is near zero and occurs when the variance in the individual gradients is very large compared to the variance in the individual gradients along the true gradient direction.
% which makes the step direction to be near orthogonal to  $\nabla F(x_k)$.} 
 %\rb{ the component of the step orthogonal to  $\nabla F(x_k)$ is extremely large. This would then the step to be nearly orthogonal to the true gradient.  
 Although we have not observed  very small values of $\beta(x_k)$ in our numerical tests, this is harmful in principle and to prove convergence we must be able to avoid this possibility.   
%and can even fail to converge because \eqref{variancet}  allows sample gradients $\nabla F_{S_k}$ that are arbitrarily long relative to $\| \nabla F(x_k) \|$ if they are nearly orthogonal to $\nabla F(x_k)$. 
We propose a test that imposes a loose bound on the component of $\nabla F_{S_k}(x_k)$ orthogonal  to the true gradient.
%on the angle between  $\nabla F_{S_k}$ and $\nabla F(x_k)$, in expectation. 
This test, together with \eqref{inner-i}, allows us to prove a linear convergence result for the adaptive sampling algorithm,  when $F$ is strongly convex. 

To motivate the orthogonality test, we note that the component of $\nabla F_{S_k}(x_k)$ orthogonal to $\nabla F(x_k)$ is 0 in expectation, i.e., 
\begin{equation*} 
\E\left[\nabla F_{S_k}(x_k) - \frac{\nabla F_{S_k}(x_k)^T\nabla F(x_k)}{\|\nabla F(x_k)\|^2}\nabla F(x_k)\right] =0,
\end{equation*}  
but that is not sufficient.
%Now in order to bound the variance of this component, we need to limit the maximum value that can be taken. This value is decided based on the maximum acute angle that is allowed in the method. From the descent test, \textcolor{red}{\eqref{inner-i}  we have that} the maximum component along the gradient direction is bounded and is given by,
%\begin{equation}
%\frac{\sqrt{\E \left[\left( \nabla F_{S_k}(x_k)^T\nabla F(x_k)\right)^2\right]}}{\|\nabla F(x_k)\|} \leq \sqrt{1 + \theta^2}\|\nabla F(x_k)\|.
%\end{equation}
We must bound the variance of this orthogonal component, and to achieve this we require that  the sample size $|S_k|$ be large enough to satisfy,
\begin{equation} \label{orthogonal}
\E\left[\left\|\nabla F_{S_k}(x_k) - \frac{\nabla F_{S_k}(x_k)^T\nabla F(x_k)}{\|\nabla F(x_k)\|^2}\nabla F(x_k)\right\|^2\right] \leq \,\nu^2 \, \|\nabla F(x_k)\|^2 ,
\end{equation} 
for some positive constant $\nu$ whose choice is discussed in Section~\ref{practical}. 
For a given sample size, this condition can be expressed, using the true variance of individual gradients, as
\begin{equation} \label{orth-i}
\frac{\E\left[\left\|\nabla F_i(x_k) - \frac{\nabla F_i(x_k)^T\nabla F(x_k)}{\|\nabla F(x_k)\|^2}\nabla F(x_k)\right\|^2\right]}{|S_k|} \leq \nu^2 \|\nabla F(x_k)\|^2.
\end{equation} 
The \emph{(exact variance) orthogonality test} states that if this inequality is not satisfied, the sample size $|S|$ should be increased. 

Reasoning as in \eqref{final}, we can derive a variant of the orthogonality test based on sample approximations. This \emph{(approximate) orthogonality test} is given by
\begin{equation} \label{orthogonal-p}
\frac{1}{|S_k| - 1} \frac{\sum_{i \in S_k}\left\|\nabla F_i(x_k) - \frac{\nabla F_i(x_k)^T\nabla F_{S_k}(x_k)}{\|\nabla F_{S_k}(x_k)\|^2}\nabla F_{S_k}(x_k)\right\|^2}{|S_k|} \leq \nu^2 \|\nabla F_{S_k}(x_k)\|^2.
\end{equation}

It is  interesting to note that, since \eqref{variancet} implies that the root mean square of the component of the step along the gradient is bounded below by $\sqrt{1 - \theta^2}\|\nabla F(x_k)\|$,  imposition of \eqref{orth-i} will tend to keep the tangent of the angle $\chi_k$ between $\nabla F_{S_k}(x_k)$ and $\nabla F(x_k)$ below $\nu/ \sqrt{1 - \theta^2}$, providing a limit on the near orthogonality of these two vectors.

%We also note that we want the angle to be bounded away from orthogonal. Therefore we say $\tan (\gamma) = \delta < \infty$. \textcolor{red}{[do we ever use this $\delta $ notation?]}
\subsection{Convergence Analysis}
The orthogonality test, in conjunction with the inner product test allows the algorithm to make sufficient progress at every iteration, in expectation. More precisely, we now establish three convergence results for the \emph{exact} versions of these two tests, namely \eqref{inner-i} and \eqref{orth-i}. Our results apply to iteration \eqref{iter} with a fixed steplength.
We start by establishing a technical lemma.

%%%%%%%%%%%%%%%%%%%%%%%%%%%%%%%%%%%%%%%%%%%%%%
\begin{lem} \label{c-thmlin} 
	Suppose that $F$ is twice continuously differentiable and that  there exists a constant $L > 0$ such that 
	\begin{equation}  \label{c-lip}
	 \nabla^2 F(x) \preceq L I, \quad \forall x \in \R^d.
	\end{equation} 
	Let $\{x_k\}$ be the iterates generated by  iteration \eqref{iter} with any $x_0$, where $|S_{k}|$ is chosen such that the (exact variance) inner product test \eqref{inner-i} and the (exact variance) orthogonality test \eqref{orth-i}  are satisfied at each iteration for any given  constants $\theta > 0$ and $\nu > 0$. Then, for any $k$,
	\begin{align}\label{c-sample-norm}
	\E\left[\|\nabla F_{S_k}(x_k)\|^2\right] 
	%&\leq (1 + \theta^2)\|\nabla F(x_k)\|^2 + \nu^2 \|\nabla F(x_k)\|^2 \nonumber \\
	& \leq(1 + \theta^2 + \nu^2)\|\nabla F(x_k)\|^2 .
	\end{align}
	Moreover, if the steplength satisfies 
	\begin{equation}   \label{c-stepform}
	\alpha_k= \alpha \leq \frac{1}{(1 + \theta^2+\nu^2)L},
	\end{equation}
	we have that 
	\begin{align} 
	\E[F(x_{k+1})] & \leq \E[F(x_k)]  - \frac{\alpha}{2} \|\nabla F(x_k)\|^2   \label{c-lin-ineq} .
	\end{align}  
\end{lem}
\begin{proof} Since \eqref{orth-i} is satisfied, we have that \eqref{orthogonal} holds. Thus, recalling 
	 \eqref{recall} we have that \eqref{orthogonal} can be written as
	
	\begin{align} \nonumber
	\E\left[\left\|\nabla F_{S_k}(x_k) -  \frac{\nabla F_{S_k}(x_k)^T\nabla F(x_k)}{\|\nabla F(x_k)\|^2}\nabla F(x_k)\right\|^2\right] & ~~~~~~~~~~~~~~~~~~~~~~~~~~~~~~~~~~~~~~~~~~~~~~~\end{align}
\begin{align}
	 & ~~~~~~~~~ = ~ \E\left[\|\nabla F_{S_k}(x_k)\|^2\right] -\frac{2 \E \left[\left( \nabla F_{S_k}(x_k)^T\nabla F(x_k)\right)^2\right]}{\|\nabla F(x_k)\|^2} 
	 + \frac{\E \left[\left( \nabla F_{S_k}(x_k)^T\nabla F(x_k)\right)^2\right]}{\|\nabla F(x_k)\|^2} \nonumber\\
	&~~~~~~~~~ = ~ \E\left[\|\nabla F_{S_k}(x_k)\|^2\right] - \frac{\E \left[\left( \nabla F_{S_k}(x_k)^T\nabla F(x_k)\right)^2\right]}{\|\nabla F(x_k)\|^2}  \nonumber \\
	&  ~~~~~~~~~ \leq ~ \nu^2 \,\|\nabla F(x_k)\|^2. \nonumber
	\end{align}
	Therefore,
	
	\begin{equation}\label{step-two-term}
	\E\left[\|\nabla F_{S_k}(x_k)\|^2\right] \leq \frac{\E \left[\left( \nabla F_{S_k}(x_k)^T\nabla F(x_k)\right)^2\right]}{\|\nabla F(x_k)\|^2} + \nu^2 \|\nabla F(x_k)\|^2.
	\end{equation}
	To bound the first term on the right side of this inequality, we use the inner product test. Since $|S_k|$  satisfies   \eqref{inner-i}, the inequality \eqref{variancet} holds, and this in turn yields
	\begin{align*} 
	\E \left[\left( \nabla F_{S_k}(x_k)^T\nabla F(x_k)\right)^2\right] &\leq  \|\nabla F(x_k)\|^4 + \theta^2 \|\nabla F(x_k)\|^4 \\
	&= (1 + \theta^2) \|\nabla F(x_k)\|^4.
	\end{align*}  
	Substituting in \eqref{step-two-term}, we get the following bound on the length of the search direction:
	\begin{align*}
	\E\left[\|\nabla F_{S_k}(x_k)\|^2\right] 
	%&\leq (1 + \theta^2)\|\nabla F(x_k)\|^2 + \nu^2 \|\nabla F(x_k)\|^2 \nonumber \\
	& \leq(1 + \theta^2 + \nu^2)\|\nabla F(x_k)\|^2 ,
	\end{align*}
	which proves \eqref{c-sample-norm}.
	 Using this inequality, \eqref{iter}, \eqref{c-lip}, and \eqref{c-stepform} we have,
	\begin{align} 
	\E[F(x_{k+1}) ] &\leq \E[F(x_k)]  -\E\left[\alpha \nabla F_{S_k}(x_k)^T\nabla F(x_k)\right] + \E\left[\frac{L\alpha^2}{2} \|\nabla F_{S_k}(x_k)\|^2\right] \nonumber\\
	&= \E[F(x_k) ]  -\alpha \|\nabla F(x_k)\|^2 + \frac{L\alpha^2}{2} \E[\|\nabla F_{S_k}(x_k)\|^2] \nonumber\\
&\leq \E[F(x_k)]  - \alpha \|\nabla F(x_k)\|^2 +\frac{ \alpha^2}{2}{L(1 + \theta^2+\nu^2)}\|\nabla F(x_k)\|^2\nonumber\\	
&\leq \E[F(x_k) ]  - \frac{\alpha}{2} \|\nabla F(x_k)\|^2  . \nonumber
	\end{align}
\end{proof}

\noindent
%%%%%%%%%%%%%%%%%%%%%%%%%%%%%%%%%%%%%%%%%%%%%%%
We now show that   iteration \eqref{iter}, using a fixed steplength $\alpha$, is linearly convergent when $F$ is strongly convex. In the discussion that follows, $x^*$ denotes the  minimizer of $F$.
\begin{thm} \label{thmlin} (Strongly Convex Objective.)
	Suppose that $F$ is twice continuously differentiable and that  there exist constants $0 < \mu \leq L$ such that 
	\begin{equation}
	\mu I   \preceq \nabla^2 F(x) \preceq L I, \quad \forall x \in \R^d.
	\end{equation} 
	Let $\{x_k\}$ be the iterates generated by  iteration \eqref{iter} with any $x_0$, where $|S_{k}|$ is chosen such that the (exact variance) inner product test \eqref{inner-i} and the (exact variance) orthogonality test \eqref{orth-i}  are satisfied at each iteration for any given  constants $\theta > 0$ and $\nu > 0$. Then, if the steplength satisfies \eqref{c-stepform}
	we have that   
	\begin{equation} \label{linear}
	\E[F(x_k) - F(x^*)] \leq \rho^k (F(x_0) - F(x^*)),
	\end{equation}
         where 
	\begin{equation} \label{mualpha}
	\rho= 1 -\mu \alpha  .
	\end{equation}
	In particular, if  $\alpha$  takes its maximum value in \eqref{c-stepform}, i.e., $\alpha =1/( 1 + \theta^2+\nu^2)L$, we have
	\begin{equation} \label{parsl}
	\rho= 1 -\frac{\mu}{L(1 + \theta^2+\nu^2)}.
	\end{equation}
\end{thm}
%%%%

%%%

\begin{proof} It is well known \cite{bertsekas2003convex} that for strongly convex functions   
\[
\|\nabla F(x_k)\|^2 \geq 2\mu [F(x_k) - F(x^*)].
\]
 Substituting this in \eqref{c-lin-ineq} and subtracting $F(x^*)$ from both sides we obtain,
\begin{align}
\E[F(x_{k+1}) - F(x^*)] &\leq  \E[F(x_k) - F(x^*)](1 -\mu \alpha ) , \nonumber 
	\end{align}
	from which the theorem follows.
\end{proof}

\noindent
Note that when $\theta = \nu = 0$ we recover the classical result for the exact gradient method.
We now consider the case when $F$ is convex, but not strongly convex.
\newpage
\begin{thm} \label{thmsublincon}
	(General Convex Objective.) Suppose that $F$ is twice continuously differentiable and convex, and that there exists a constant $ L > 0$ such that 
	\begin{equation}  \label{teen}
	\nabla^2 F(x) \preceq L I, \quad \forall x \in \R^d.
	\end{equation} 
	Let $\{x_k\}$ be the iterates generated by  iteration \eqref{iter} with any $x_0$, where $|S_{k}|$ is chosen such that the exact variance inner product test \eqref{inner-i} and  orthogonality test \eqref{orth-i}  are satisfied at each iteration for any given  constants $\theta > 0$ and $\nu > 0$. Then, if the steplength satisfies the strict version of \eqref{c-stepform}, that is 
	\begin{equation}   \label{stepform}
	\alpha_k = \alpha < \frac{1}{(1 + \theta^2+\nu^2)L},
	\end{equation}
	we have   for any positive integer $T$,
	\begin{align*}
	\min_{0 \leq k \leq T-1}\E \left[ F \left(x_{k} \right)\right] - F^* \leq \frac{1}{2\alpha c T}\|x_0 - x^*\|^2 ,\\
	\end{align*}
	where  $F^*$ is the optimal function value, the constant $c>0$ is given by $c = 1 - L \alpha (1 + \theta^2 + \nu^2) $,  and $x^* \in  \{x: x= \arg \min_x F(x)\}.$
\end{thm}

\begin{proof}
	From Lemma~\ref{c-thmlin} we have that
	\begin{align*}
	\E\left[\|\nabla F_{S_k}(x_k)\|^2\right] 
	%&\leq (1 + \theta^2)\|\nabla F(x_k)\|^2 + \nu^2 \|\nabla F(x_k)\|^2 \nonumber \\
	& \leq(1 + \theta^2 + \nu^2)\|\nabla F(x_k)\|^2 .
	\end{align*}
	Using this inequality and \eqref{teen} and considering any $x^* \in X^*$ we have,
	\begin{align}
	\E [\|x_{k+1} - x^*\|^2] &= \|x_k - x^*\|^2 - 2\alpha \E [\nabla F_{S_k}(x_k)^T (x_k - x^*)] + \alpha^2\E [\|\nabla F_{S_k}(x_k)\|^2] \nonumber\\
	&\leq \|x_k - x^*\|^2 - 2\alpha \nabla F(x_k)^T (x_k - x^*) + \alpha^2(1 + \theta^2 + \nu^2)\|\nabla F(x_k)\|^2 \nonumber\\
	&\leq \|x_k - x^*\|^2 - 2\alpha (F(x_k) - F^*) +  \alpha^2(1 + \theta^2 + \nu^2)\|\nabla F(x_k)\|^2 ,   \label{pre-yuri}
	\end{align}
where the last inequality follows from convexity of $F$. 

Now, for any function having Lipschitz continuous gradients 
\begin{equation}   \label{yuri}
  \|\nabla F(x)\|^2 \leq 2L (F(x) - F^*).
  \end{equation}
  This result is shown in \cite{nesterov2013introductory}, but for the sake of completeness we give here a  proof. Since $F$ has a Lipschitz continuous gradient, 
   	\begin{align*}
	F\left(x - \frac{1}{L} \nabla F(x)\right) &\leq F(x) + \nabla F(x)^T\left(x - \frac{1}{L} \nabla F(x) - x\right) + \frac{L}{2}\left\| x - \frac{1}{L} \nabla F(x) - x\right\|^2 \\
	&= F(x) - \frac{1}{L}\|\nabla F(x)\|^2 + \frac{1}{2L}\|\nabla F(x)\|^2 \\
	&=F(x) - \frac{1}{2L}\|\nabla F(x)\|^2 .
	\end{align*}
	Recalling that $F^*$ is the optimal function value  we have,
	\begin{align*}
	F^* \leq F\left(x - \frac{1}{L} \nabla F(x)\right) \leq F(x) - \frac{1}{2L}\|\nabla F(x)\|^2,
	\end{align*}
	which proves \eqref{yuri}. 
	
	Substituting \eqref{yuri} in \eqref{pre-yuri} we obtain 
\begin{align*}
	\E [\|x_{k+1} - x^*\|^2] 
	&\leq \|x_k - x^*\|^2 - 2\alpha (F(x_k) - F^*) +  2L\alpha^2(1 + \theta^2 + \nu^2)(F(x_k) - F^*) \nonumber\\
	&=\|x_k - x^*\|^2 - 2\alpha (F(x_k) - F^*)\left(1 - L\alpha(1 + \theta^2 + \nu^2)\right) 	 \nonumber\\
	&= \|x_k - x^*\|^2 - 2\alpha c (F(x_k) - F^*),
	\end{align*}	
 by the definition of $c$. We can write this inequality as
	\begin{align*}
	\E [F(x_k)] - F^* \leq \frac{1}{2\alpha c} \left(\E [\|x_k - x^*\|^2] - \E [\|x_{k+1} - x^*\|^2] \right),
	\end{align*}
	and summing, we obtain
	 \begin{align*}
	 \min_{0 \leq k \leq T-1}\E \left[ F \left(x_{k} \right)\right] - F^* &\leq \sum_{k=0}^{T-1}\frac{1}{T}(\E [F(x_{k})] - F^*) \\
	 &\leq \frac{1}{2 \alpha c T}(\E [\| x_0 - x^*\|^2] - \E [\|x_{T} - x^*\|^2]) \\
	 &\leq \frac{1}{2 \alpha cT}\| x_0 - x^*\|^2.
	 \end{align*} 
\end{proof}
This result establishes a sublinear rate of convergence in function values by referencing the best function value obtained after every $T$ iterates.	We now consider the case when $F$ is nonconvex and bounded below.

    					\begin{thm} \label{thmsublin}
    						(Nonconvex Objective.) Suppose that $F$ is twice continuously differentiable and bounded below,  and that  there exist a constant $ L > 0$ such that 
    						\begin{equation}
    					     \nabla^2 F(x) \preceq L I, \quad \forall x \in \R^d.
    						\end{equation} 
    						Let $\{x_k\}$ be the iterates generated by  iteration \eqref{iter} with any $x_0$, where $|S_{k}|$ is chosen such that the (exact variance) inner product test \eqref{inner-i} and the (exact variance) orthogonality test \eqref{orth-i}  are satisfied at each iteration for any given  constants $\theta > 0$ and $\nu > 0$. Then, if the steplength satisfies 
    						\begin{equation}   \label{stepformnon}
    						\alpha_k = \alpha \leq \frac{1}{(1 + \theta^2+\nu^2)L},
    						\end{equation}
    						then
\begin{equation}   \label{convergence}
      \lim_{k \rightarrow \infty} \mathbb{E} [\| \nabla F(x_k)\|^2] \rightarrow 0.
\end{equation}  
Moreover, for any positive integer $T$ we have that 
    						\begin{align*}
    						\min_{0\leq k \leq T-1} \E [\|\nabla F(x_k)\|^2] &\leq \frac{2}{\alpha T}  (F(x_0) - F_{min}),
    						\end{align*}
    						where  $F_{min}$ is a lower bound on  $F$ in $\mathbb{R}^d$.
    					\end{thm}
    					\begin{proof}
					From Lemma~\ref{c-thmlin} we have
    					\begin{align*}
    						\E[F(x_{k+1})] &\leq \E[F(x_k)]  - \frac{\alpha}{2}\E [\|\nabla F(x_k)\|^2],
    					\end{align*}	
    				    and hence
    				    \[
    				    \E [\|\nabla F(x_k)\|^2] \leq \frac{2}{\alpha} \E [F(x_k) - F(x_{k+1})] .
    				    \]
    					Summing both sides of this inequality from $k= 0$ to $T-1$, and since
					$F$ is bounded below by  $F_{min}$, we get
    					\[
    					\sum_{k=0}^{T-1}\E [\|\nabla F(x_k)\|^2] \leq \frac{2}{\alpha} \E [F(x_0) - F(x_{\mbox{\sc t}})] 
					\leq \frac{2}{\alpha}  [F(x_0) - F_{min}].
    					\]
					Taking limits, we obtain
\[
    \lim_{T \rightarrow \infty} \sum_{k=0}^{T-1}\E [\|\nabla F(x_k)\|^2] < \infty,
 \]
 which implies \eqref{convergence}.
%    					We also have that
%    					\begin{align*}
%    					\frac{1}{T}\sum_{k=0}^{T-1}\E [\|\nabla F(x_k)\|^2] &\leq 
%					%\frac{2}{\alpha T} \E [F(x_0) - F(x_{T})] \\
%    					 \frac{2}{\alpha T} (F(x_0) - F_{min}).
%    					\end{align*}
    					We can also conclude that
    					\begin{align*}
    					\min_{0\leq k \leq T-1} \E [\|\nabla F(x_k)\|^2] \leq & \frac{1}{T}\sum_{k=0}^{T}\E [\|\nabla F(x_k)\|^2] 
					\leq \frac{2}{\alpha T}  (F(x_0) - F_{min}).
    					\end{align*}
    				\end{proof}	
    				
This theorem shows that the sequence of gradients $\{ \|\nabla F(x_k) \|\}$ converges to zero, in expectation. It also establishes a global sublinear rate of convergence of the smallest gradients generated after every $T$ steps. Results with a similar flavor have been established by Ghadimi and Lan \cite{ghadimi2013stochastic} in the context of nonconvex stochastic programming.

%%%%%%%%%%%%%%%%%%%%%%%%%%%      
\section{Practical Implementation}  \label{practical}
%%%%%%%%%%%%%%%%%%%%%%%%%%%%

In this section, we  describe  our line search procedure, and  present a technique for making the algorithm robust in the  early stages of a run when sample variances and sample gradients are unreliable.  We also  discuss a heuristic that determines  how much to increase the sample size $|S_k|$. We then present the complete algorithm, followed by a discussion of the choice of  some important algorithmic parameters.

%%%%%%%%%%%%%%%%%%%%
\subsection{Line Search} 
We select the steplength parameter $\alpha_k$ by a backtracking line search based on the sampled function $F_{S_k}$  and an adaptive estimate $L_k$ of the Lipschitz constant of the gradient. The initial value of the steplength at the $k$-th iteration of the algorithm is given by $\alpha_k= 1/L_k$.
 If sufficient decrease in $F_{S_k}$ is not obtained, $L_k$ is increased by a constant factor until such a decrease is achieved. 
 
 Now, since overestimating the Lipschitz constant leads to unnecessarily small steps, at every outer iteration of the algorithm the initial value $L_k$ is set to a fraction of the previous estimate $L_{k-1}$.
This reset strategy has been used in deterministic convex optimization, 
% \cite{beck2009fast}
but requires some attention in the stochastic setting.  Specifically, decreasing the Lipschitz constant by a fixed fraction at every iteration may result in an inadequate steplength. We propose a variance-based rule described below to compute a contraction factor $\zeta_k$ at every iteration. Our line search strategy is summarized in Algorithm~\ref{alg:lip}. \\
\begin{algorithm}[H] 
	\caption{Backtracking Line Search}
	\label{alg:lip}
	{\bf Input:} $L_{k-1} > 0$,  $\eta > 1$
	
	\begin{algorithmic}[1]
		%\State Shuffle dataset $S$
		\State Compute $\zeta_k$ as given in \eqref{lipschtiz-decrease}
		%\State Set $L_{k-1}=L_{k-1}/\zeta $; \quad  $k \rightarrow 0$ \Comment{Decrease the Lipschitz constant}
		\State Set $L_{k} = L_{k-1}/\zeta_k$  \Comment{Decrease the Lipschitz constant}
		\State Compute $F_{new}= F_{S_k}\left(x_k - \frac{1}{L_k}\nabla F_{S_k}(x_k) \right)$
		\While {$ F_{new} > F_{S_k}(x_k) - \frac{1}{2L_k}\|\nabla F_{S_k}(x_k)\|^2$} \Comment {Sufficient decrease}
		%\State Set counter t =t + 1 
		\State Set $L_k = \eta L_{k-1}$ \Comment{Increase the Lipschitz constant}
		\State Compute $F_{new}= F_{s_k}\left(x_k - \frac{1}{L_k}\nabla F_{S_k}(x_k) \right)$
		\EndWhile
	\end{algorithmic}
\end{algorithm}

The expansion factor $\eta$ in Step 5 is set to $\eta=1.5$ in our experiments. To determine the contraction factor $\zeta_k$, we reason as follows.
From \eqref{iter} we have 
\[
    \E[F(x_{k+1})] - F(x_k) \leq -\alpha_k \|\nabla F(x_k)\|^2 + \frac{ \alpha_k^2 L}{2} \E[\| \nabla F_{S_k}(x_k)\|_2^2].
\]
Thus we can guarantee an expected decrease in the true objective function if the right hand side is negative, i.e.,
\begin{equation} \label{exp-dec}
\frac{L\alpha_k^2}{2}\left(\var \left(\nabla F_{S_k}(x_k)\right) + \|\nabla F(x_k)\|^2\right) \leq \alpha_k\|\nabla F(x_k)\|^2 ,
\end{equation}  
where 
\begin{equation*}
\var \left(\nabla F_{S_k}(x_k)\right) = \E[\|\nabla F_{S_k}(x_k) - \nabla F(x_k)\|^2] .
\end{equation*}
This is the same variance as  in the norm test \eqref{normt}. As was done in that context we first note that \eqref{exp-dec} holds if 
\begin{equation*} 
\frac{L\alpha_k^2}{2}\left(\var \left(\nabla F_{i}(x_k)\right)/|S_k| + \|\nabla F(x_k)\|^2\right) \leq \alpha_k\|\nabla F(x_k)\|^2.
\end{equation*} 
Next, we approximate the true gradient and true variance using a sampled gradient and a sampled variance to obtain
\begin{equation}    \label{simons}
\frac{L\alpha_k^2}{2}\left(\var_{i \in S_k} \left(\nabla F_{i}(x_k)\right)/|S_k| + \|\nabla F_{S_k}(x_k)\|^2\right) \leq \alpha_k\|\nabla F_{S_k}(x_k)\|^2,
\end{equation} 
where 
\begin{equation*}
\var_{i \in S_k} \left(\nabla F_i(x_k)\right) = \frac{1}{|S_k| - 1} \sum_{i \in S_k}\|\nabla F_i(x_k) - \nabla F(x_k)\|^2.
\end{equation*}
We wish to compute an appropriate value $L_k$ and set $\alpha_k= 1/L_k$. Therefore, since we can assume that $L$ is large enough so that $L_{k-1} < L$, inequality \eqref{simons} gives
%Assuming that the current Lipschitz constant is accurate and we employ $\alpha_k=1/L_k$ in our algorithm then the above condition states that  
\begin{align*}
\frac{1}{L_k}\|\nabla F_{S_k}(x_k)\|^2 & \geq \frac{L_{k-1}}{2L_k^2}(\var_{i \in S_k}\left(\nabla F_i(x_k)\right)/|S_k|+ \|\nabla F_{S_k}(x_k)\|^2) 
\end{align*}
or
\begin{align*}
L_k & \geq \frac{L_{k-1}}{2}\left(\frac{\var_{i \in S_k} \left(\nabla F_i(x_k)\right)}{|S_k|\|\nabla F_{S_k}(x_k)\|^2} + 1\right) 
    \equiv \frac{L_{k-1}}{2} a_k.
\end{align*}
This indicates that when decreasing the Lipschitz estimate by the rule $L_k = L_{k-1}/\zeta_k$ at the start of each iteration, 
  $\zeta_k$ may be chosen as
\begin{equation}\label{lipschtiz-decrease}
%\zeta_k = \max\left(1, \frac{2}{\frac{\var_{i \in S_k}\left(\nabla F_{i}(x_k)\right)}{|S_k|\|\nabla F_{S_k}(x_k)\|^2} + 1}\right).
\zeta_k = \max\left(1, \frac{2}{a_k}\right)
\end{equation} 
Therefore, $\zeta_k \in [1,2]$, meaning $L_k$ could be left unchanged $(\zeta_k =1)$, or reduced by a factor of at most 2.

Algorithm~\ref{alg:lip} is similar in form to the one proposed in \cite{beck2009fast} for  deterministic functions. However, our line search operates with sampled (e.g. inaccurate) function values,  and the adaptive setting has allowed us to employ a variance-based rule described above for estimating the initial value of $L_k$ at every iteration.
%and the adaptive setting suggest a decrease in the Lipschitz constant estimate when sample sizes have been increased. 
A  line search based on sampled function values that is more close related to ours is  described in  \cite{schmidt2013minimizing}, but it differs from ours in that they use a fixed contraction factor $\zeta$ throughout the algorithm.

%%%%%%%%%%%%%%%%%
\subsection{Sample Control in the Noisy Regime}
%%%%%%%%%%%%%%%%%

In the previous sections, we presented two forms of the inner product and orthogonality tests: one based on the population statistics and one based on samples. Since in many applications only sample statistics are available,  our practical implementation of the algorithm imposes the inner product and orthogonality tests by verifying \eqref{final} instead of \eqref{inner-i}, and \eqref{orthogonal-p} instead of \eqref{orth-i}, respectively.  The \emph{augmented inner product test} consists of the inner product test  \eqref{final} together with the orthogonality test \eqref{orthogonal-p}.

Our numerical experience indicates that these sample approximations are sufficiently accurate, except if we choose to start the algorithm with a very small sample size, say 3, 5, 10. In this highly noisy regime, our conditions may not control the sample correctly  because for small samples, $ \|\nabla F_{S_k}(x_k)\|$ is often much larger than the true gradient $ \|\nabla F(x_k)\|$, and therefore the tests  \eqref{final} and \eqref{orthogonal-p} are  too easily satisfied, preventing increases in the sample size. 
(These difficulties can also arise when using the norm test.)

To obtain a more accurate  estimate of $\nabla F_{S_k}(x_k)$ to use in \eqref{final} and \eqref{orthogonal-p}, we employ the following strategy.
Whenever the sample sizes remain constant for a certain number of iterations, say $r$, we compute the running average of the most recent sample gradients:
\begin{equation} \label{running-avg}
	g_{\rm avg} \defeq \  \frac{1}{r} \Big[\sum_{j=k-r+1}^{k} \nabla F_{S_j}(x_j)\Big].
\end{equation} 
Ideally, $r$ should be chosen such that the iterates in this summation are close enough to provide a good approximation of the gradient at $x_k$ and so that there are enough samples for $g_{\rm avg}$ to be meaningful. (A reasonable default value could be $r=10$.)
If the length of $g_{\rm avg}$ is  small compared with the length of  $\nabla F_{S_k}(x_k)$, we view this as an indication that the latter is not accurate and the sample size should be increased. That is, if 
\begin{equation} \label{gammatest}
	\|g_{\rm avg}\| < \gamma \|\nabla F_{S_k}(x_k)\|, \quad \gamma \in (0,1),
\end{equation}
then the sample size is increased using a rule given in Section~\ref{increase}. The choice of the parameter $\gamma$ is important, we choose it so that under some assumptions inequality \eqref{gammatest} implies
\begin{equation}   \label{storm}
 \|\nabla F_{S_k}(x_k)\| \geq \omega \|\nabla F(x_k)\|,
 \end{equation}
for some constant $\omega$ (say 10).
To relate $\gamma$ and $\omega$, note that since we used $r$ times more samples in computing $g_{\rm avg}$  compared to $\nabla F_{S_k}(x_k)$,    $\| g_{\rm avg}\|$ should be a better estimate of $\| \nabla F(x_k)\|$,  by a factor of $\sqrt{r}$. Although this may be optimistic since the $r$ points $x_j$ used in the running average do not coincide, we nevertheless assume that

\begin{align*}
\|g_{\rm avg}\| -  \|\nabla F(x_k)\|&\approx \frac{1}{\sqrt{r}} \left( \|\nabla F_{S_k}(x_k)\|-\|\nabla F(x_k)\| \right)    \\
\end{align*}
so that
\[
\frac{\|g_{\rm avg}\|}{ \|\nabla F_{S_k}(x_k)\|}  \approx  \frac{1}{\sqrt{r}} +  \left( 1 -  \frac{1}{\sqrt{r}}\right) \frac{\|\nabla F(x_k)\|}{ \|\nabla F_{S_k}(x_k)\|}   . 
\]   

Recalling \eqref{gammatest} and \eqref{storm}, we then set 
\[
     \gamma = \frac{1}{\sqrt{r}} + \left( 1 -  \frac{1}{\sqrt{r}}\right)  \frac{1}{\omega} . 
\]
For example, if $\omega = 10$ and $r=10 $ then satisfaction of (\ref{gammatest}) with  $\gamma=0.38$ should correspond to satisfaction of (\ref{storm}). We choose this value in practice.

%%%%%%%%%%%%%%%%%%%%%%%%%
\subsection{Increasing the Sample Size}   \label{increase} 
When increasing the sample size, it is natural to require that the new sample satisfy the inner product and the orthogonal tests at the current iterate.  We employ the following heuristic approach  that aims to achieve this goal. 

Suppose we wish to choose a larger sample size $|\hat{S}_k|$.  We assume that increases in sample sizes are gradual enough so that, for any given iterate $x_k$,

\begin{equation*}  
{\rm Var}_{i \in \hat{S}_k}(\nabla F_i(x_k)^T\nabla F_{\hat{S}_k}(x_k)) \approxeq {\rm Var}_{i \in S_k}(\nabla F_i(x_k)^T\nabla F_{S_k}(x_k)).
\end{equation*}

The same logic can be applied to the orthogonality test, i.e.,
\begin{align*}
&{\rm Var}_{i \in \hat{S}_k}\left(\nabla F_i(x_k) - \frac{\nabla F_i(x_k)^T\nabla F_{\hat{S}_k}(x_k)}{\|\nabla F_{\hat{S}_k}(x_k)\|^2}\nabla F_{\hat{S}_k}(x_k)\right) \nonumber\\
&\approxeq {\rm Var}_{i \in S_k}\left(\nabla F_i(x_k) - \frac{\nabla F_i(x_k)^T\nabla F_{S_k}(x_k)}{\|\nabla F_{S_k}(x_k)\|^2}\nabla F_{S_k}(x_k)\right) ,
\end{align*}
and let us also assume that
\begin{equation*} 
\|\nabla F_{\hat{S}_k}(x_k)\| \approxeq \|\nabla F_{S_k}(x_k)\|.
\end{equation*}
Under these simplifying assumptions, we see that the inner product test \eqref{final} and the orthogonality test \eqref{orthogonal-p} are satisfied if we choose $|\hat{S}_k|$ to be
\begin{equation} \label{heuristic-samplesize}
|\hat{S}_k| = \max \left( \frac{{\rm {\rm Var}}_{i \in S_k}(\nabla F_i(x_k)^T\nabla F_{S_k}(x_k))}{\theta^2\|\nabla F_{S_k}(x_k)\|^4}, \frac{{\rm Var}_{i \in S_k}\left(\nabla F_i(x_k) - \frac{\nabla F_i(x_k)^T\nabla F_{S_k}(x_k)}{\|\nabla F_{S_k}(x_k)\|^2}\nabla F_{S_k}(x_k)\right)}{\nu^2\|\nabla F_{S_k}(x_k)\|^2}\right).
\end{equation}
In the case when the sample size is very small and we use the running average \eqref{running-avg}, we increase the sample size according to
\begin{equation} \label{backup-samplesize}
|\hat{S}_k| = \max \left( \frac{{\rm Var}_{i \in S_k}(\nabla F_i(x_k)^Tg_{\rm avg})}{\theta^2\|g_{\rm avg}\|^4}, \frac{{\rm Var}_{i \in S_k}\left(\nabla F_i(x_k) - \frac{\nabla F_i(x_k)^Tg_{\rm avg}}{\|g_{\rm avg}\|^2}g_{\rm avg}\right)}{\nu^2\|g_{\rm avg}\|^2}\right).
\end{equation}
Although these heuristics can be unreliable when the assumptions made in their derivation are not satisfied, they have worked well in  our experiments.

%%%%%%%%%%%%%%%%%%%%%%
\subsection{The Complete Algorithm}
%%%%%%%%%%%%%%%%%%%%%%
The final version of the adaptive sampling algorithm that incorporates all the practical implementation techniques described in this section is given in  Algorithm~\ref{alg:complete}.

\newpage
\begin{algorithm} 
	\caption{Adaptive Sampling Method}
	\label{alg:complete}
	{\bf Input:} Initial iterate $x_0$, initial sample  $S_0$. \\
	Sample Test Parameters: $\theta > 0, \nu >0, r>0, \gamma \in (0,1)$ \\
	 Steplength Parameters: $L_0 > 0$, $ \eta > 1$.\\ 
	Set $k \leftarrow 0$\\
	{\bf Repeat} until a convergence test is satisfied:
	\begin{algorithmic}[1]
		\State Compute $d_k = -\nabla F_{S_k}(x_k)$
		\State Compute $L_k$ using Algorithm \eqref{alg:lip}
		\State Set $\alpha_k = 1/L_k$
		\State Compute new iterate: $x_{k+1} = x_k + \alpha_k d_k$
		\State Set $k \leftarrow k + 1$
		\State Set $|S_k| = |S_{k-1}|$ and choose a new sample $S_k$
		%\State \jn{Choose a new sample $S_k$.}
		\If {condition \eqref{final} or condition \eqref{orthogonal-p} is not satisfied}
		\State Compute $|S_k|$ using \eqref{heuristic-samplesize} and choose a new sample $S_k$
		\EndIf
		\If {$\left(|S_k|=|S_{k-1}|=\cdots = |S_{k-r}|\right)$ } 
			\State Compute the running average gradient $g_{\rm avg}$ given in \eqref{running-avg}
			\If {$\|g_{\rm avg}\| < \gamma \|\nabla F_{S_k}(x_k)\|$}
				\If {condition \eqref{final} or condition \eqref{orthogonal-p} is not satisfied using $g_{\rm avg}$ instead of $\nabla F_{S_k}(x_k)$}
					\State Compute $|S_k|$ using \eqref{backup-samplesize} and choose a new sample $S_k$
				\EndIf			
			\EndIf
		\EndIf
		%\State Verify the conditions \eqref{scvr-final-condition} and \eqref{final-scvr-orth} with current $|S_k|$
		%\State Choose a sample $S_k$ such that the condition \eqref{scvr-final-condition} is satisfied.
		%\State Compute the sample variance defined in \eqref{final}.
		%\State If condition \eqref{scvr-final-condition} is not satisfied, Choose $|S_k|$ using formula \eqref{heuristic-samplesize}. 
	\end{algorithmic}
\end{algorithm}

\bigskip
Let us now discuss the choice of the parameters $\theta$ and $\nu$ that govern the inner product and orthogonality tests. Both have statistical significance, and have considerable impact on the algorithm. For example,  $\theta$ affects the value of the steplength $\alpha$ in the analysis presented in the previous section (see \eqref{c-stepform}), and consequently the linear convergence constant (see \eqref{mualpha}, \eqref{parsl}).

%\subsection{Parameter Selection}   \label{shall}
Concerning the parameter $\theta$  in the inner product test \eqref{final}, we know that the sample mean is a random quantity that  approximately follows a normal distribution, if a significant number of samples are used. In our case, this means,
\begin{equation}
\frac{\nabla F_{S_k}(x_k)^T\nabla F(x_k) - \|\nabla F(x_k)\|^2}{\left(\frac{\sigma}{\sqrt{|S_k|}}\right)} \sim \mathcal{N}(0,1),
\end{equation}
where $\sigma^2 =\E \left[\left( \nabla F_{i}(x_k)^T\nabla F(x_k) - \|\nabla F(x_k)\|^2 \right)^2\right] $ is the population variance employed in \eqref{inner-i}.

The one-sided $100(1 - \varrho)\%$ confidence interval for $\nabla F_{S_k}(x_k)^T\nabla F(x_k)$ is given by
\[
\left(\|\nabla F(x_k)\|^2 -z_{\varrho}\frac{\sigma}{\sqrt{|S_k|}}, \infty\right).
\]
Since our goal is to have $\nabla F_{S_k}(x_k)^T\nabla F(x_k)>0$,   we propose that the lower value be nonnegative, which implies
\[
\frac{\sigma^2}{|S_k|} \leq \frac{1}{z_{\varrho}^2}\|\nabla F(x_k)\|^4.
\] 
Therefore, we see from \eqref{final} that $\theta$ is  related to  $z_{\varrho}$. Smaller $\theta$ values increase the probability of obtaining a descent direction, and at the same time  promote larger sample sizes.  In particular, the value $\theta =0.9$  corresponds to a probability of  about $ 0.8$; we have observed that this value works well in our numerical tests. 
%We note that the choice of $\theta$ is significant in the algorithm as it affects the value of the steplength $\alpha$ in the analysis presented in the previous section (see \eqref{stepform}), and the rate of convergence (see \eqref{mualpha}, \eqref{parsl}).
 
Let us now consider the choice of $\nu$, which governs the orthogonality test \eqref{orthogonal-p}.  As discussed in the paragraph that follows \eqref{orthogonal-p},  $\nu$ affects the tangent of the angle $\chi_k$
  between $\nabla F_{S_k}(x_k)$ and $\nabla F(x_k)$. In practice, we choose $\nu$ to be sufficiently large such that the orthogonality test does not influence the sample size selection for most of the run, while at the same time ensuring that the angle $\chi_k$ is bounded away from $90^\circ$. This reasoning leads us to the choice $\nu =\tan 80^\circ = 5.84$. Although in theory one can choose much larger values of $\nu$,  this value works well in our tests --- but like $\theta$, $\nu$ should be regarded as a tunable parameter in the algorithm.

\section{Numerical Experiments}   \label{numerical}
In this section, we report the results of numerical experiments that illustrate the performance of the method proposed in this paper. We consider binary classification problems where the objective function is given by the logistic loss with $\ell_2$ regularization:
\begin{equation} \label{logistic-loss}
R(x)=\frac{1}{N} \sum_{i=1}^{N}\log(1 + \exp(-z^ix^Ty^i)) + \frac{\lambda}{2}\|x\|^2, \quad\mbox{with}  \quad \lambda = \frac{1}{N} .
\end{equation} 
 We used the datasets listed in Table~\ref{tab11}.
\begin{table}[htp]
	\centering
	\label{data sets}
	\begin{tabular}{|c||c|c|c|}
		\hline
		Data Set  & Data Points $N$ & Variables $d$ & Reference \\ \hline
		gisette   & 6000		   & 5000		 &			\cite{guyon2004result} \\
		synthetic & 7000          & 50          &           \cite{mukherjee2013parallel}\\
		mushroom & 8124          & 112         &          \cite{Lichman2013} \\
		sido     & 12678         & 4932         &          \cite{Lichman2013} \\
		RCV1     & 20242         & 47236         &          \cite{Lichman2013} \\
		ijcnn     & 35000         & 22         &          \cite{Lichman2013} \\    
		MNIST    & 60000         & 784         &      \cite{lecun2010mnist}     \\ 
		real-sim     & 65078         & 20958         &          \cite{Lichman2013} \\ 
		covertype & 581012       & 54          &           \cite{blackard1999comparative}\\ \hline
	\end{tabular}
	\caption{Characteristics of the binary datasets used in the experiments.}
	\label{tab11}
\end{table} 

Our goal is to compare the performance of the \emph{augmented inner product test} implemented in Algorithm~\ref{alg:complete}, which imposes both \eqref{final} and \eqref{orthogonal-p}, and the norm test \eqref{dss-sample-test}, which as mentioned above, dates back to Carter \cite{carter} and has  recently received much attention.
 The practical form of the norm test is  
\begin{equation} \label{final-dss}
\frac{1}{|S_k| - 1} \frac{\sum_{i \in S_k}\|\nabla F_i(x_k) - \nabla F_{S_k}(x_k)\|^2}{|S_k|} \leq \theta^2 \|\nabla F_{S_k}(x_k)\|^2 .
\end{equation}

An approximation $R^*$ of the optimal function value was computed for each problem  by running the L-BFGS method until $\|\nabla R(x_k)\|_{\infty} \leq 10^{-8}$. For all the experiments, we provide plots with respect to effective gradient evaluations and iterations. Effective gradient evaluations represent the total number of equivalent full gradients $\nabla R$ and full functions $R$ computed in a run. All runs were terminated when $\|\nabla R(x_k)\|_{\infty} \leq 10^{-6}$ or if 100 epochs (passes through the whole data set) are performed. We observed that for the finite sum problem \eqref{logistic-loss}, best performance is attained when the samples are chosen at random \emph{without} replacement. 
 
\subsection{Tests with Fixed Steplengths}
We first test the algorithm using a constant steplength rather than the  line search (Algorithm~\ref{alg:lip}) in order to evaluate the adaptive sampling strategies in a simple setting. Using a fixed steplength $\alpha_k=\alpha$, and tuning it for each problem so as to achieve best performance is a popular strategy in machine learning. In our experiments we select  $\alpha$ from the set $\{2^{-10}, 2^{-7}, \cdots, 2^{15}\}$. All other parameters of the algorithms are set as follows:  $\theta = 0.9$, $\nu = \tan(80^o)=5.84$, $r=10$, $\gamma =0.38$. The initial sample size was set to $|S_0| =2$.

Figure~\ref{synthetic-ex1funcvals} reports the performance of the two methods for the {\tt synthetic} dataset. The vertical axis measures the error in the function, $R(x) - R^*$, and the horizontal axis  the number of effective gradient evaluations or iterations.  (Results for the other datasets are given in Appendix~\ref{addnumerical}.) For this problem, the best steplength for both methods was $\alpha= 2^0$.  The inner product test is clearly more efficient in terms of effective gradient evaluations, which is indicative of the total computational work and CPU time, but  the norm  test requires fewer iterations for reasons discussed in the next experiment. 

\begin{figure}[!htp]
	\begin{centering}
		\includegraphics[width=0.45\linewidth]{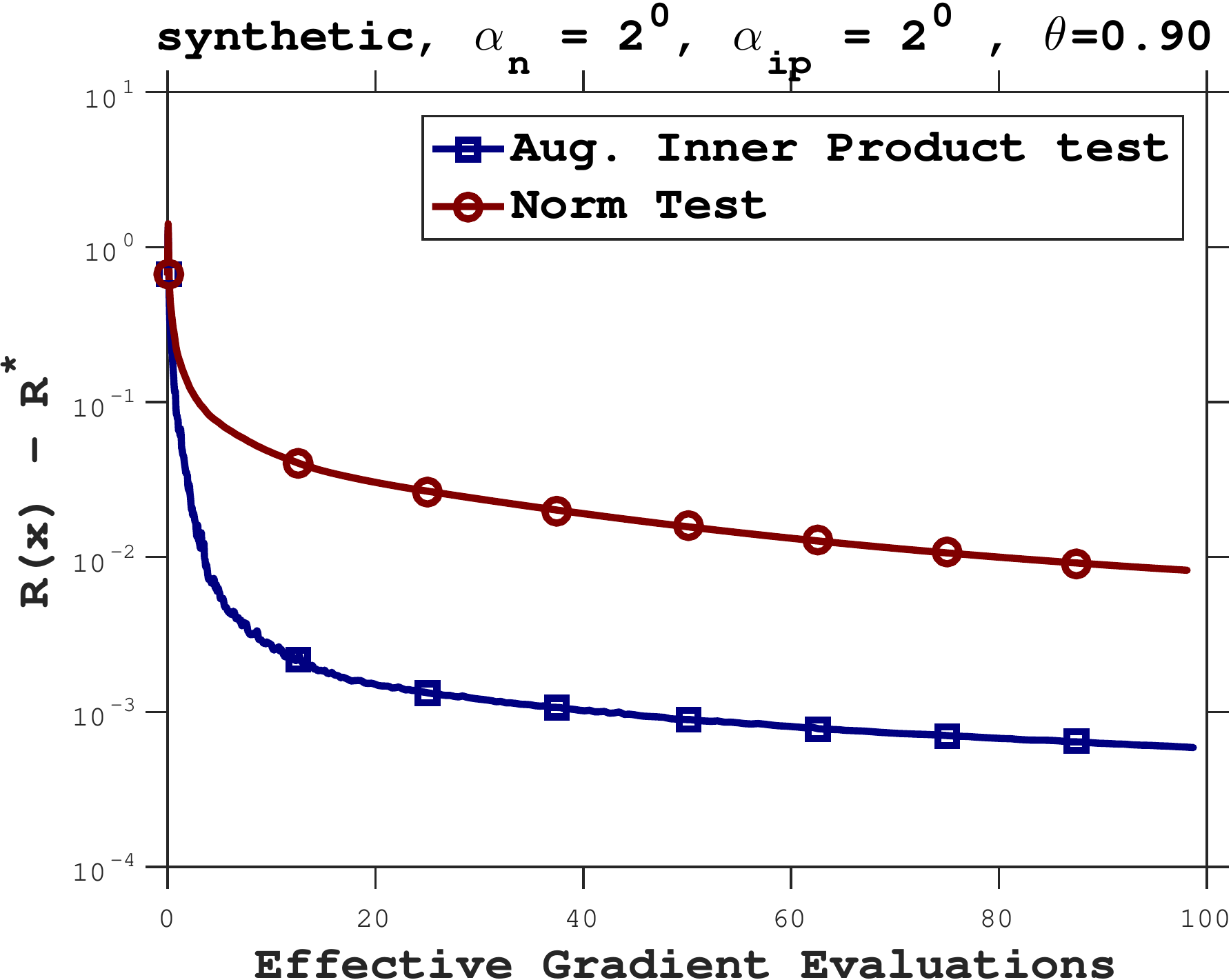}
		\includegraphics[width=0.45\linewidth]{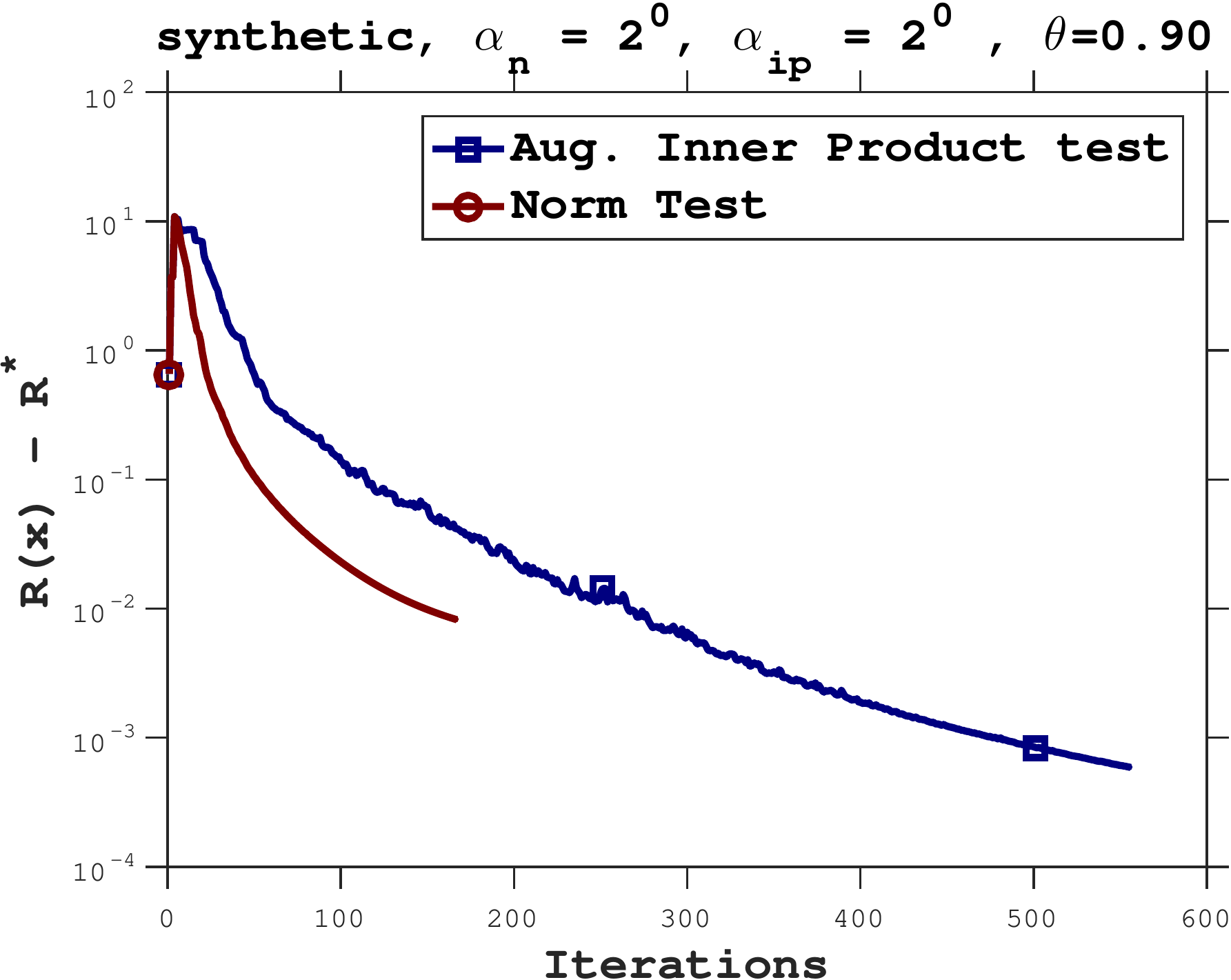}
		\par\end{centering}
	\caption{ {\tt \large synthetic} dataset: Performance of the augmented inner product test and norm test.
	% with $\theta=0.9$ and for constant $\alpha_{ip} = 2^0$, $\alpha_n = 2^0$ with initial samplesize $S_0 =2$.
	 Left: Function error vs. effective gradient evaluations; Right: Function error vs. iterations.} 
	\label{synthetic-ex1funcvals} 
\end{figure}

In Figure~\ref{synthetic-ex1batchsizes}, we plot sample size $|S_k|$ (as a percentage of the total dataset) at each iteration. We observe that the augmented inner product test increases the sample size slower than the norm test. The latter approximates the full gradient method earlier on, explaining the smaller number of iterations observed in Figure~\ref{synthetic-ex1funcvals}.

\begin{figure}[htp!]
	\begin{centering}
		\includegraphics[width=0.5\linewidth]{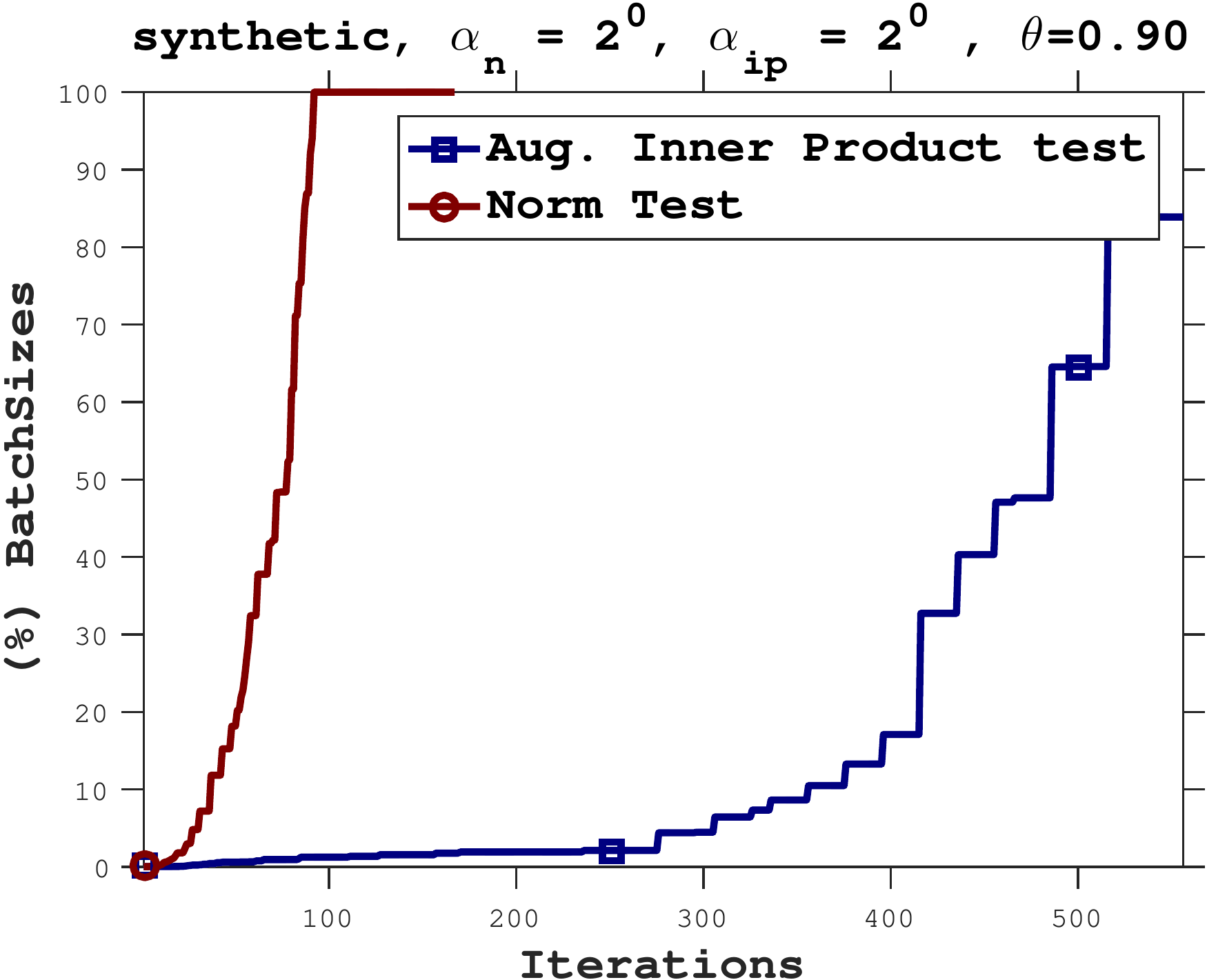}
		\par\end{centering}
	\caption{ {\tt \large synthetic} dataset: Growth of the sample (batch) size $|S_k|$ for the two methods, as a function of iterations. } 
	\label{synthetic-ex1batchsizes} 
\end{figure}

We also plot, in Figure \ref{synthetic-ex1angles},  the angle (in degrees) between the sampled gradient $\nabla F_{S_k} (x_k)$ and true gradient $\nabla F(x_k)$, at each iteration. 
%Here the vertical axis represent the angle in degrees. 
Note that the norm test rapidly reduces the angle between these two vectors, whereas the augmented inner product test allows for more freedom, with the angle varying significantly, but averaging about $70^\circ$. 
Results on the remaining data sets are given in the Appendix~\ref{addnumerical}. 

%We observe that in case of augmented inner product test the angles are concentrated around 80-90 degrees where as in case of norm test the angles are much smaller. This is due to the phenomena that norm test is too restrictive and so results in higher samples and better angles where as augmented inner product test is less restrictive and simply ensures only descent. 

%We also note that the discrepancies at some iterations where the sampled gradients are not descent directions is due to the fact that we are using sample variance and sample gradient in the descent test. 

\begin{figure}[htp]
	\begin{centering}
		\includegraphics[width=0.5\linewidth]{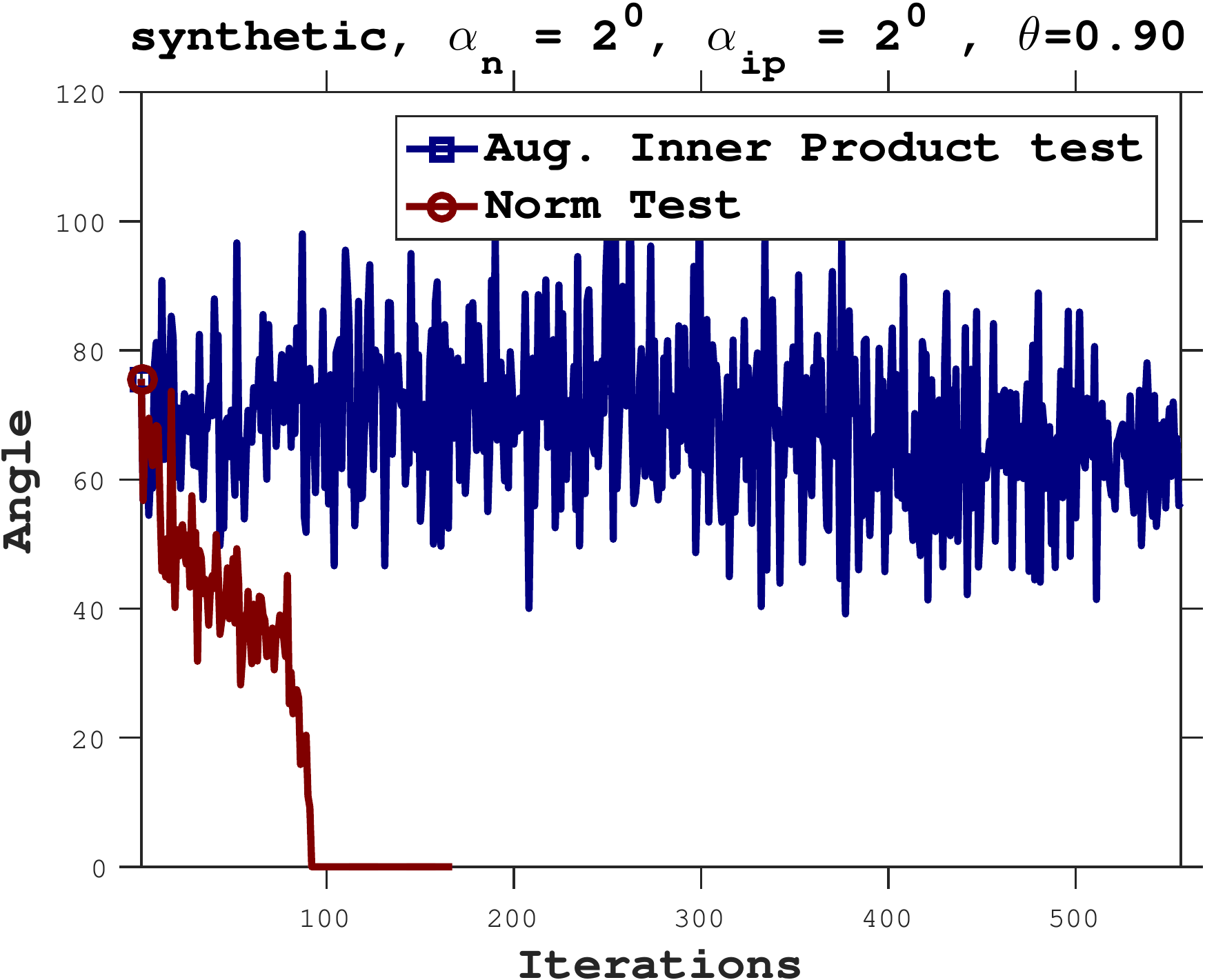}
		\par\end{centering}
	\caption{ {\tt \large synthetic} dataset: Angle between sampled gradient and true gradient vs. iterations for the two methods.} 
	\label{synthetic-ex1angles} 
\end{figure}

%%%%%%%%
\subsection{Performance of the Complete Algorithm}

We  now compare the performance of Algorithm~\ref{alg:complete} against the variant that employs the norm test \eqref{final-dss}. 
Both utilize the line search procedure described in Algorithm~\ref{alg:lip}.  As before, we set  $\theta = 0.9, \nu =5.84, r=10, \gamma =0.38$ . In Algorithm~\ref{alg:lip}, the initial estimate of the Lipschitz constant is $ L_0=1$, and we use $\eta=1.5$; we observed that these two values give  good performance for all our test problems.

The plots in Figure~\ref{synthetic-ex2funcvals} report function error vs. effective gradient evaluations, and  batch sizes vs. iterations, respectively. Effective gradient evaluations include the function evaluations during the line search procedure, i.e., Algorithm~\ref{alg:lip}. We observe that the augmented inner product test outperforms the norm test and increases the sample sizes  less aggressively.

\begin{figure}[!htp]
	\begin{centering}
		\includegraphics[width=0.45\linewidth]{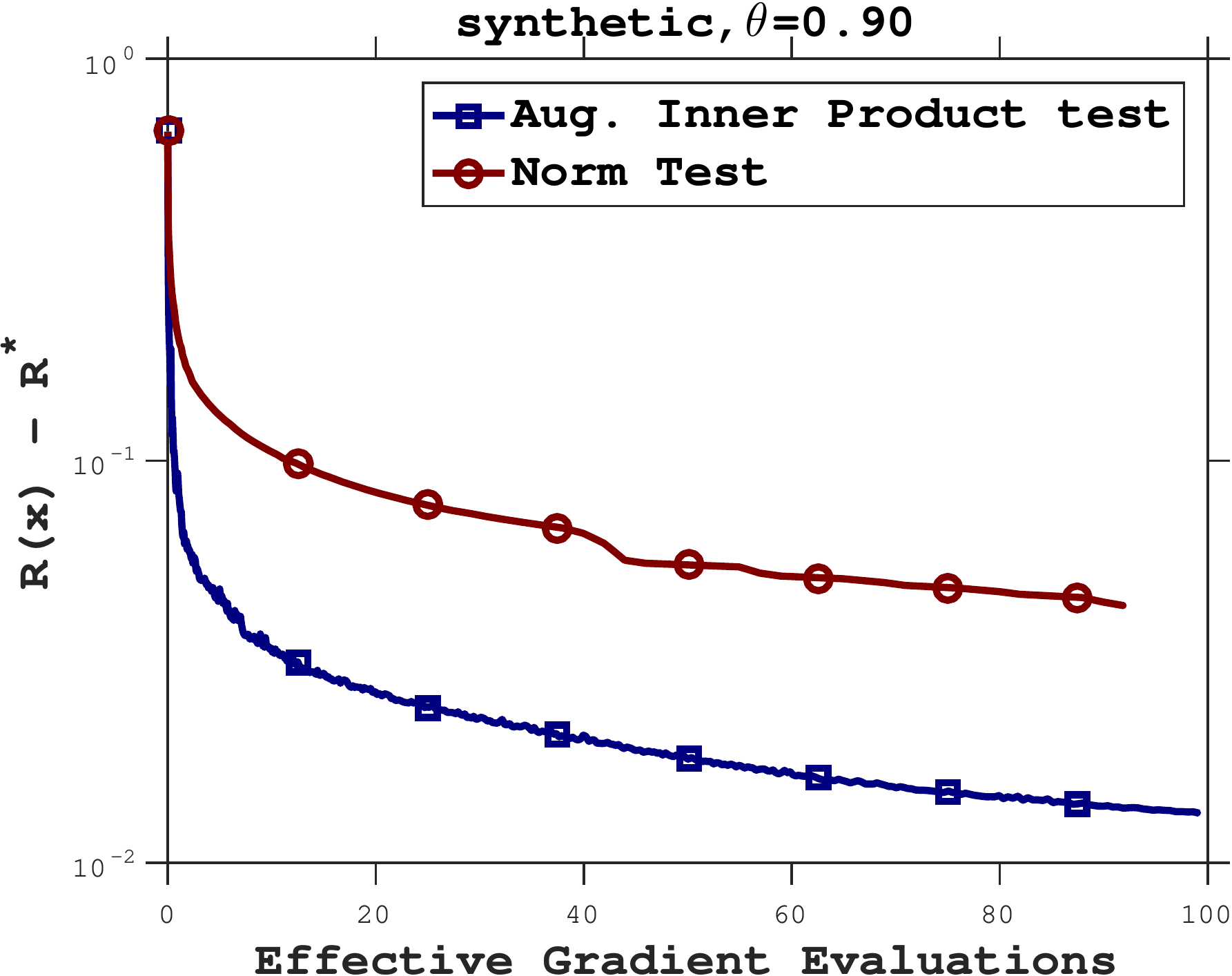}
		\includegraphics[width=0.45\linewidth]{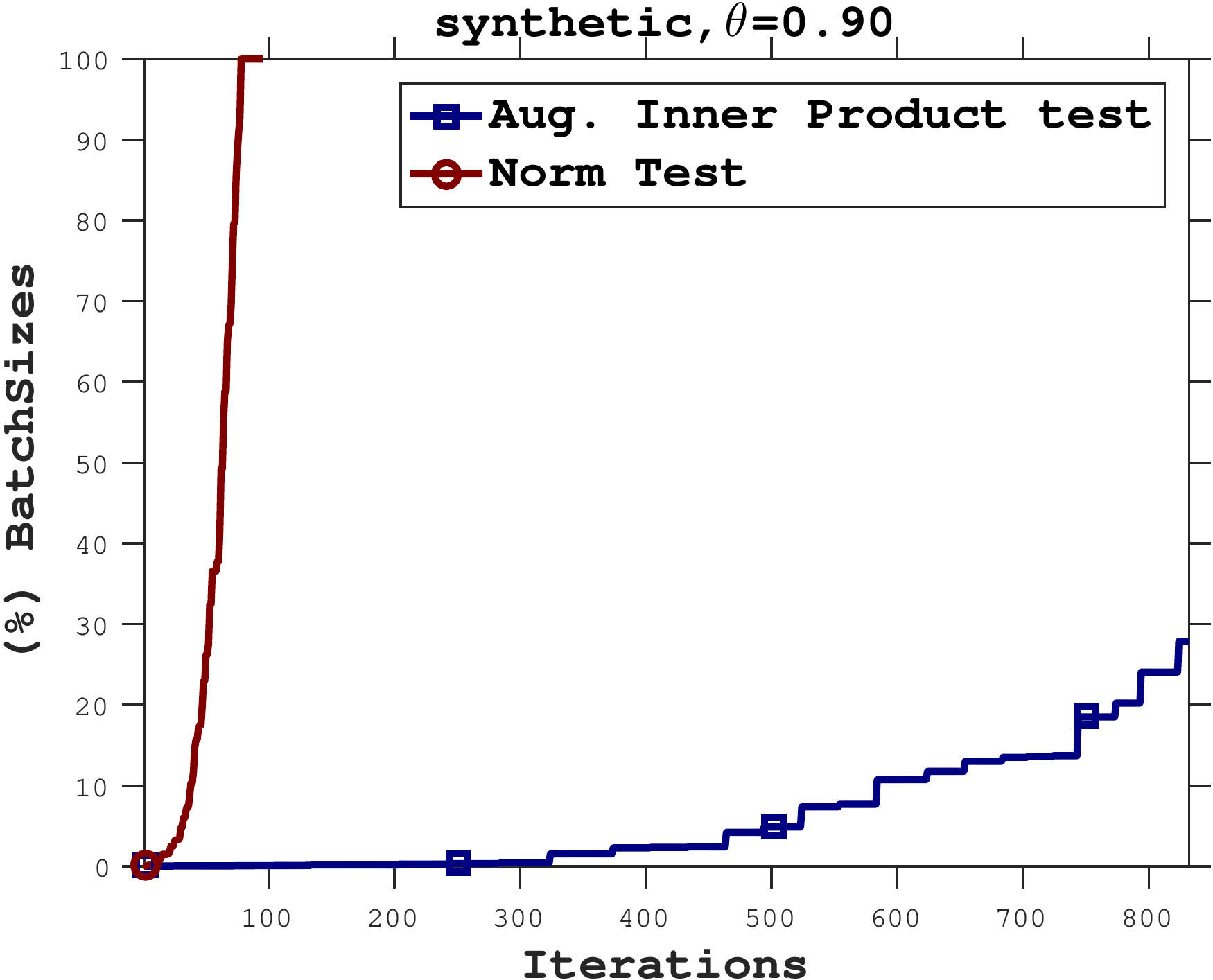}
		\par\end{centering}
	\caption{ {\tt \large synthetic} dataset:  Performance of Algorithm~\ref{alg:complete} using the augmented inner product test and using the norm test. Left: Function error vs effective gradient evaluations; Right: Batch size $|S_k|$ vs. iterations.} 
	\label{synthetic-ex2funcvals} 
\end{figure}

To illustrate the performance of the line search procedure, we report in Figure~\ref{synthetic-ex2steps}  the steplengths chosen at each iteration by the two versions of the algorithm. As a reference, we also plot the best value ($\alpha=1$) for the case when a fixed steplength is used. We observe that Algorithm~\ref{alg:complete}, using the augmented inner product test,  produced smaller steplengths initially, and that they were increased gradually until they approached 1.
\begin{figure}[!htp]
	\begin{centering}
		\includegraphics[width=0.5\linewidth]{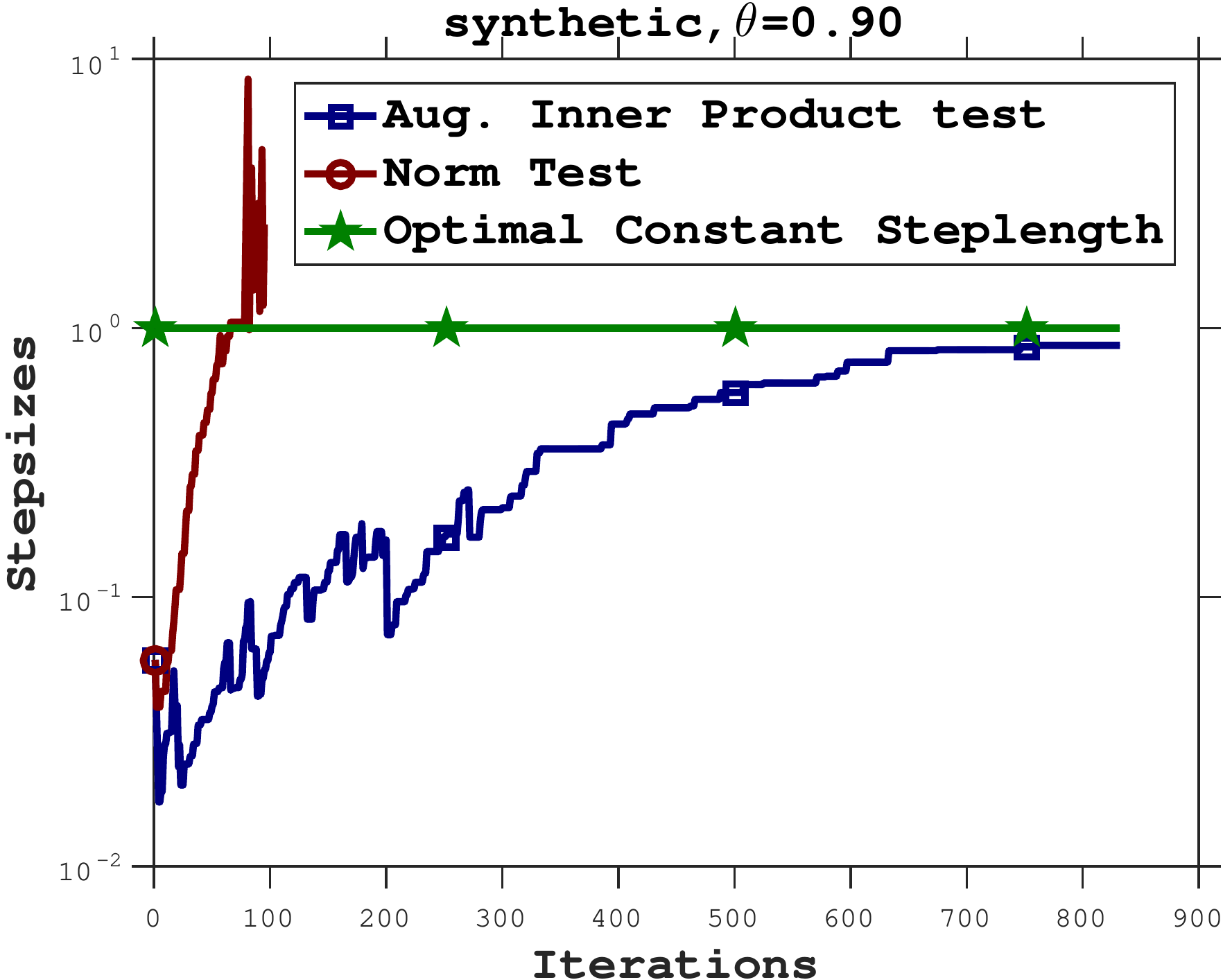}
		\par\end{centering}
	\caption{ {\tt \large synthetic} dataset: Performance of Algorithm~\ref{alg:complete} using the inner product test and using the norm test. Stepsizes vs. iterations} 
	\label{synthetic-ex2steps} 
\end{figure}

 The results in Appendix~\ref{addnumerical} show the same pattern as in the {\tt synthetic} dataset: Algorithm~\ref{alg:complete}  using the inner product test is more efficient, and increases the sample size more gradually, than the norm test.

\section{Final Remarks }   \label{sec:final}
Stochastic optimization algorithms that increase the sample size during the course of a run are interesting from a practical and theoretical perspective. We regard them as one of the most promising variance reducing methods for optimization applications involving very large data sets, as is the case in machine learning. Complexity analysis \cite{byrd2012sample} has shown that adaptive sampling algorithms can be efficient in terms of total computational work. However, implementing them in practice has been challenging because a natural criterion (\emph{the norm test}) for determining when gradient approximations should be improved  is difficult to control and often leads to inefficient behavior. In this paper, we propose a different strategy that provides more freedom in the choice of the search direction; we refer to it as the \emph{inner product test}. We argue that this technique is algorithmically appealing, and is more efficient in practice than the classical norm test strategy.

 In this paper, we considered  a first-order method but our methodology is amenable to the development of second-order methods \cite{friedlander2012hybrid,roosta2016sub,bollapragada2016exact}. An interesting open question is how to employ our inner product test in sub-sampled Newton and quasi-Newton methods for stochastic optimization. More generally, we believe that the concept underlying the inner product test  can be the basis for new variance reducing optimization algorithms. 

\bigskip\noindent{\em Acknowledgement.} We thank Albert Berahas for  many useful suggestions on how to improve the manuscript.

\bibliographystyle{plain}
\bibliography{Adasample}
%%%%%%%%%%%%%%%%%%%%%%%%%%

\newpage
\appendix

%%%%%%%%%%%%%%%%%%%%%%
\section{Proof of equations \eqref{beta1}-\eqref{beta2}}
%%%%%%%%%%%%%%%%%%%%%%
\begin{proof}
	The numerator of the left hand side term in the condition \eqref{inner-i} can be simplified by expanding the variance term and the inner product. That is,
	\begin{align*} \label{scvr-innerproduct}
	\E \left[\left( \nabla F_{i}(x)^T\nabla F(x_k) - \|\nabla F(x)\|^2 \right)^2\right] &= \E \left[\left( \nabla F_{i}(x)^T\nabla F(x)\right)^2\right] -  \|\nabla F(x)\|^4 \\
	&= \E \left[\|\nabla F_i(x)\|^2\cos^2(\chi_i)\right]\|\nabla F(x)\|^2 -  \|\nabla F(x)\|^4 .
	\end{align*} 
	 Substituting this equality in the inner product condition (\ref{inner-i}) we obtain, 
	\begin{equation} \label{scvr-cosine}
	\frac{\E \left[\|\nabla F_i(x)\|^2\cos^2(\chi_i)\right] -  \|\nabla F(x)\|^2}{|S|} \leq  \theta^2 \|\nabla F(x)\|^2.
	\end{equation}
	Let$|S_i|$ be the minimum number of samples such that this condition is satisfied. Therefore,
	\begin{align*}
	|S_i| &= \frac{\E \left[\|\nabla F_i(x)\|^2\cos^2(\chi_i)\right] -  \|\nabla F(x)\|^2}{\theta^2 \|\nabla F(x)\|^2}.
	\end{align*} 
	Now, we consider the norm test \eqref{dss-sample-test} and expand the variance in the numerator of the left hand side of the inequality to obtain
	\begin{equation} \label{dss}
	\frac{\E \left[\|\nabla F_i(x)\|^2\right] -  \|\nabla F(x)\|^2}{|S|} \leq  \theta^2 \|\nabla F(x)\|^2.
	\end{equation}
	If  $|S_n|$ is the minimum number of samples such that this condition is satisfied, we have
	\begin{align*}
	|S_n| &= \frac{\E \left[\|\nabla F_i(x)\|^2\right] -  \|\nabla F(x)\|^2}{\theta^2 \|\nabla F(x)\|^2}.
	\end{align*} 
	Therefore,
	\begin{equation} \label{ratio}
	\frac{|S_{i}|}{|S_{n}|} = \frac{\E[\|\nabla F_{i}(x)\|^2 \cos^2(\chi_i)]- \|\nabla F(x)\|^2}{\E[\|\nabla F_{i}(x)\|^2] - \|\nabla F(x)\|^2} 
	\end{equation}
	Now since,
	\begin{equation*}
	\E[\|\nabla F_{i}(x)\|^2 \cos^2(\chi_i)] \leq \E[\|\nabla F_{i}(x)\|^2],
	\end{equation*} 
	we have that $\beta(x) \leq 1$.
\end{proof}

\newpage
%%%%%%%%%%%%%%%%%%%%%%%%%%%%%%%%%%%%%%%%%%
\section{Additional Numerical Results} \label{addnumerical}
We present numerical results for the rest of the datasets listed in Table~\ref{tab11}.
\subsection{Tests with Fixed Steplengths}
First, we consider the use of a fixed steplength parameter $\alpha$, which is selected for each problem and each method so as to give optimal performance. In the figures that follow, $\alpha_n, \alpha_{ip}$ denote the steplengths used in conjunction with the norm and inner product tests, respectively. In all experiments the parameter $\theta$  in \eqref{final} and \eqref{final-dss} was set to  $\theta =0.9$, and the parameter $\nu$ in \eqref{orthogonal-p} was set to $\nu = \tan80^\circ=5.84$. 
\begin{figure}[H]
	\begin{centering}
		\includegraphics[width=0.3\linewidth]{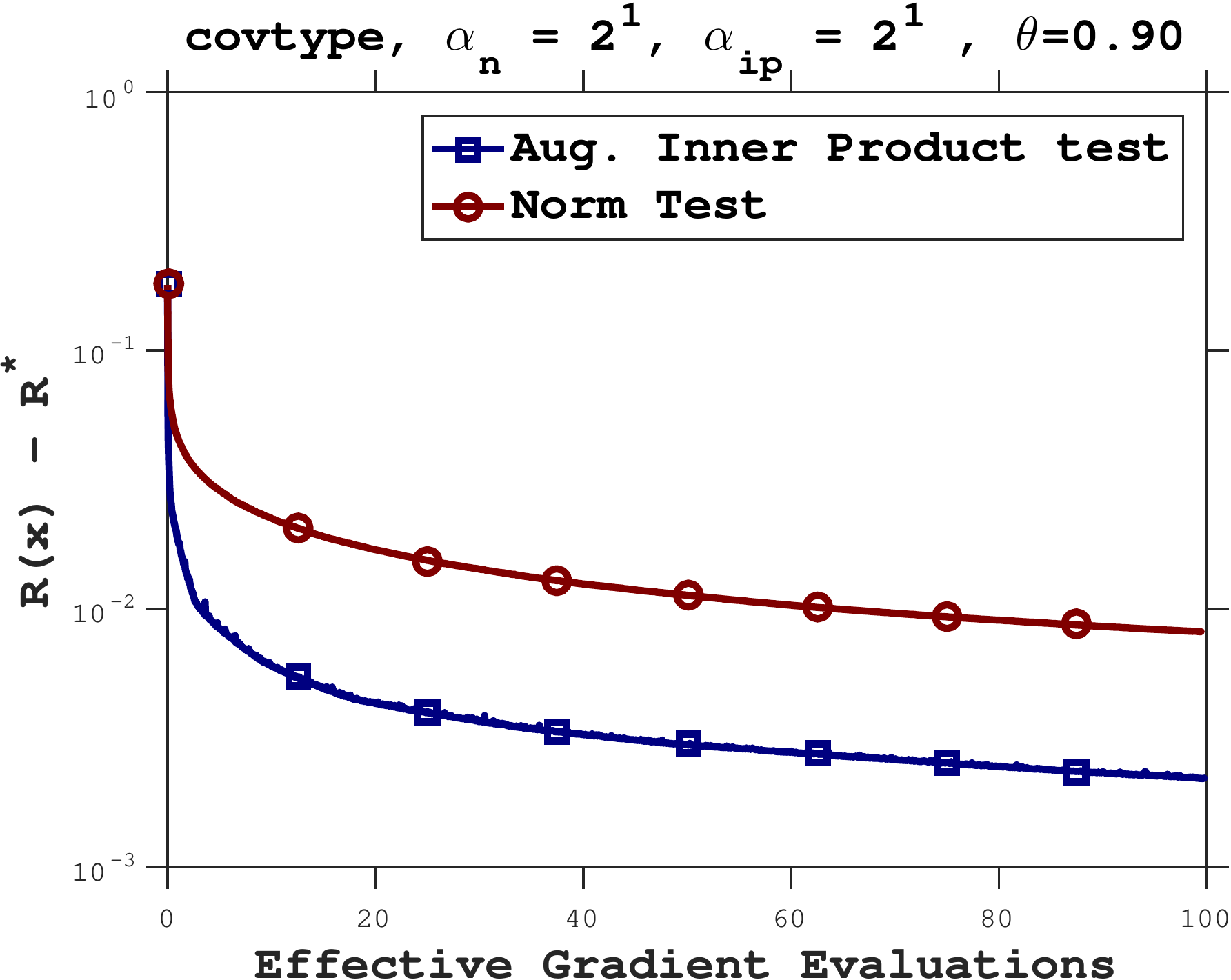}
		\includegraphics[width=0.3\linewidth]{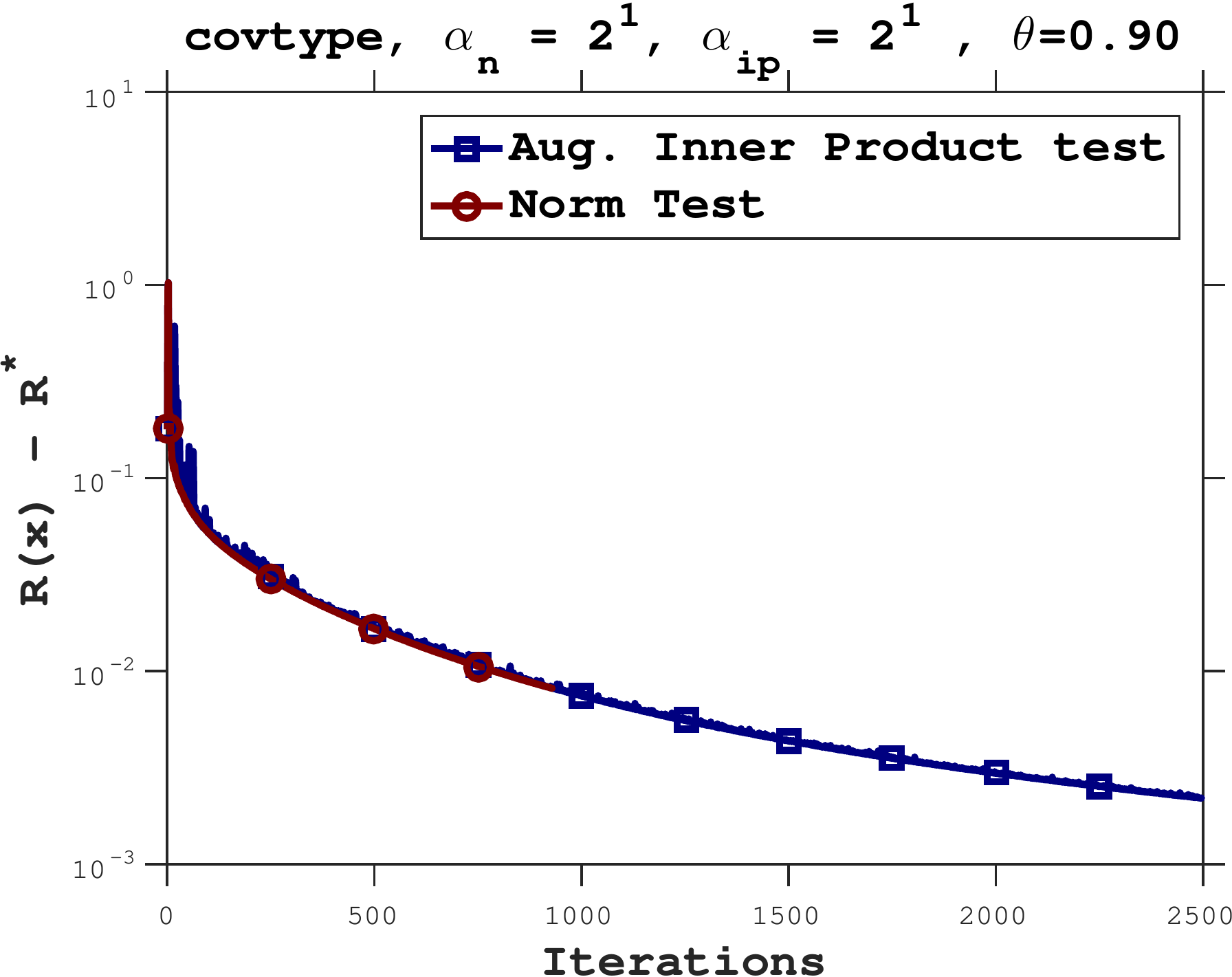}
		\includegraphics[width=0.3\linewidth]{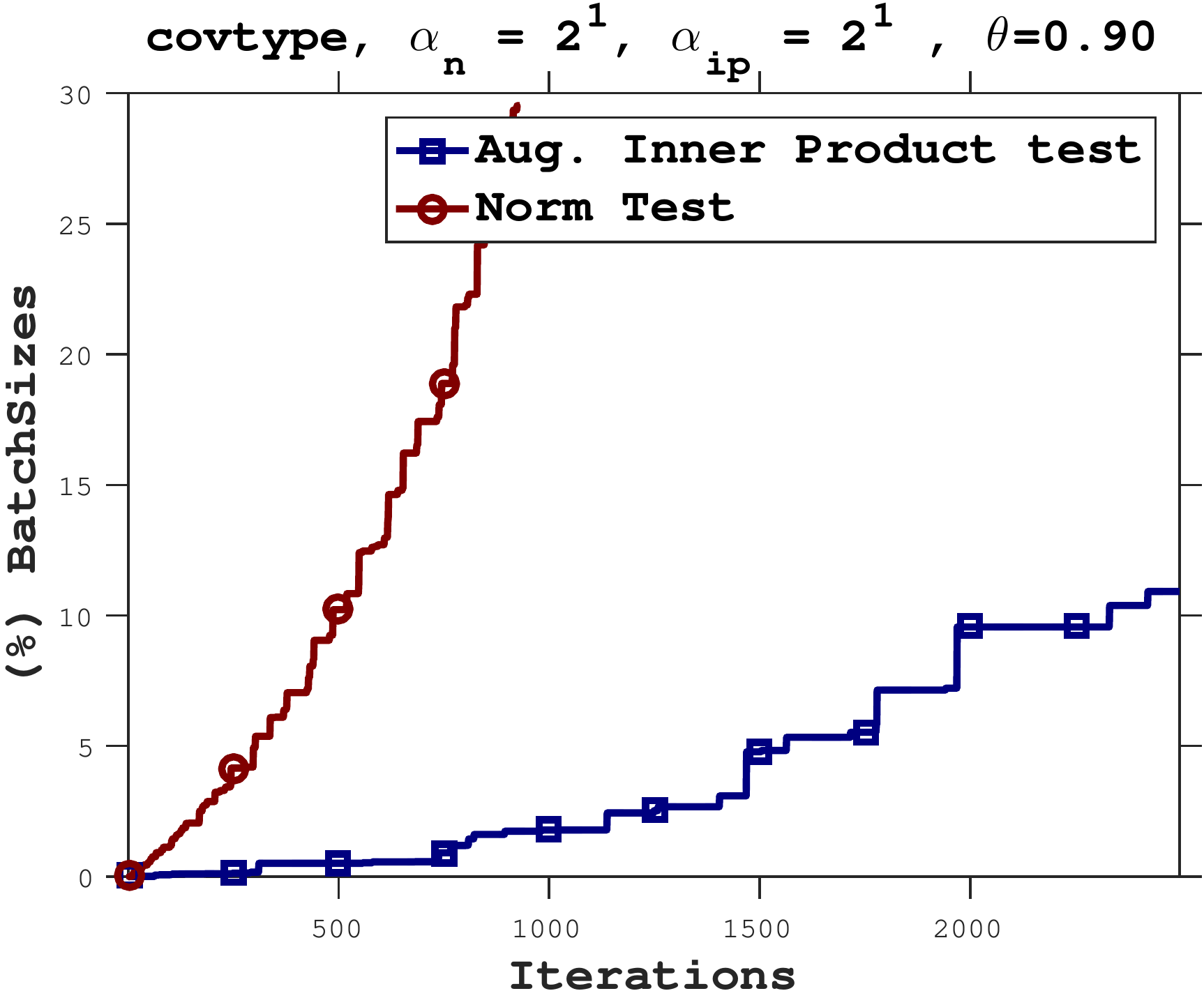}
		\par\end{centering}
	\caption{ {\tt covertype} dataset: Performance of the adaptive sampling algorithm using the augmented inner product test and the norm test. The steplengths were chosen as $\alpha_{n} =2^{1}$, $\alpha_{ip} =2^{1}$. Left: Function error vs. effective gradient evaluations; Middle: Function error vs. iterations; Right: Batch size $|S_k|$ vs. iterations.} 
	\label{covtype-expmnt1} 
\end{figure}

\begin{figure}[H]
	\begin{centering}
		\includegraphics[width=0.3\linewidth]{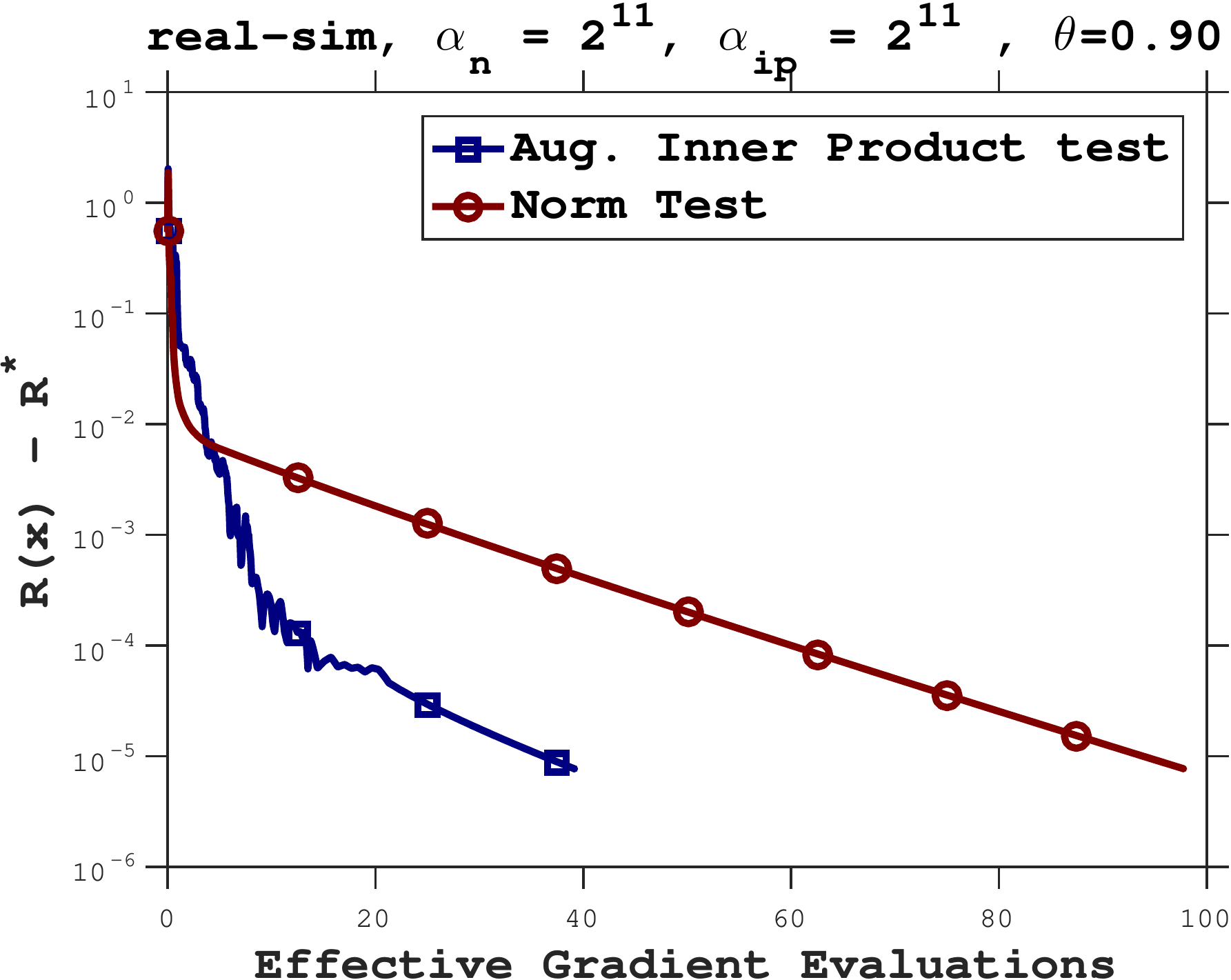}
		\includegraphics[width=0.3\linewidth]{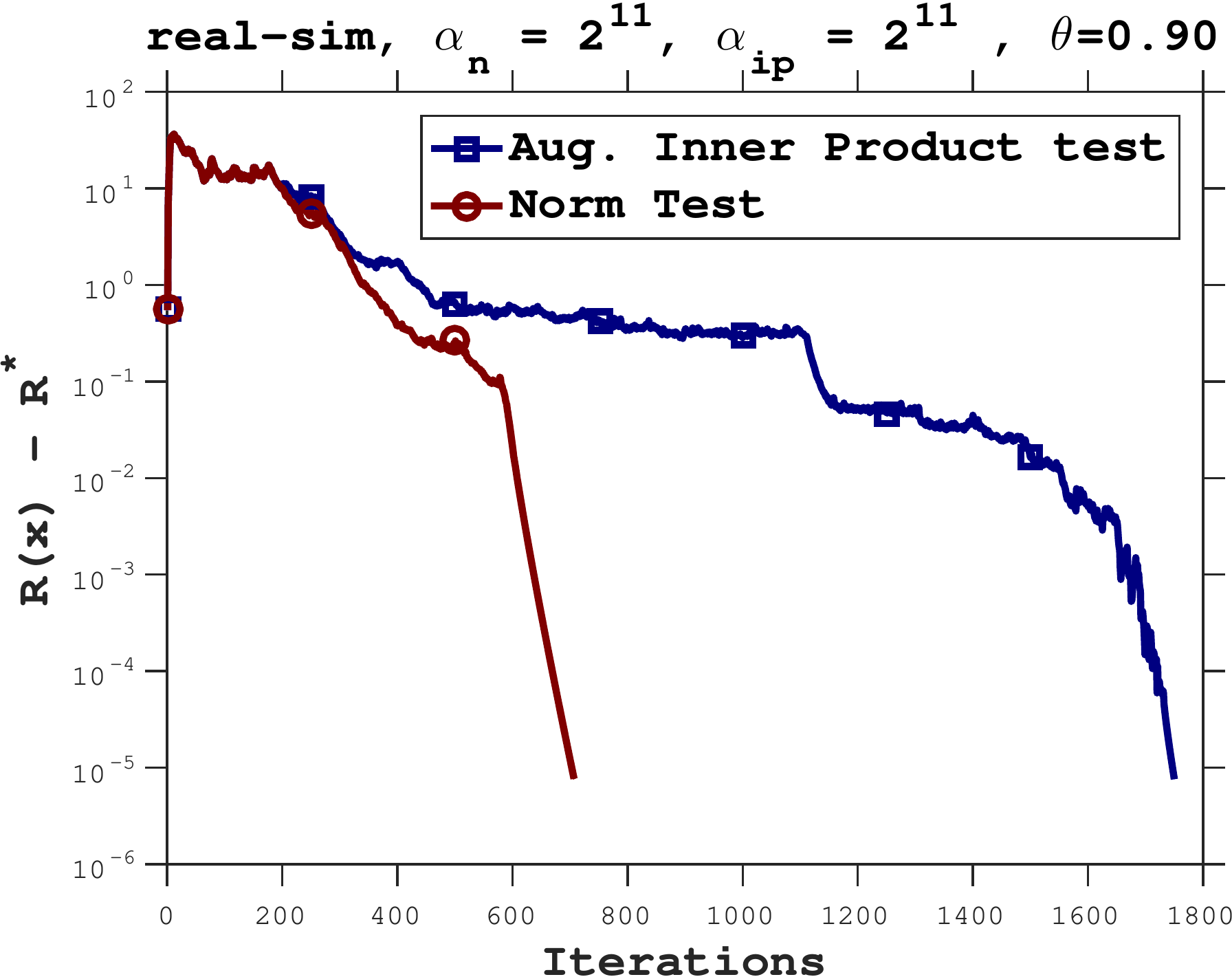}
		\includegraphics[width=0.3\linewidth]{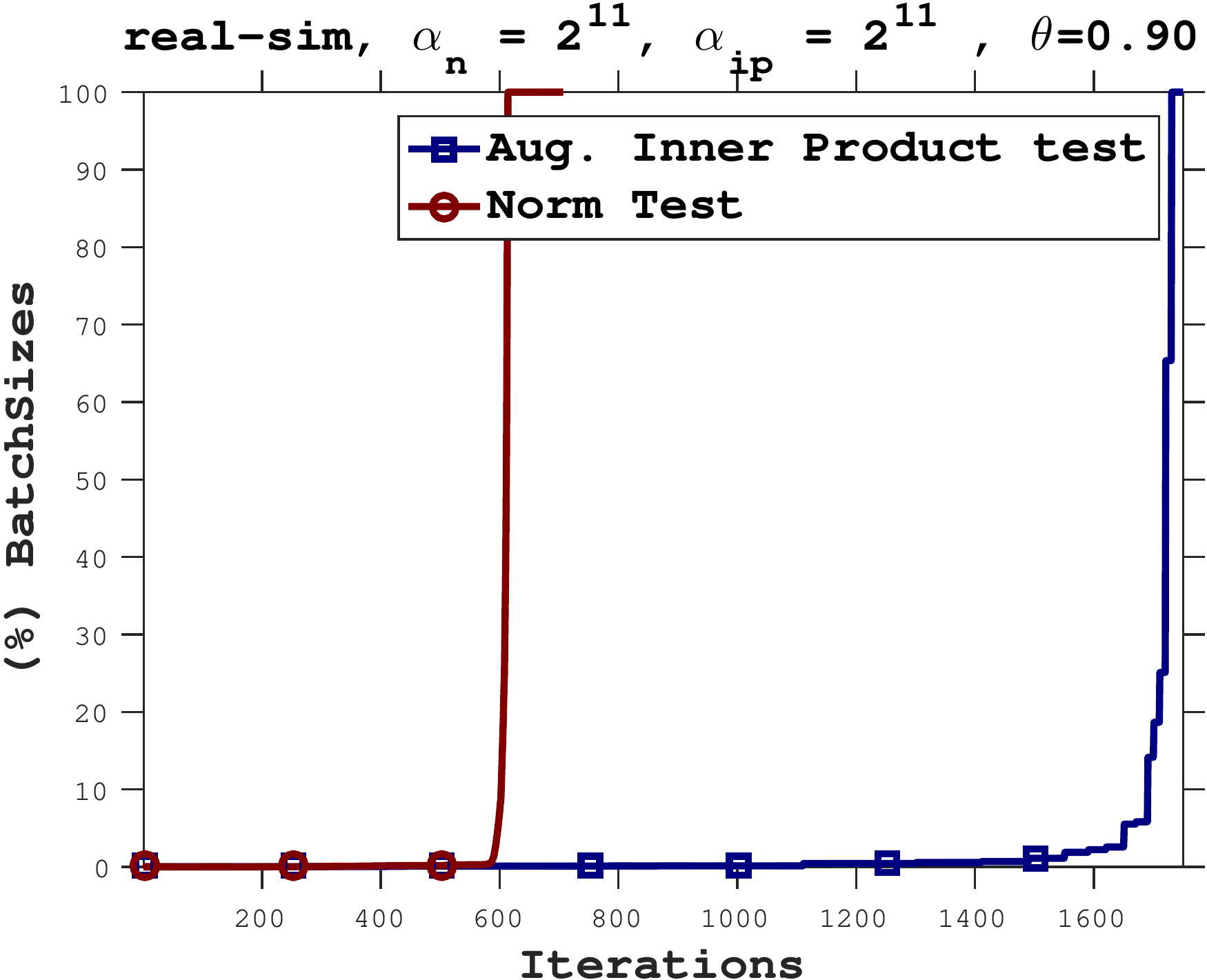}
		\par\end{centering}
	\caption{ {\tt real-sim} dataset: Performance of the adaptive sampling algorithm using the augmented inner product test and the norm test. The steplengths were chosen as $\alpha_{n} =2^{11}$, $\alpha_{ip} =2^{11}$.  Left: Function error vs. effective gradient evaluations; Middle: Function error vs. iterations; Right: Batch size $|S_k|$ vs. iterations.} 
	\label{real-sim-expmnt1} 
\end{figure}
\begin{figure}[H]
	\begin{centering}
		\includegraphics[width=0.3\linewidth]{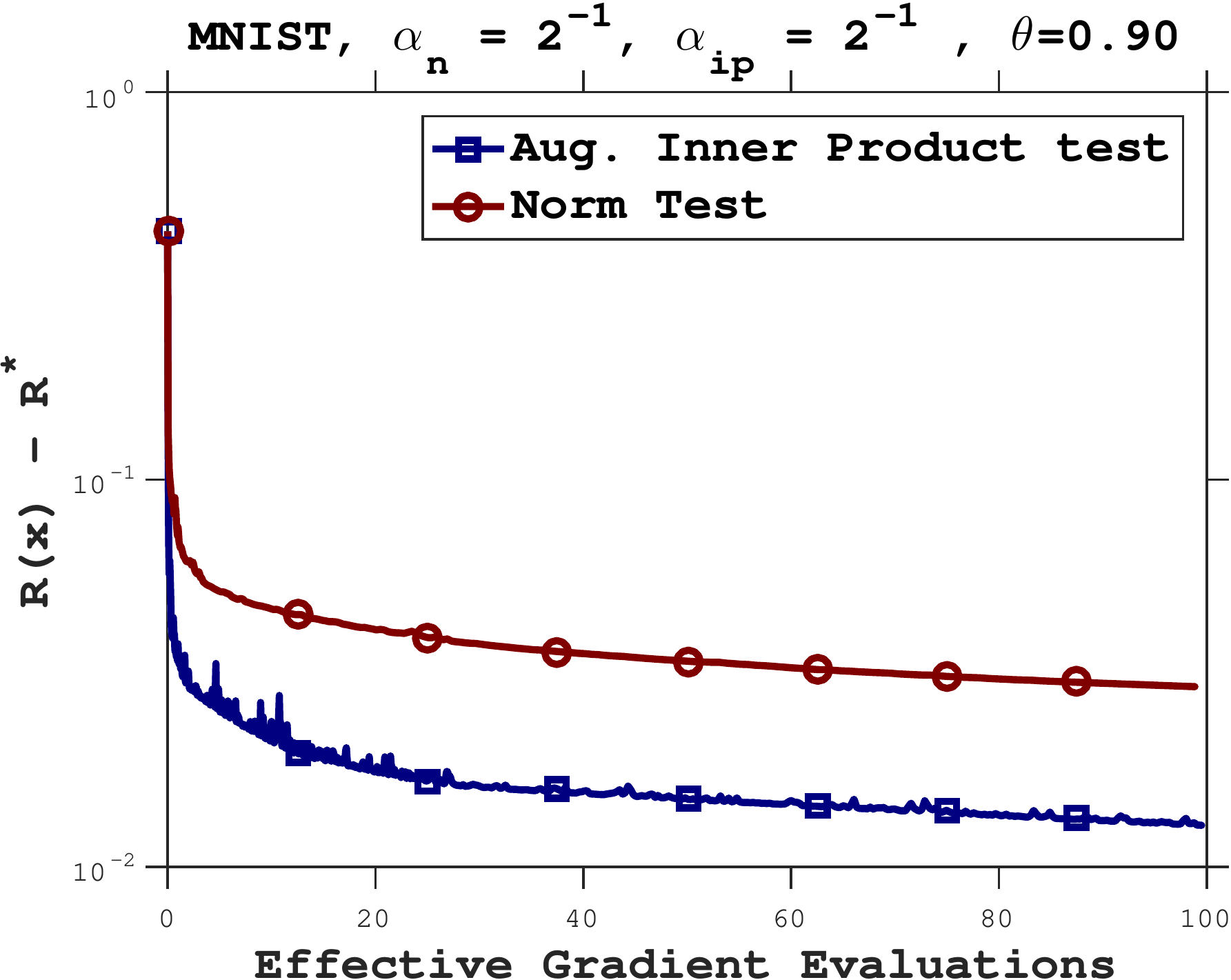}
		\includegraphics[width=0.3\linewidth]{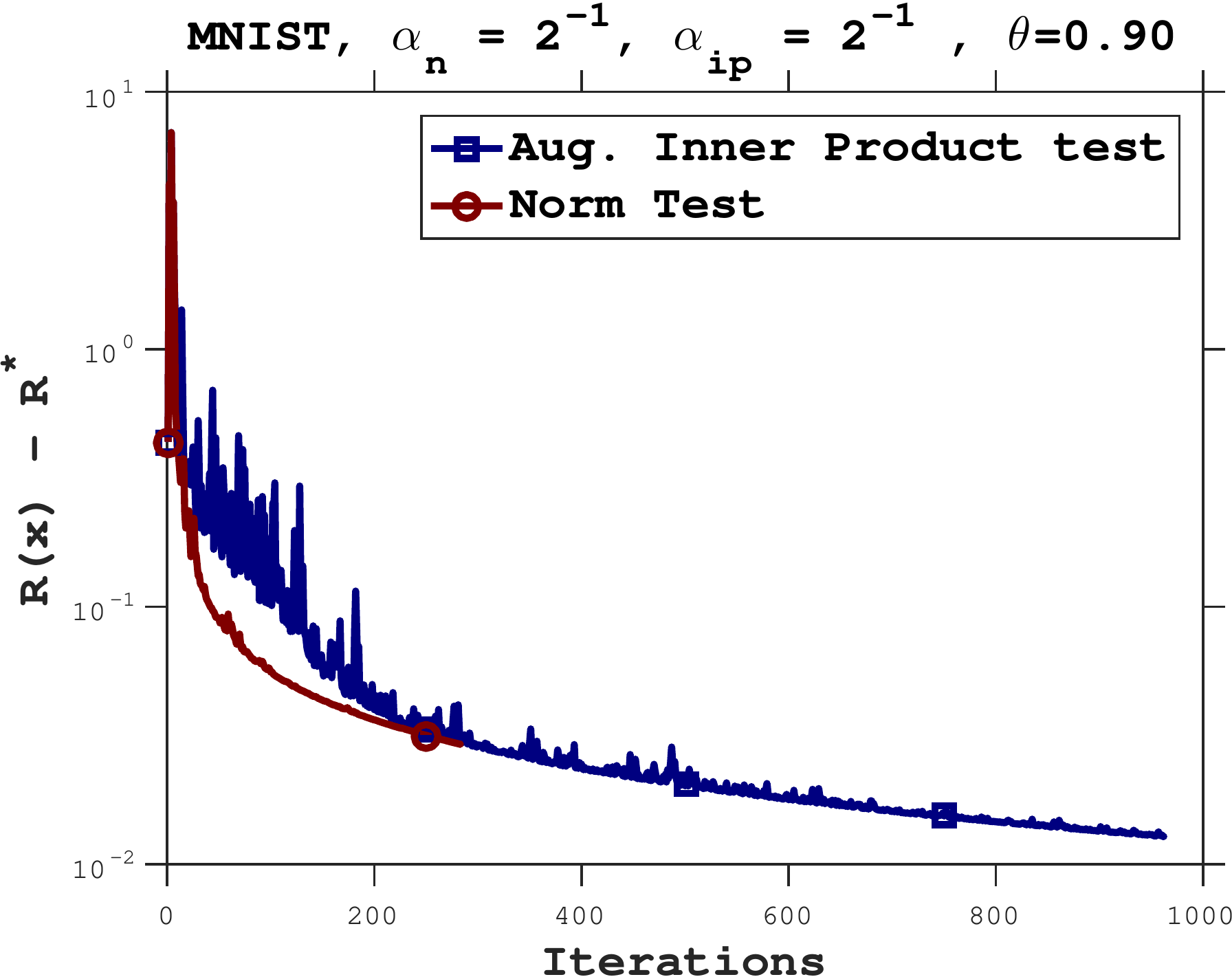}
		\includegraphics[width=0.3\linewidth]{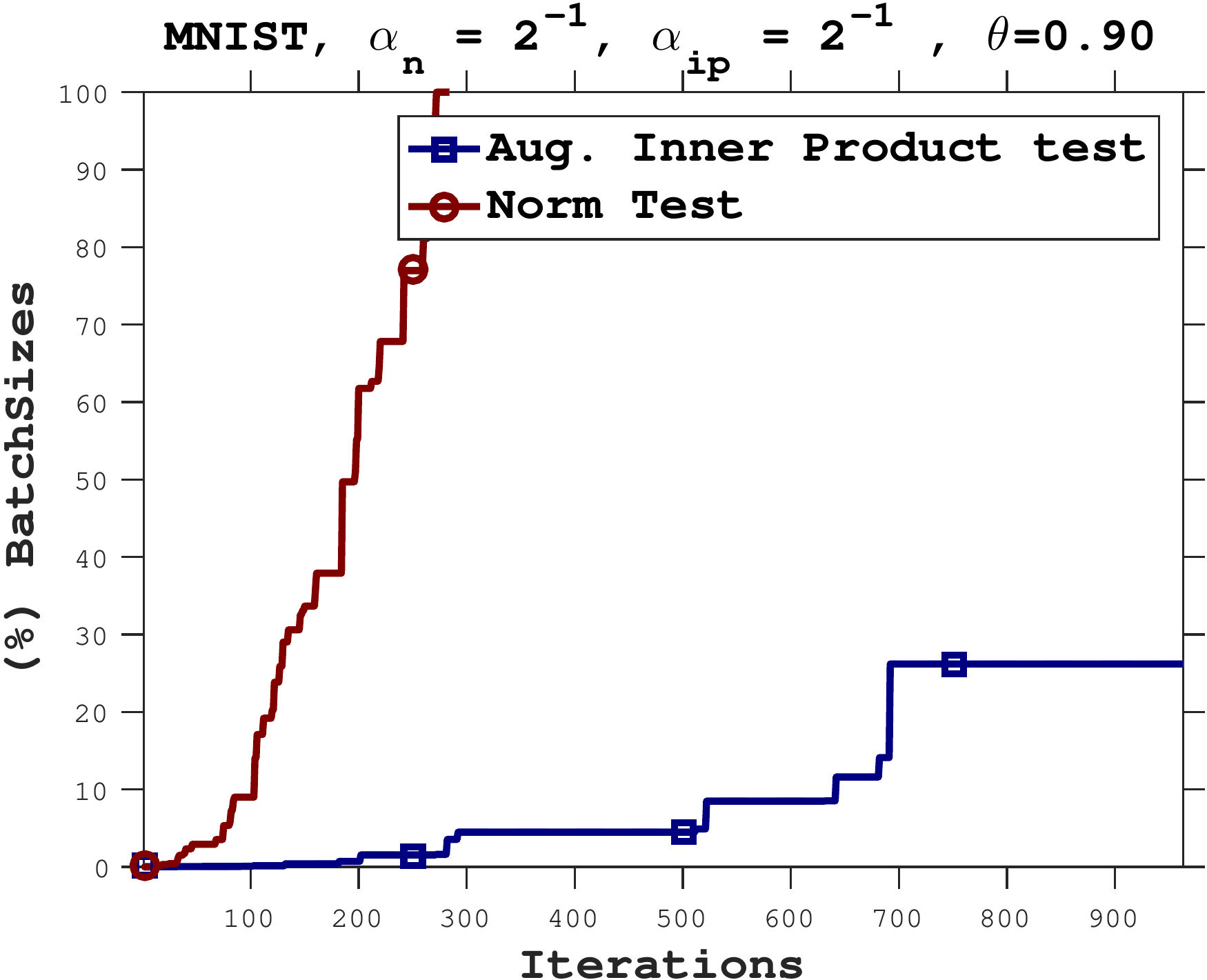}
		\par\end{centering}
	\caption{ {\tt MNIST} dataset: Performance of the adaptive sampling algorithm using the augmented inner product test and the norm test. The steplengths were chosen as  $\alpha_{n} =2^{-1}$, $\alpha_{ip} =2^{-1}$. Left: Function error vs. effective gradient evaluations; Middle: Function error vs. iterations; Right: Batch size $|S_k|$ vs. iterations.} 
	\label{MNIST-expmnt1} 
\end{figure}
\begin{figure}[H]
	\begin{centering}
		\includegraphics[width=0.3\linewidth]{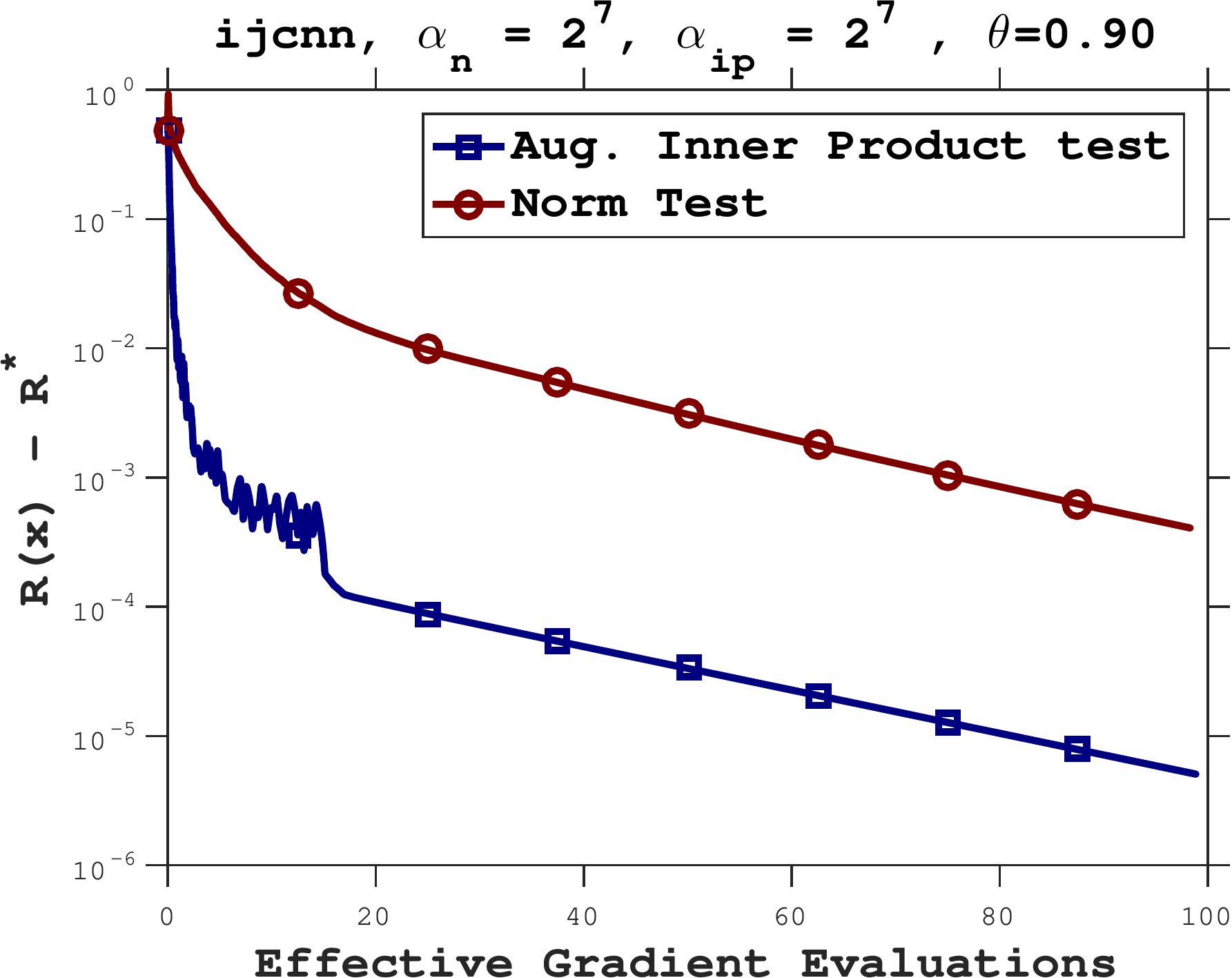}
		\includegraphics[width=0.3\linewidth]{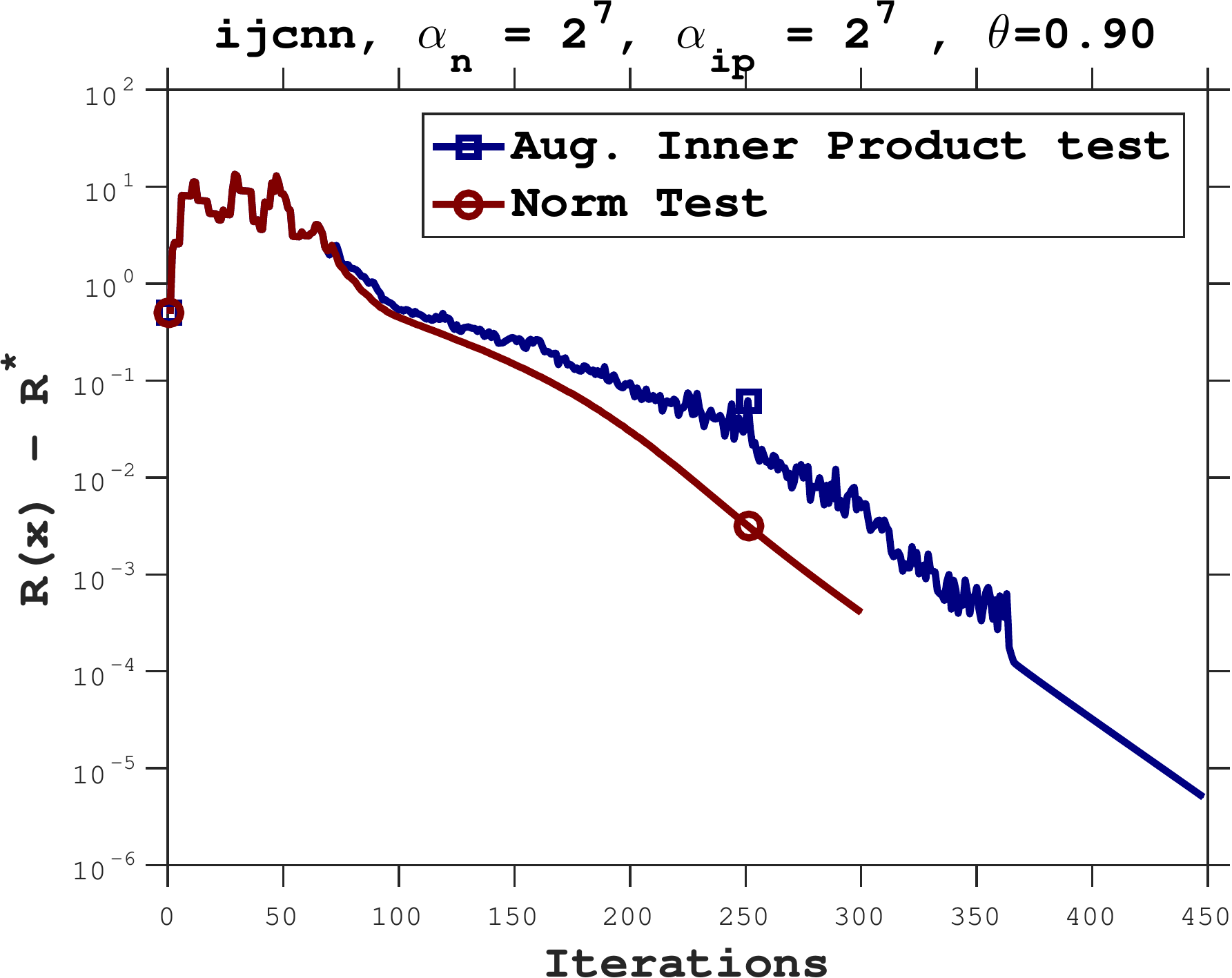}
		\includegraphics[width=0.3\linewidth]{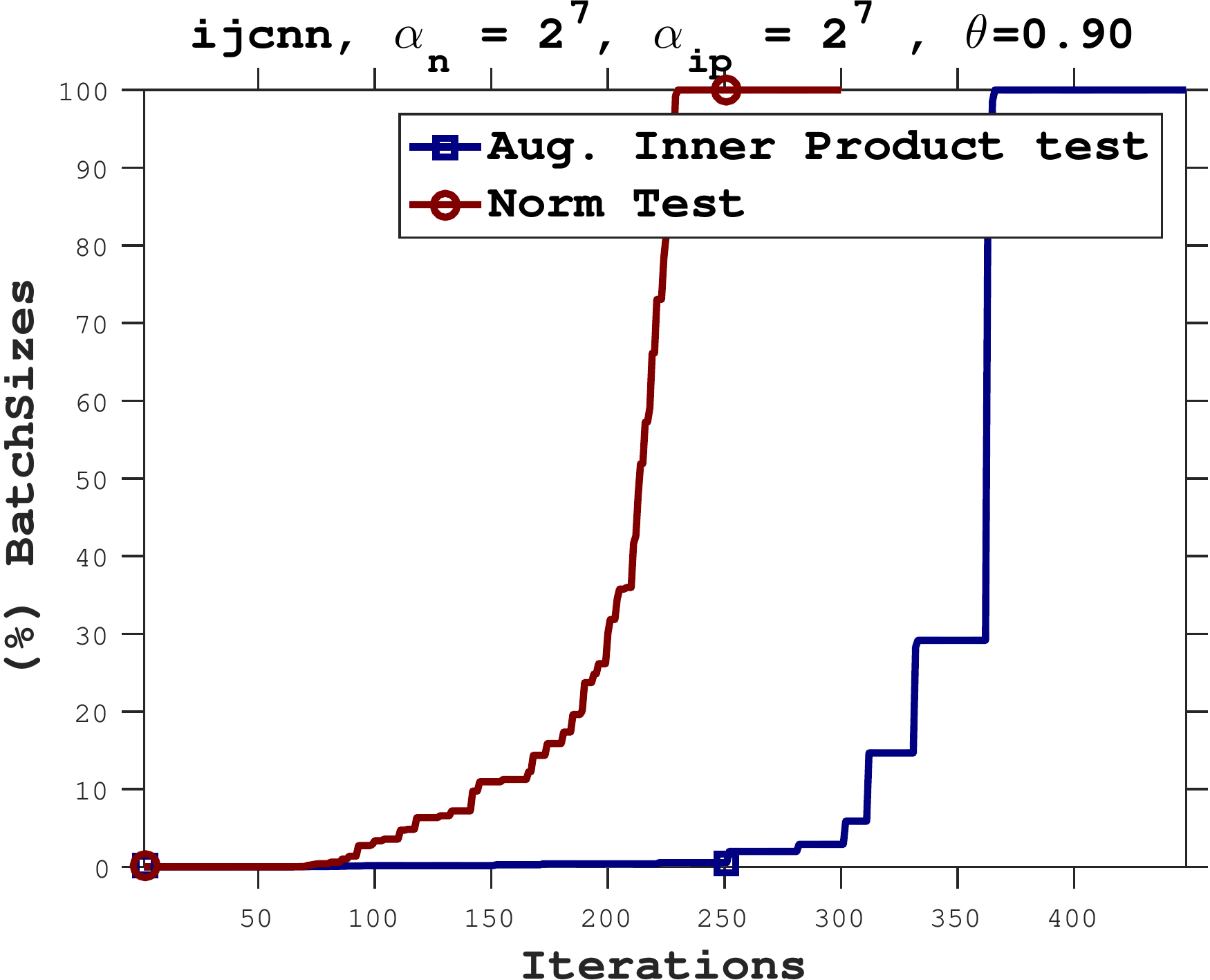}
		\par\end{centering}
	\caption{ {\tt ijcnn} dataset: Performance of the adaptive sampling algorithm using the augmented inner product test and the norm test. The steplengths were chosen as  $\alpha_{n} =2^{7}$, $\alpha_{ip} =2^{7}$.Left: Function error vs. effective gradient evaluations; Middle: Function error vs. iterations; Right: Batch size $|S_k|$ vs. iterations.} 
	\label{ijcnn-expmnt1} 
\end{figure}
\begin{figure}[H]
	\begin{centering}
		\includegraphics[width=0.3\linewidth]{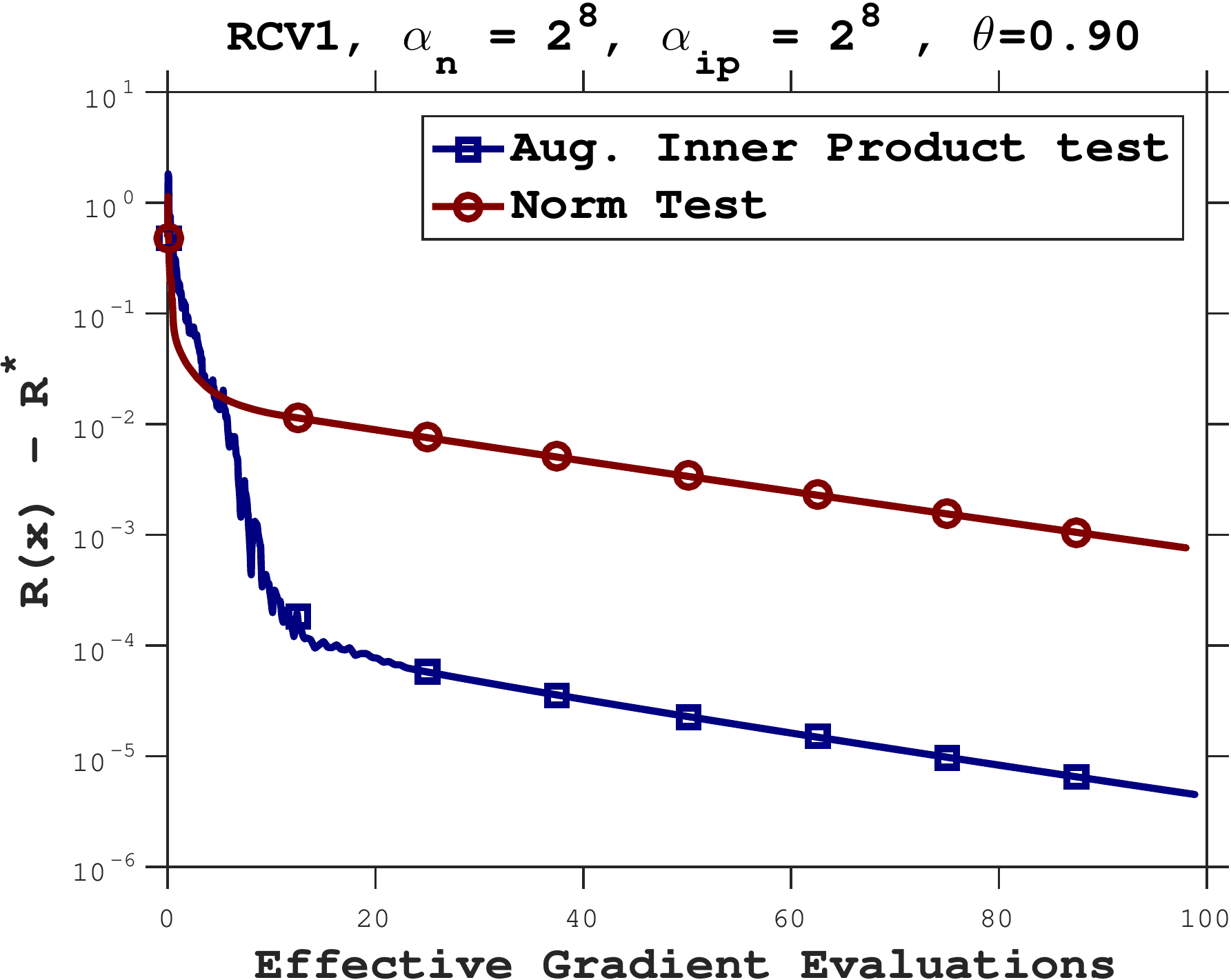}
		\includegraphics[width=0.3\linewidth]{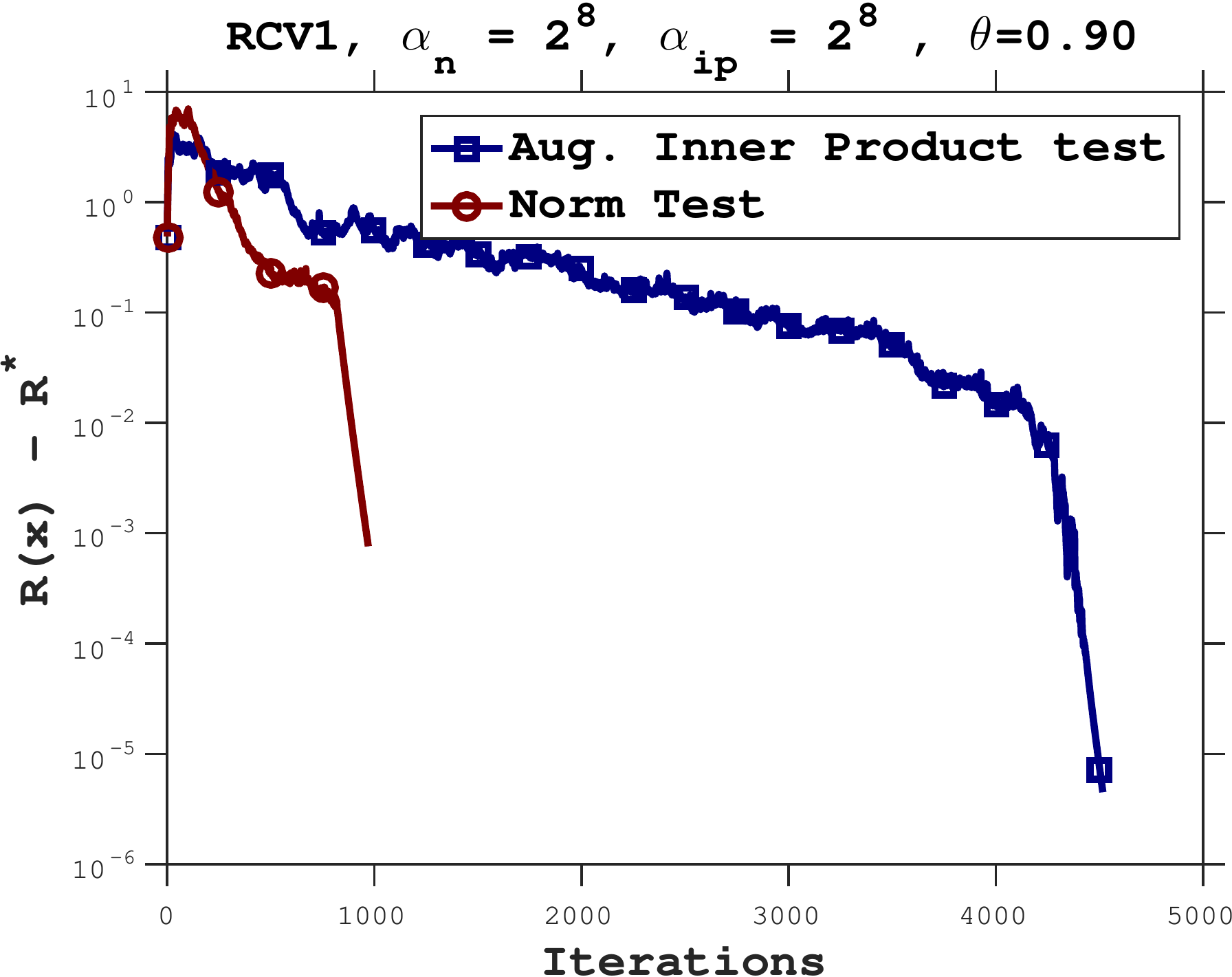}
		\includegraphics[width=0.3\linewidth]{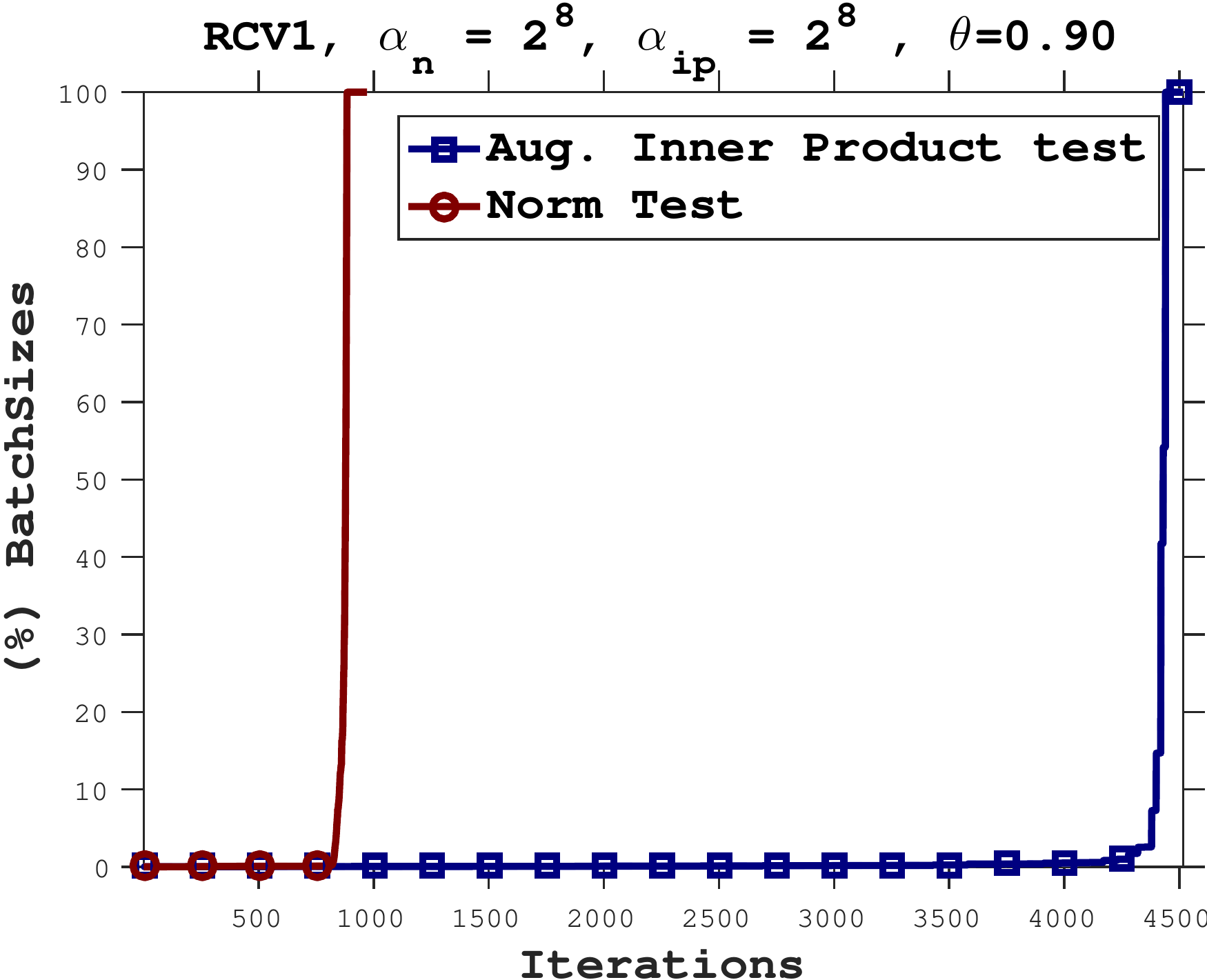}
		\par\end{centering}
	\caption{ {\tt RCV1} dataset: Performance of the adaptive sampling algorithm using the augmented inner product test and the norm test. The steplengths were chosen as  $\alpha_{n} =2^{8}$, $\alpha_{ip} =2^{8}$. Left: Function error vs. effective gradient evaluations; Middle: Function error vs. iterations; Right: Batch size $|S_k|$ vs. iterations.} 
	\label{RCV1-expmnt1} 
\end{figure}
\begin{figure}[H]
	\begin{centering}
		\includegraphics[width=0.3\linewidth]{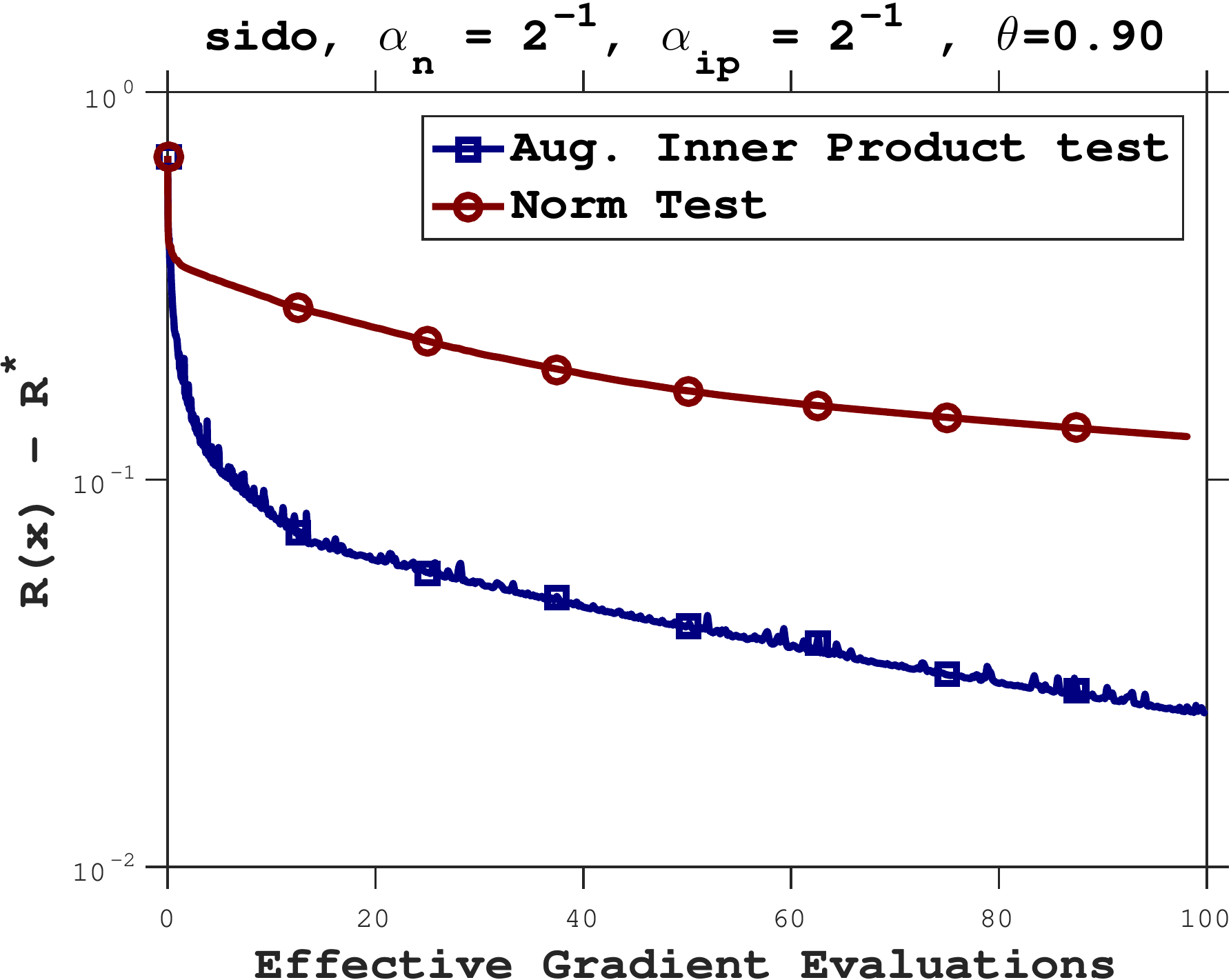}
		\includegraphics[width=0.3\linewidth]{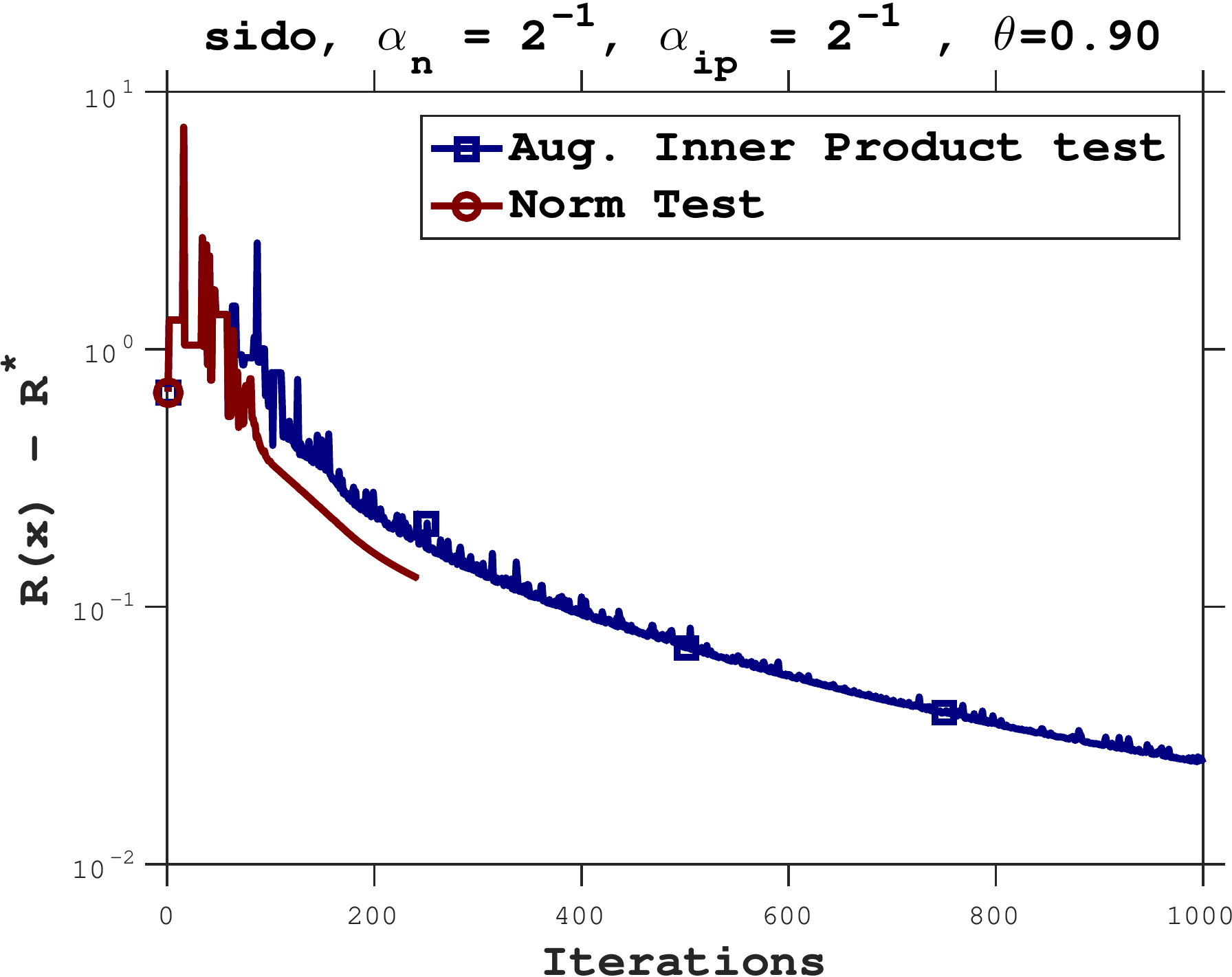}
		\includegraphics[width=0.3\linewidth]{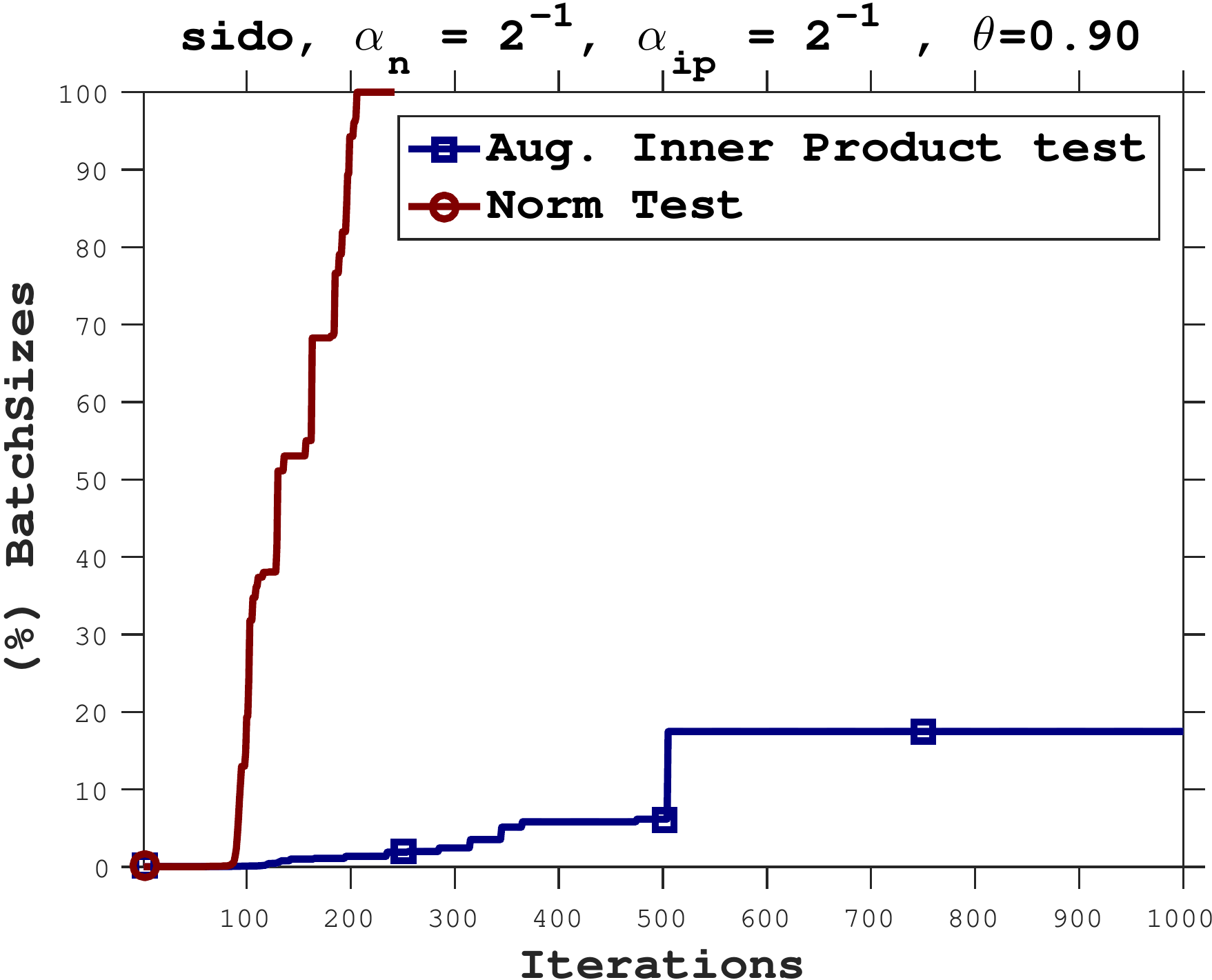}
		\par\end{centering}
	\caption{ {\tt sido} dataset: Performance of the adaptive sampling algorithm using the augmented inner product test and the norm test. The steplengths were chosen as  $\alpha_{n} =2^{-1}$, $\alpha_{ip} =2^{-1}$. Left: Function error vs. effective gradient evaluations; Middle: Function error vs. iterations; Right: Batch size $|S_k|$ vs. iterations.} 
	\label{sido-expmnt1} 
\end{figure}
\begin{figure}[H]
	\begin{centering}
		\includegraphics[width=0.3\linewidth]{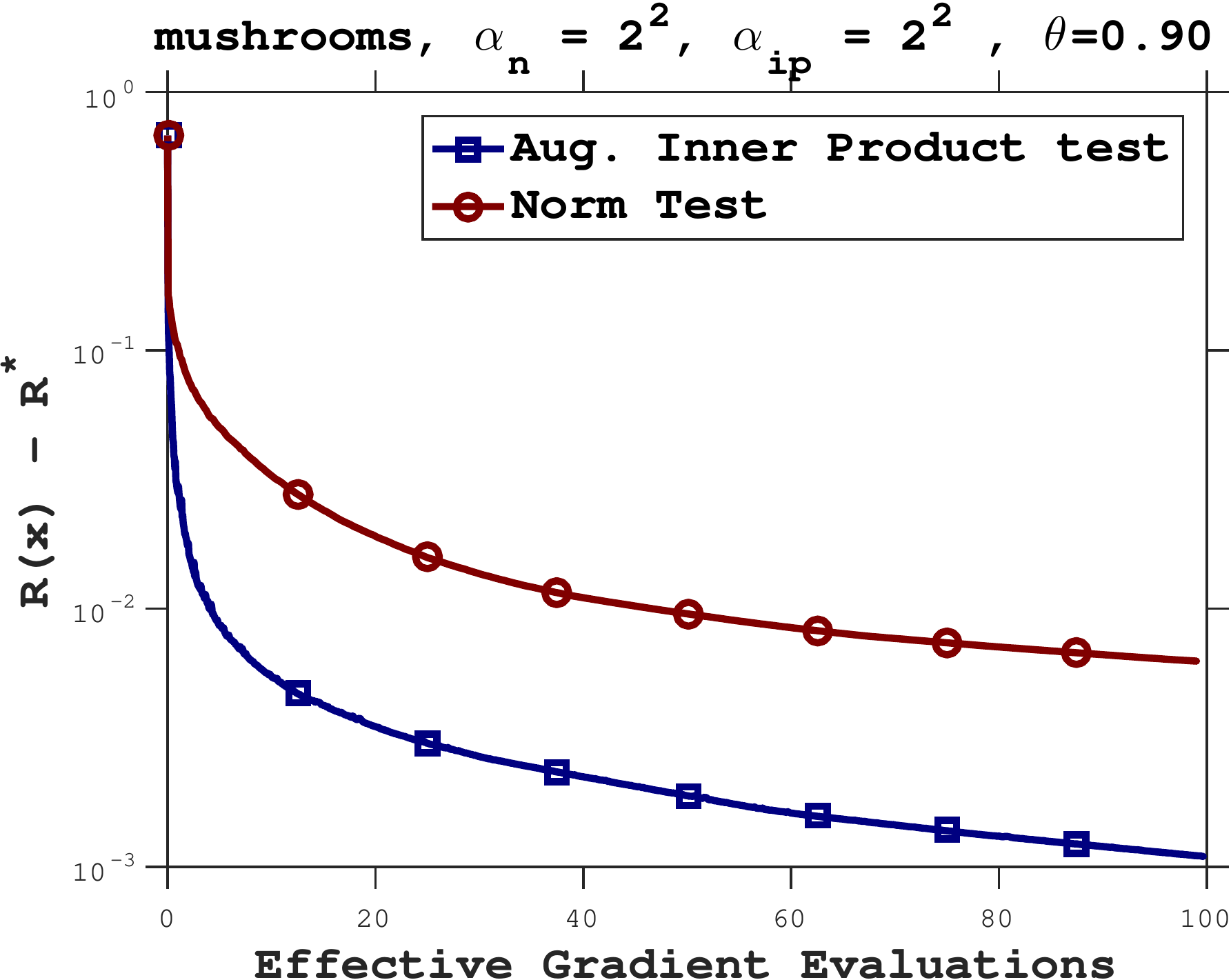}
		\includegraphics[width=0.3\linewidth]{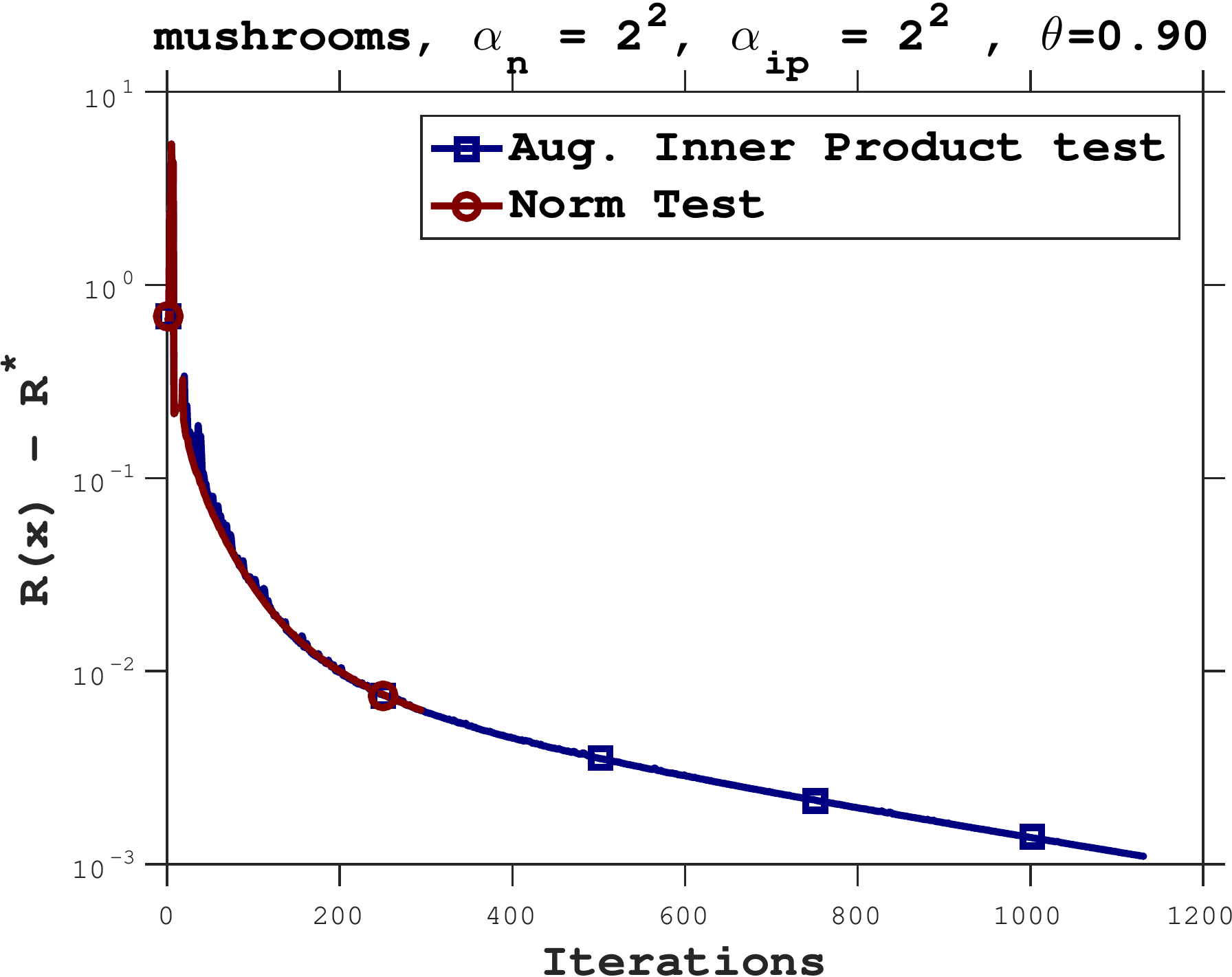}
		\includegraphics[width=0.3\linewidth]{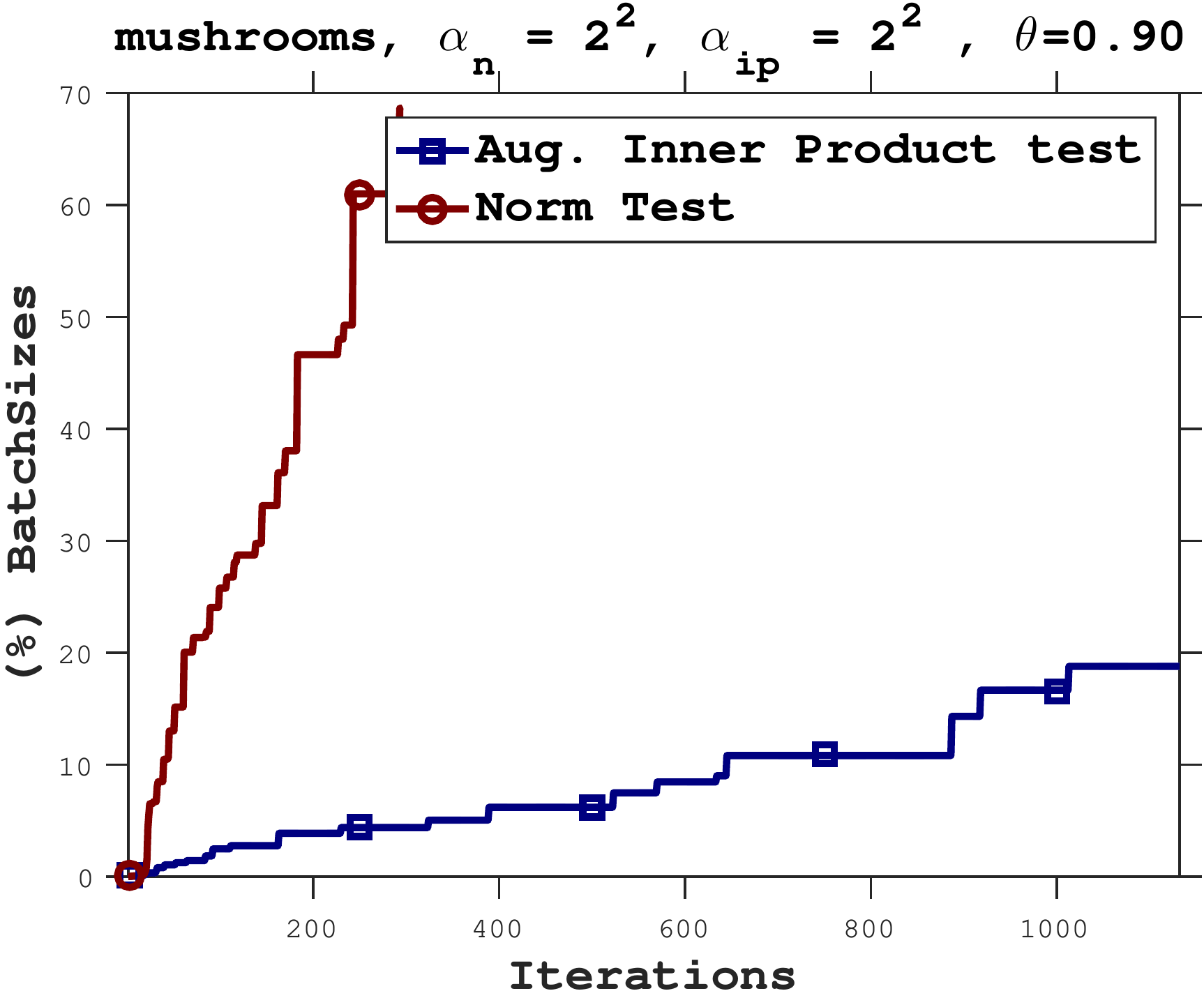}
		\par\end{centering}
	\caption{ {\tt mushrooms} dataset: Performance of the adaptive sampling algorithm using the augmented inner product test and the norm test. The steplengths were chosen as  $\alpha_{n} =2^{2}$, $\alpha_{ip} =2^{2}$. Left: Function error vs. effective gradient evaluations; Middle: Function error vs. iterations; Right: Batch size $|S_k|$ vs. iterations.} 
	\label{mushrooms-expmnt1} 
\end{figure}
\begin{figure}[H]
	\begin{centering}
		\includegraphics[width=0.3\linewidth]{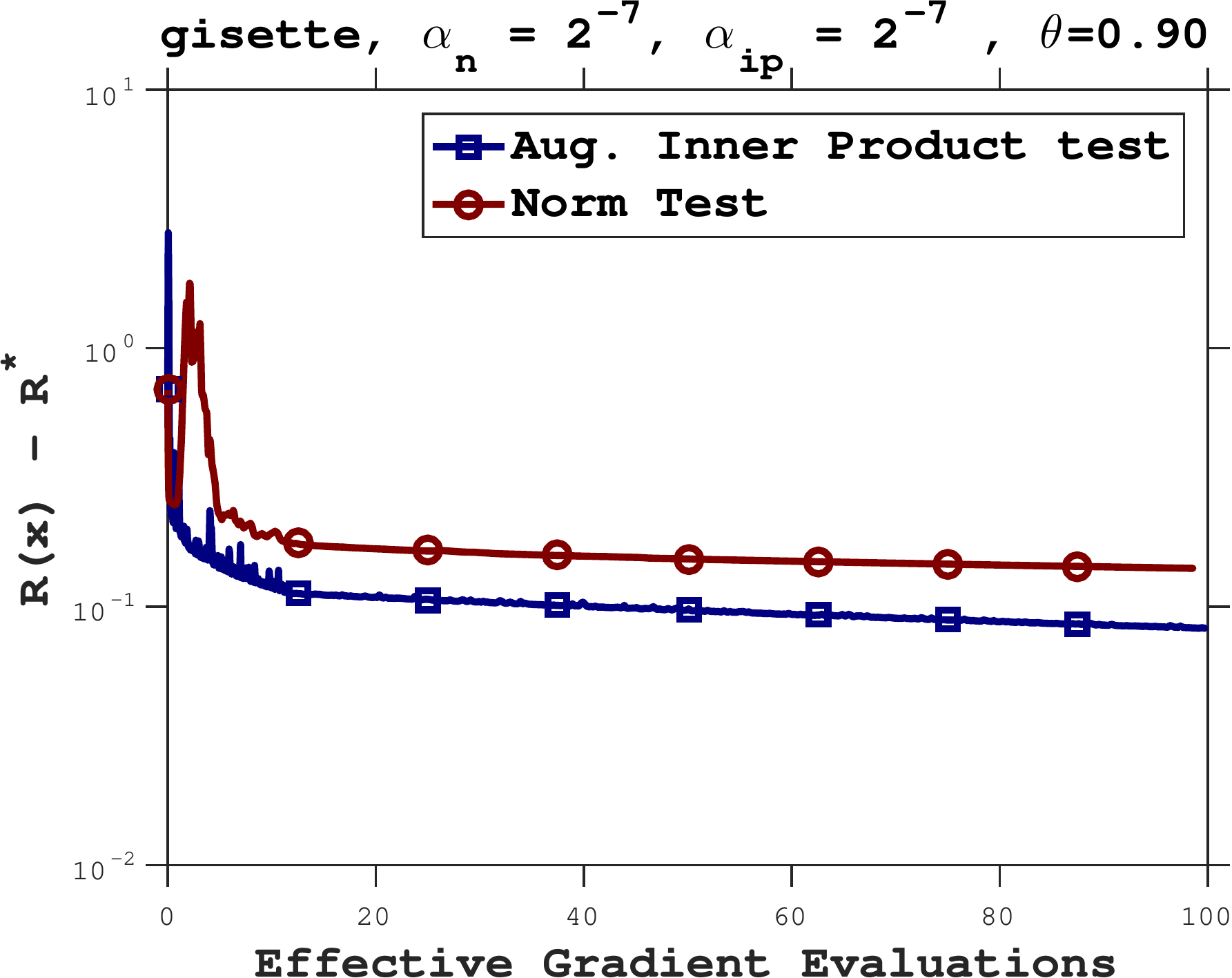}
		\includegraphics[width=0.3\linewidth]{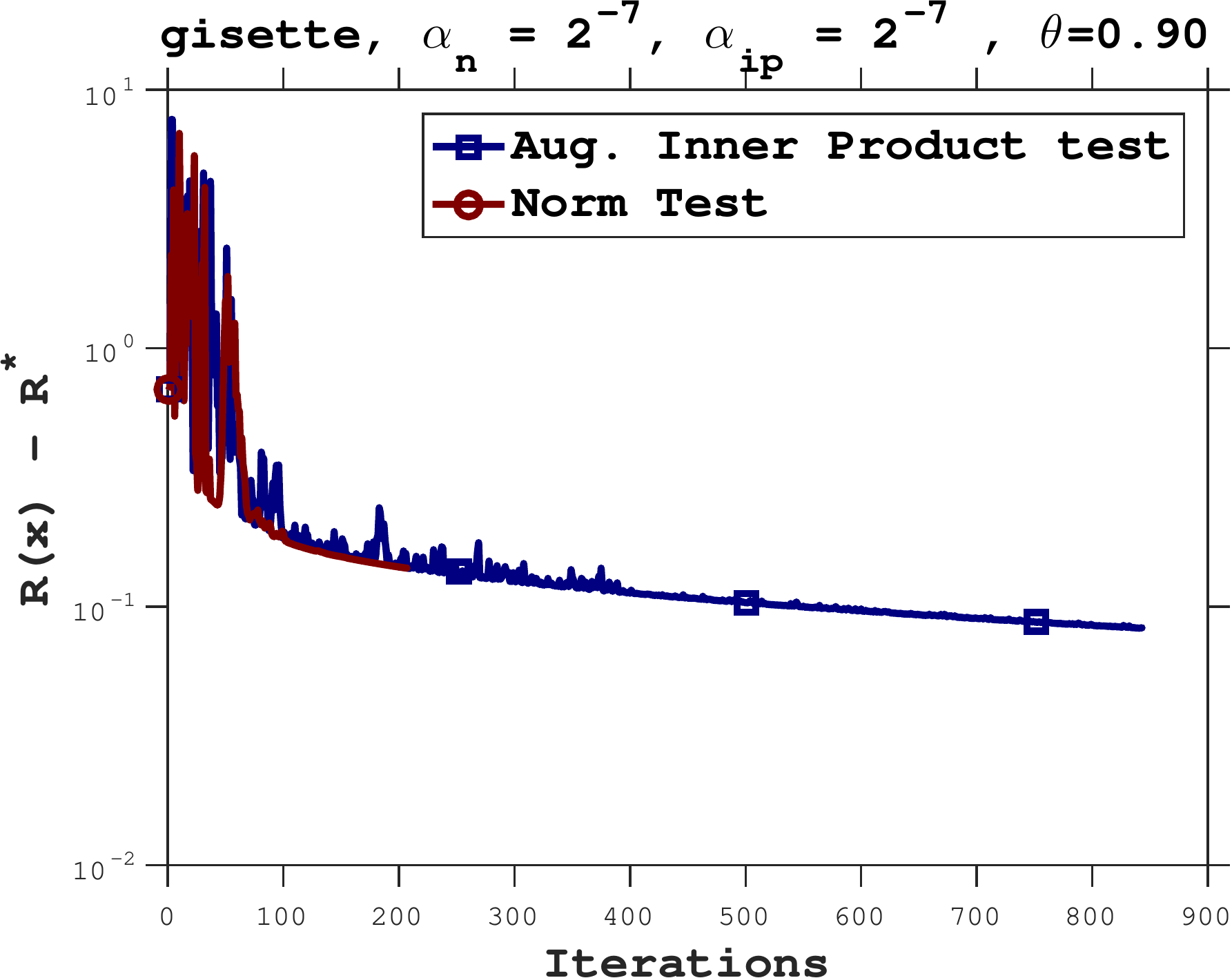}
		\includegraphics[width=0.3\linewidth]{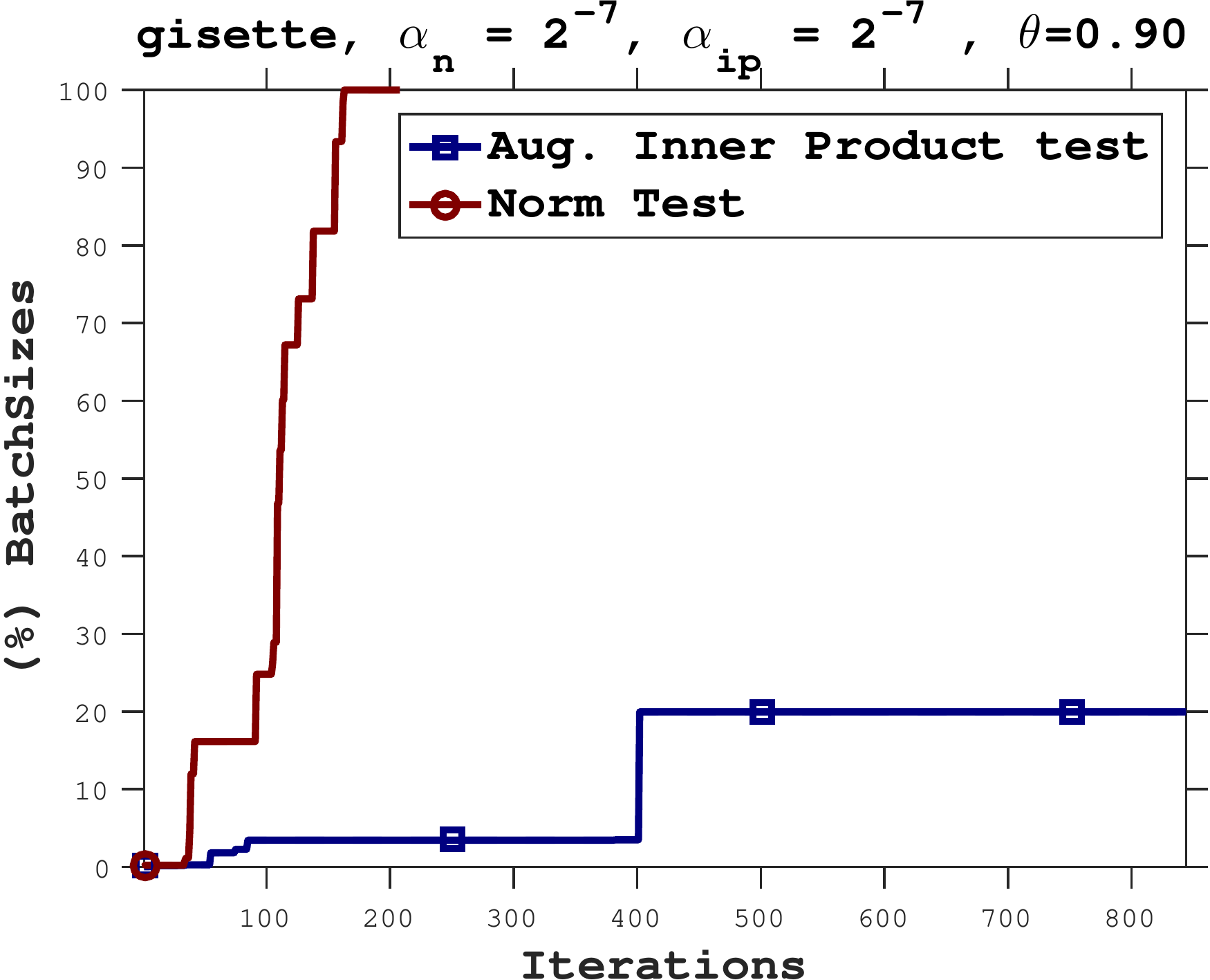}
		\par\end{centering}
	\caption{ {\tt Gisette} dataset: Performance of the adaptive sampling algorithm using the augmented inner product test and the norm test. The steplengths were chosen as $\alpha_{n} =2^{-7}$, $\alpha_{ip} =2^{-7}$. Left: Function error vs. effective gradient evaluations; Middle: Function error vs. iterations; Right: Batch size $|S_k|$ vs. iterations.} 
	\label{gisette-expmnt1} 
\end{figure}
\newpage
\subsection{Performance of the Complete Algorithm}
We now consider the performance of Algorithm~\ref{alg:complete} which uses a line search procedure to select the steplength $\alpha_k$ at every iteration. The parameters $\theta$ and $\nu$ were chosen as stated above. The initial estimate of the Lipschitz constant was  $L_0=1$, and we set $\eta=1.5$.
\begin{figure}[H]
	\begin{centering}
		\includegraphics[width=0.3\linewidth]{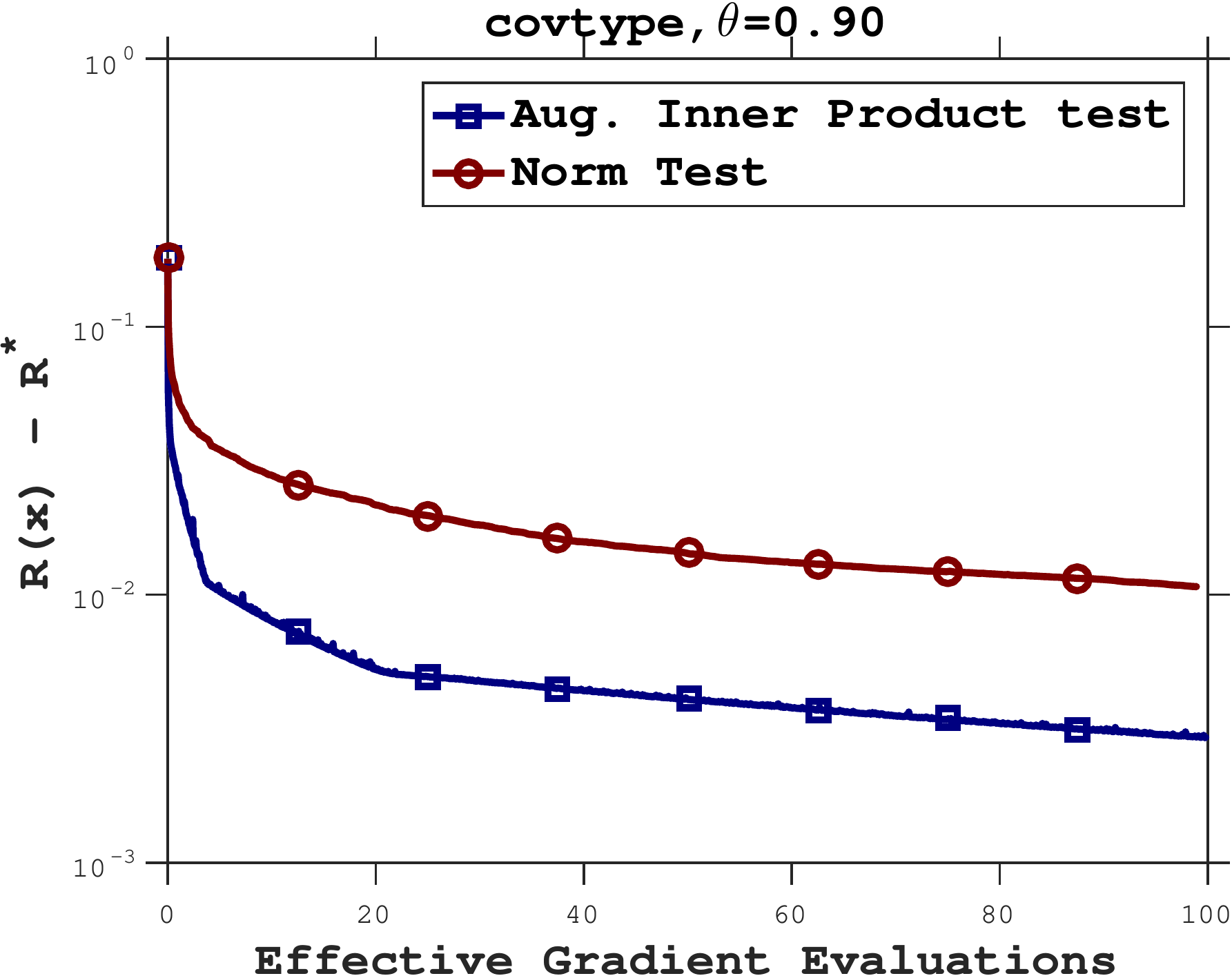}
		\includegraphics[width=0.3\linewidth]{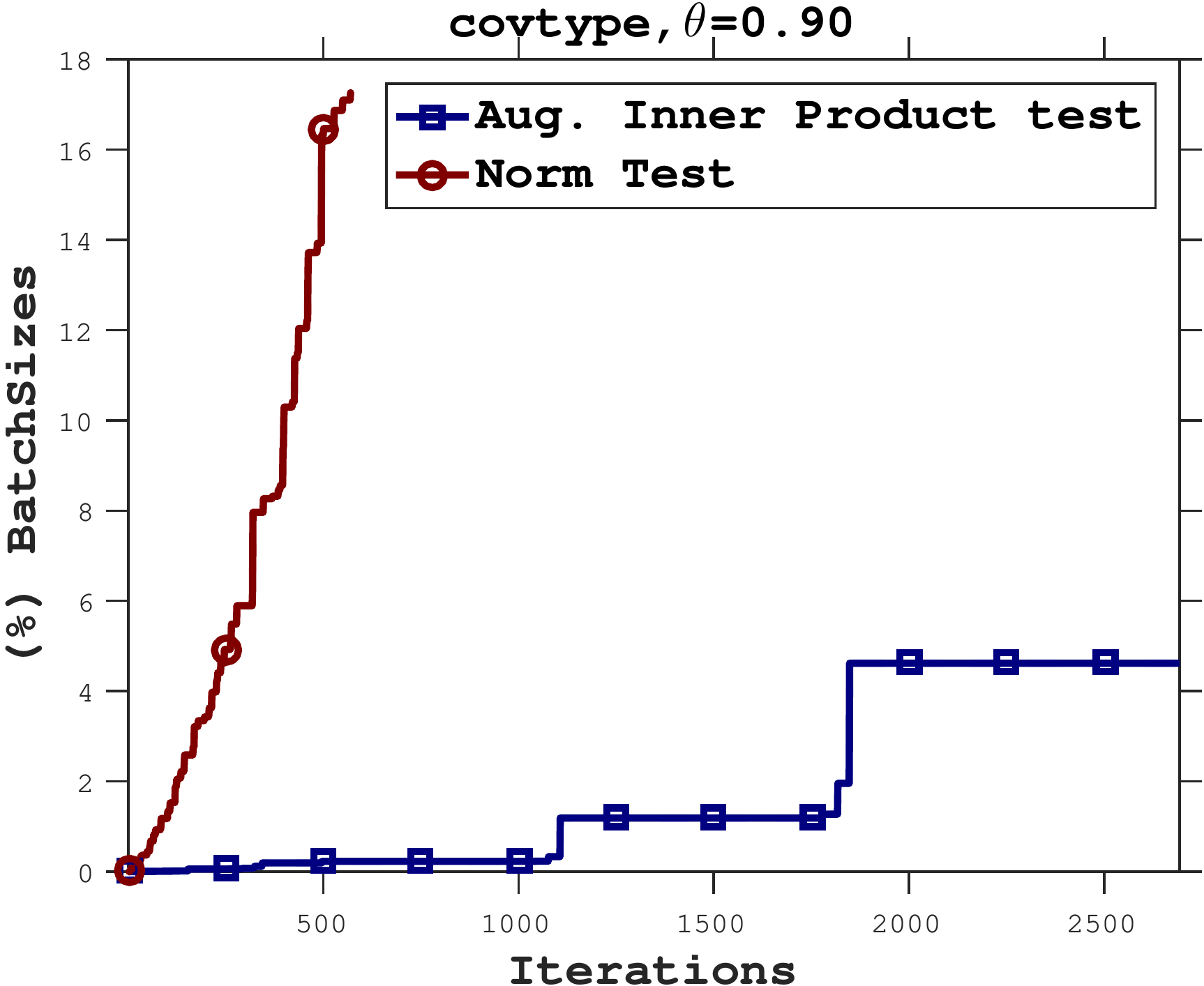}
		\includegraphics[width=0.3\linewidth]{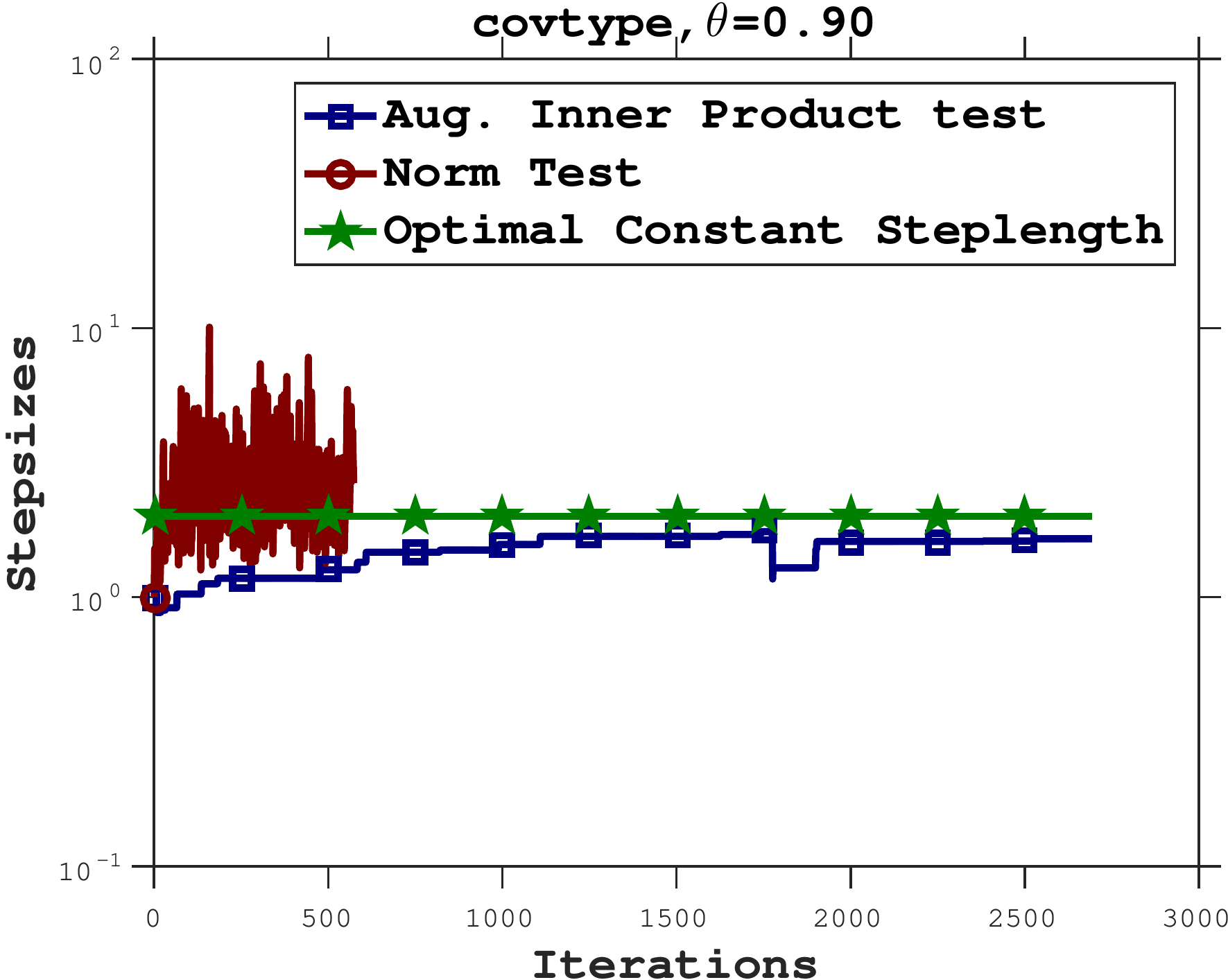}		
		\par\end{centering}
	\caption{ {\tt covertype} dataset: Performance of Algorithm~\ref{alg:complete} using the inner product test and using the norm test. Left: Function error vs. effective gradient evaluations; Middle: Batch size $|S_k|$ vs. iterations; Right: Stepsize vs. iterations. } 
	\label{covtype-expmnt2} 
\end{figure}

\begin{figure}[H]
	\begin{centering}
		\includegraphics[width=0.3\linewidth]{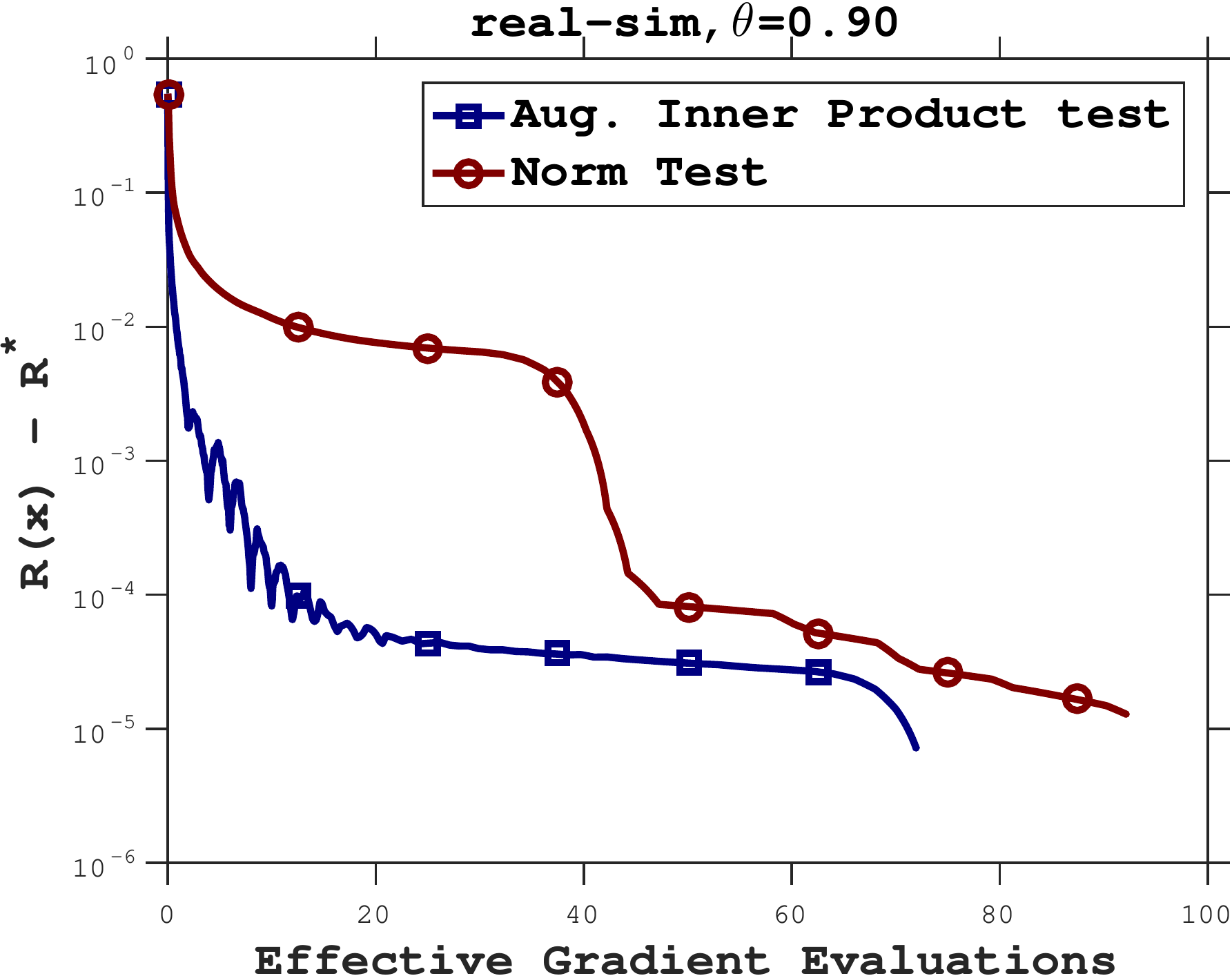}
		\includegraphics[width=0.3\linewidth]{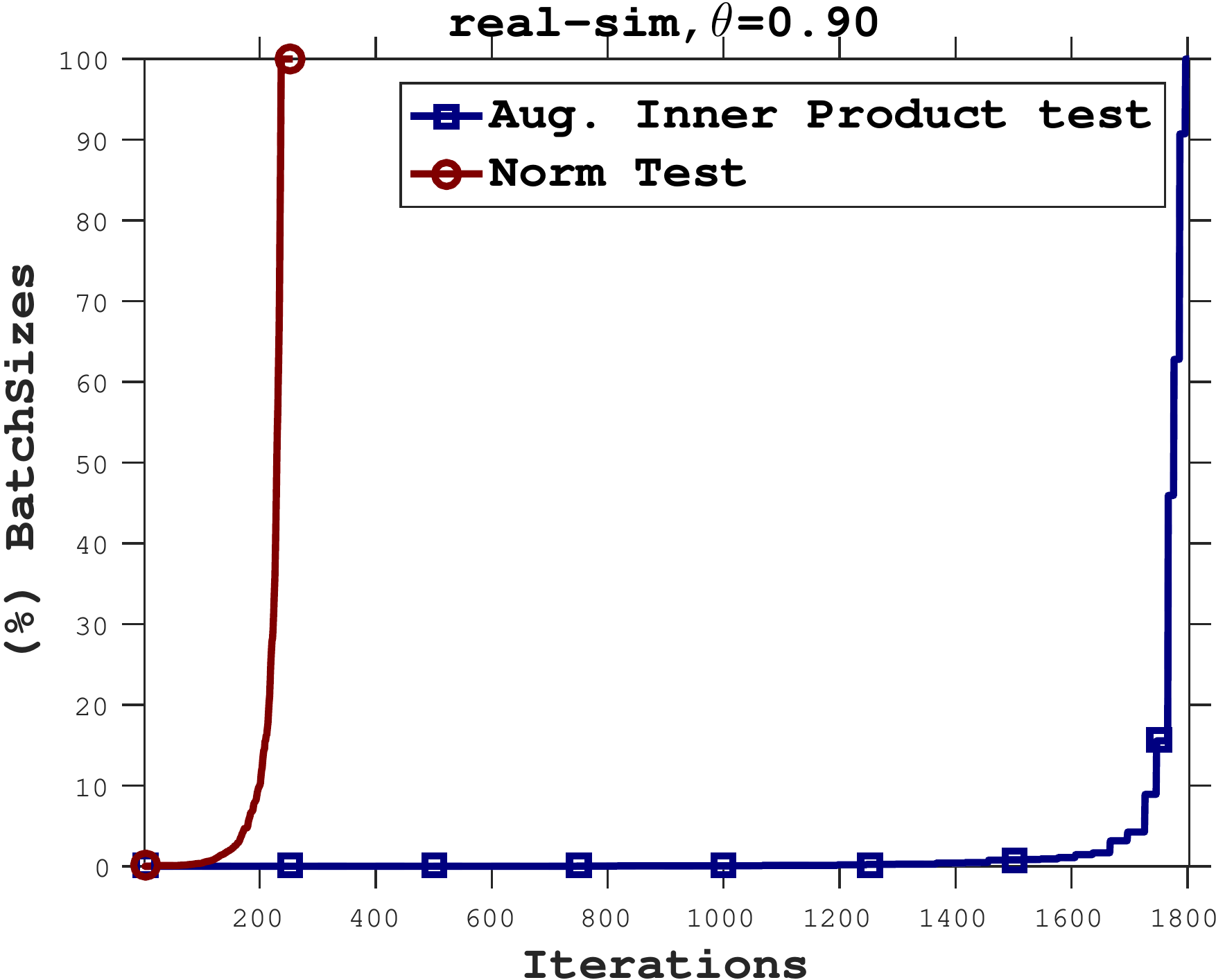}
		\includegraphics[width=0.3\linewidth]{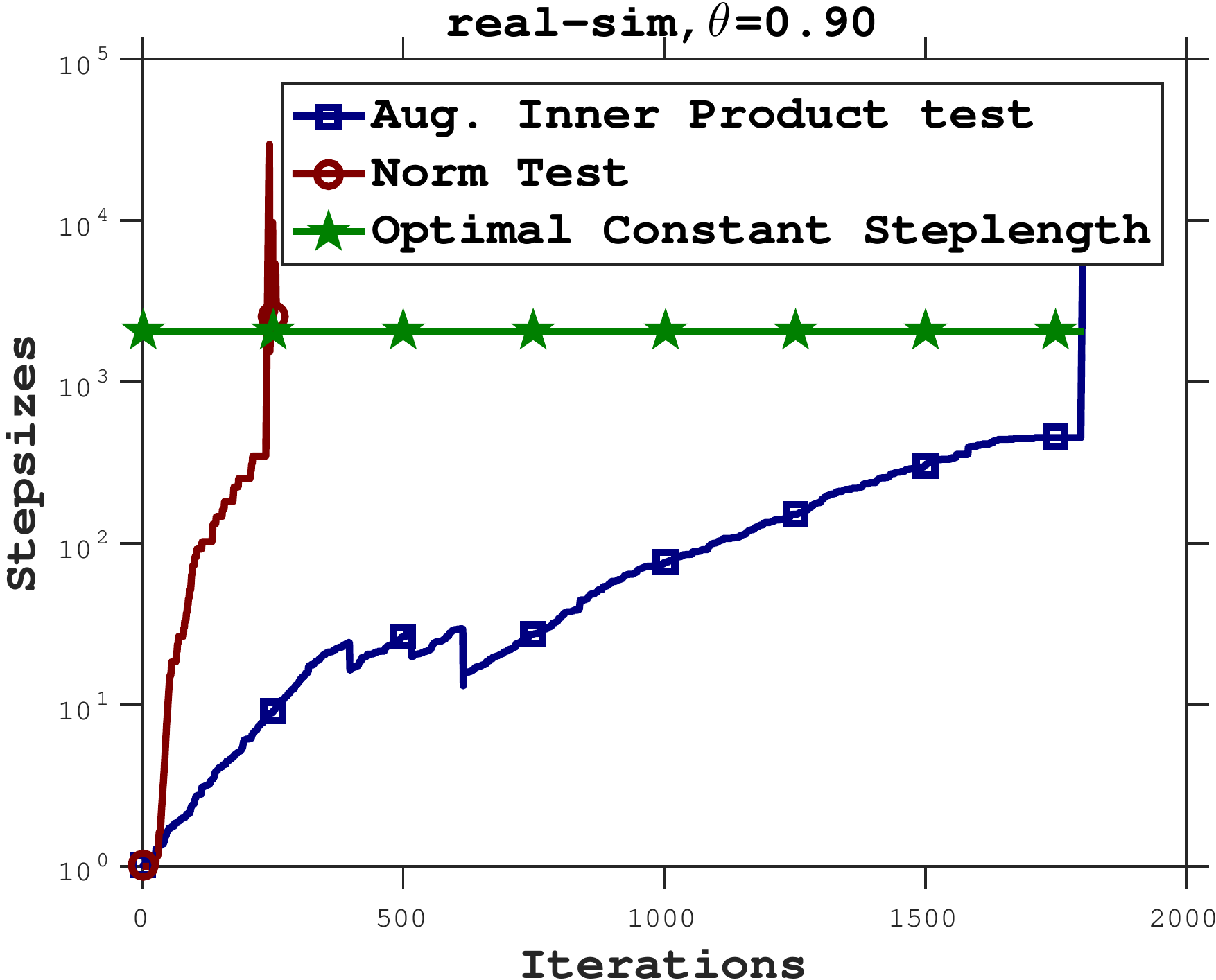}		
		\par\end{centering}
	\caption{ {\tt real-sim} dataset: Performance of Algorithm~\ref{alg:complete}  using the inner product test and using the norm test.  Left: Function error vs. effective gradient evaluations; Middle: Batch size $|S_k|$ vs. iterations; Right: Stepsize vs. iterations. } 
	\label{real-sim-expmnt2} 
\end{figure}
\begin{figure}[H]
	\begin{centering}
		\includegraphics[width=0.3\linewidth]{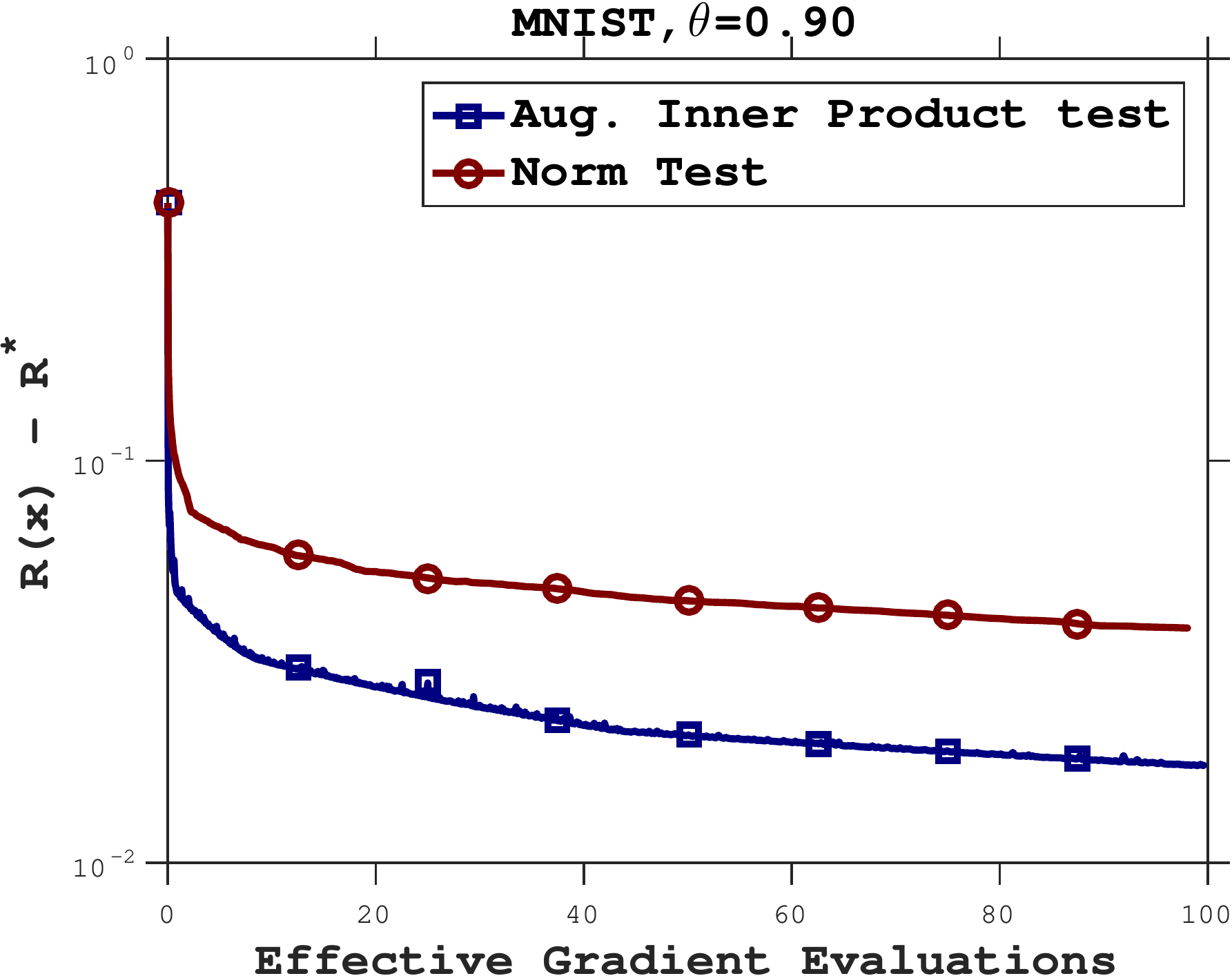}
		\includegraphics[width=0.3\linewidth]{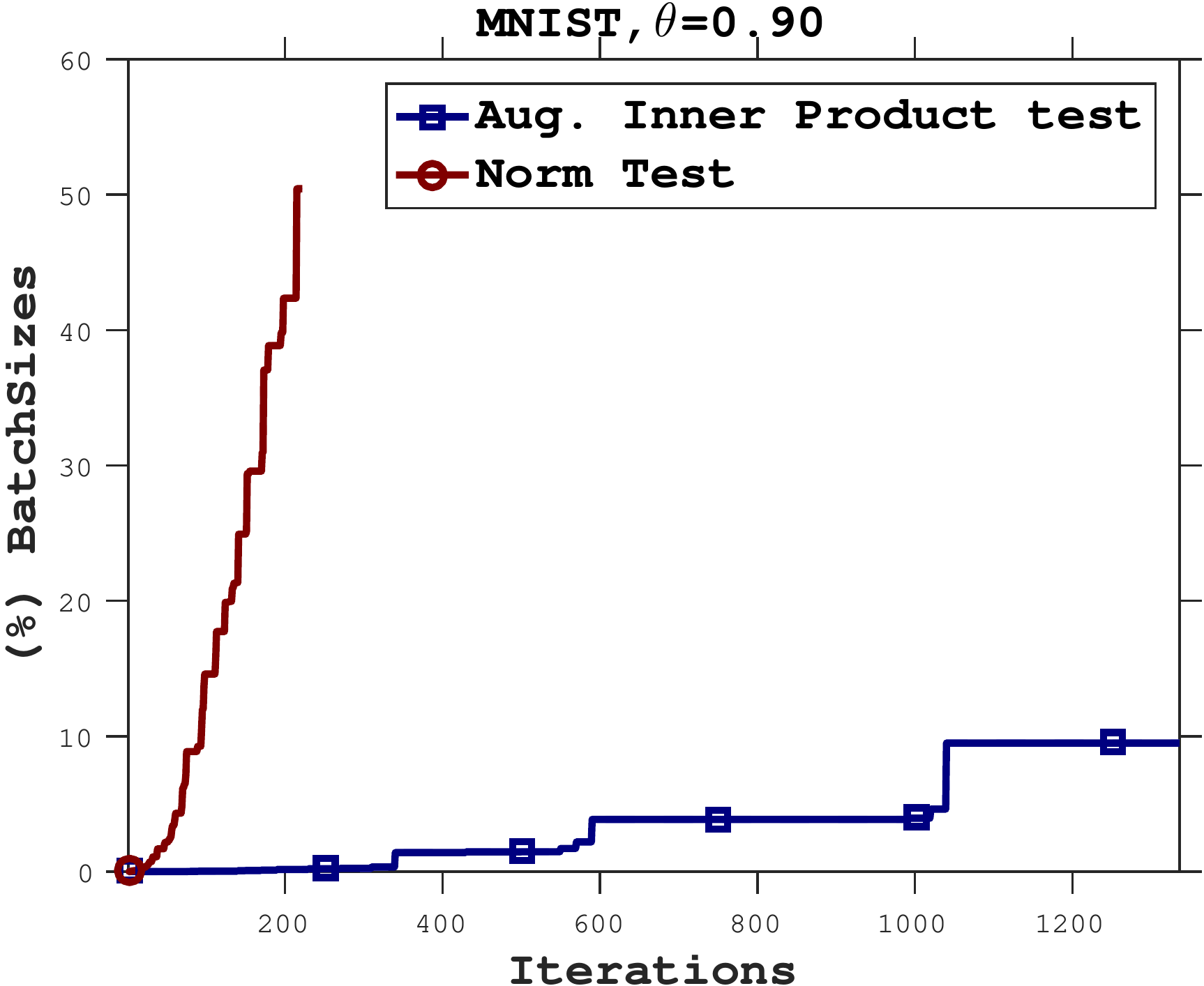}
		\includegraphics[width=0.3\linewidth]{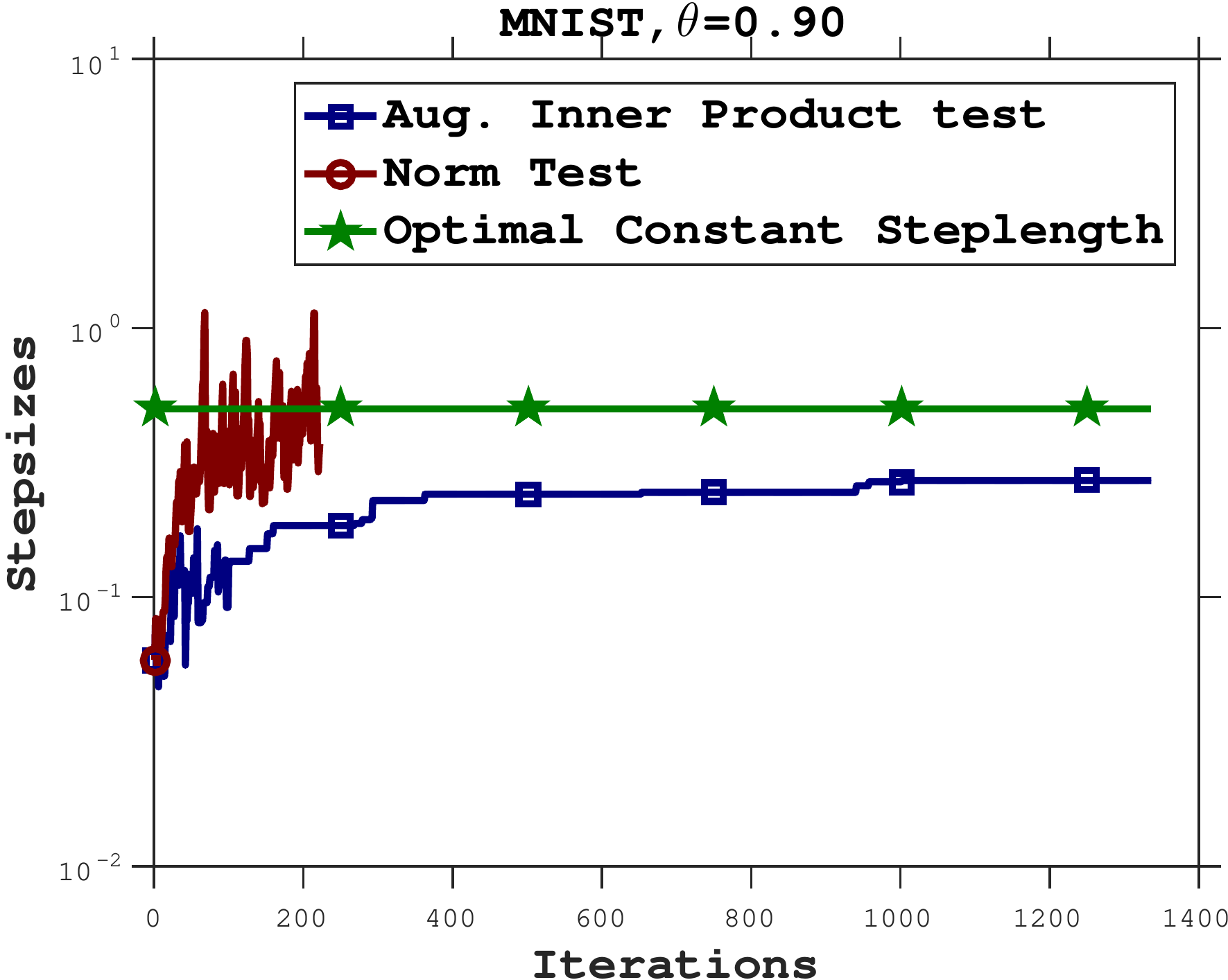}		
		\par\end{centering}
	\caption{ {\tt MNIST} dataset: Performance of Algorithm~\ref{alg:complete} using the inner product test and using the norm test. Left: Function error vs. effective gradient evaluations; Middle: Batch size $|S_k|$ vs. iterations; Right: Stepsize vs. iterations.} 
	\label{MNIST-expmnt2} 
\end{figure}
\begin{figure}[H]
	\begin{centering}
		\includegraphics[width=0.3\linewidth]{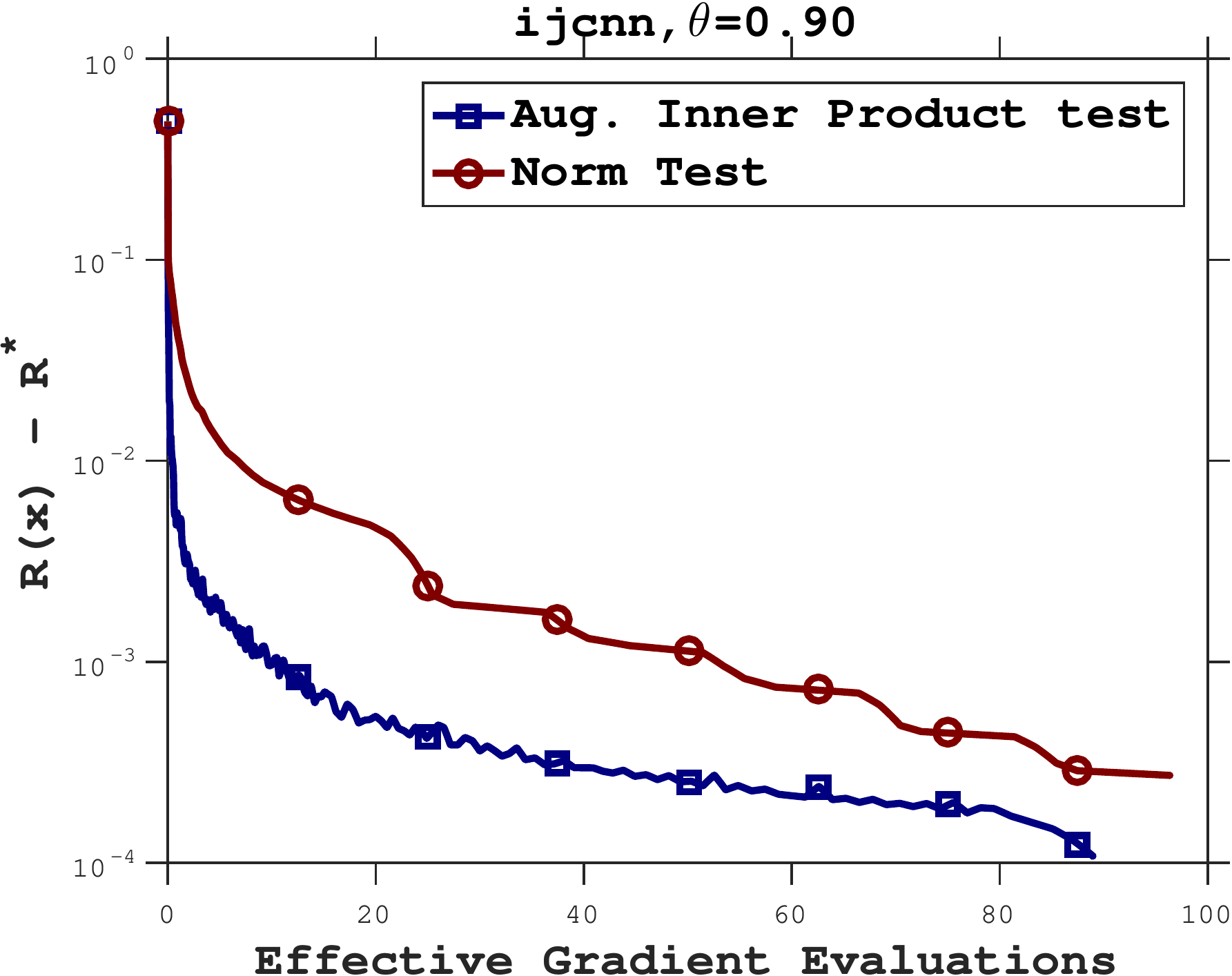}
		\includegraphics[width=0.3\linewidth]{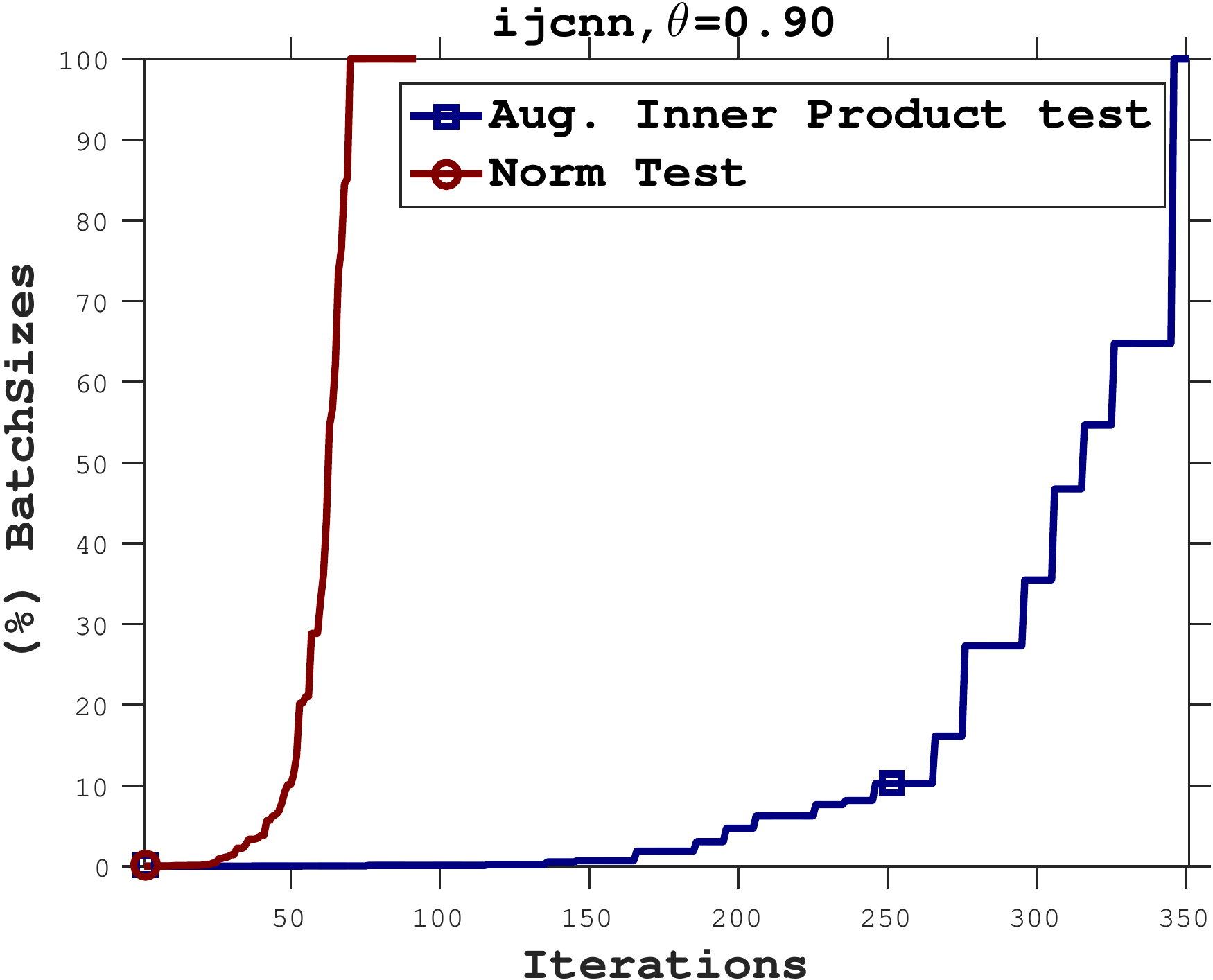}
		\includegraphics[width=0.3\linewidth]{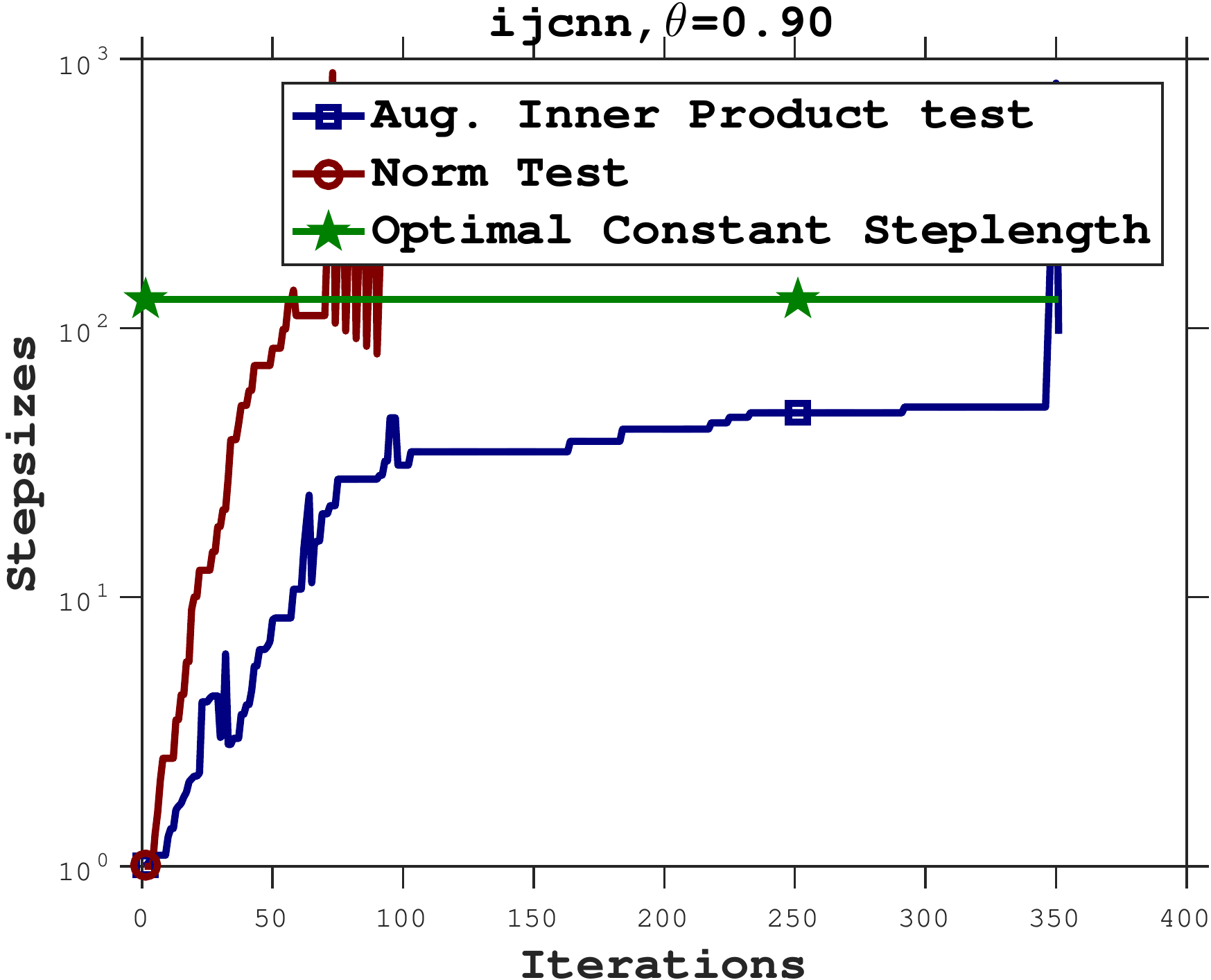}		
		\par\end{centering}
	\caption{ {\tt ijcnn} dataset: Performance of Algorithm~\ref{alg:complete} using the inner product test and using the norm test. Left: Function error vs. effective gradient evaluations; Middle: Batch size $|S_k|$ vs. iterations; Right: Stepsize vs. iterations.} 
	\label{ijcnn-expmnt2} 
\end{figure}
\begin{figure}[H]
	\begin{centering}
		\includegraphics[width=0.3\linewidth]{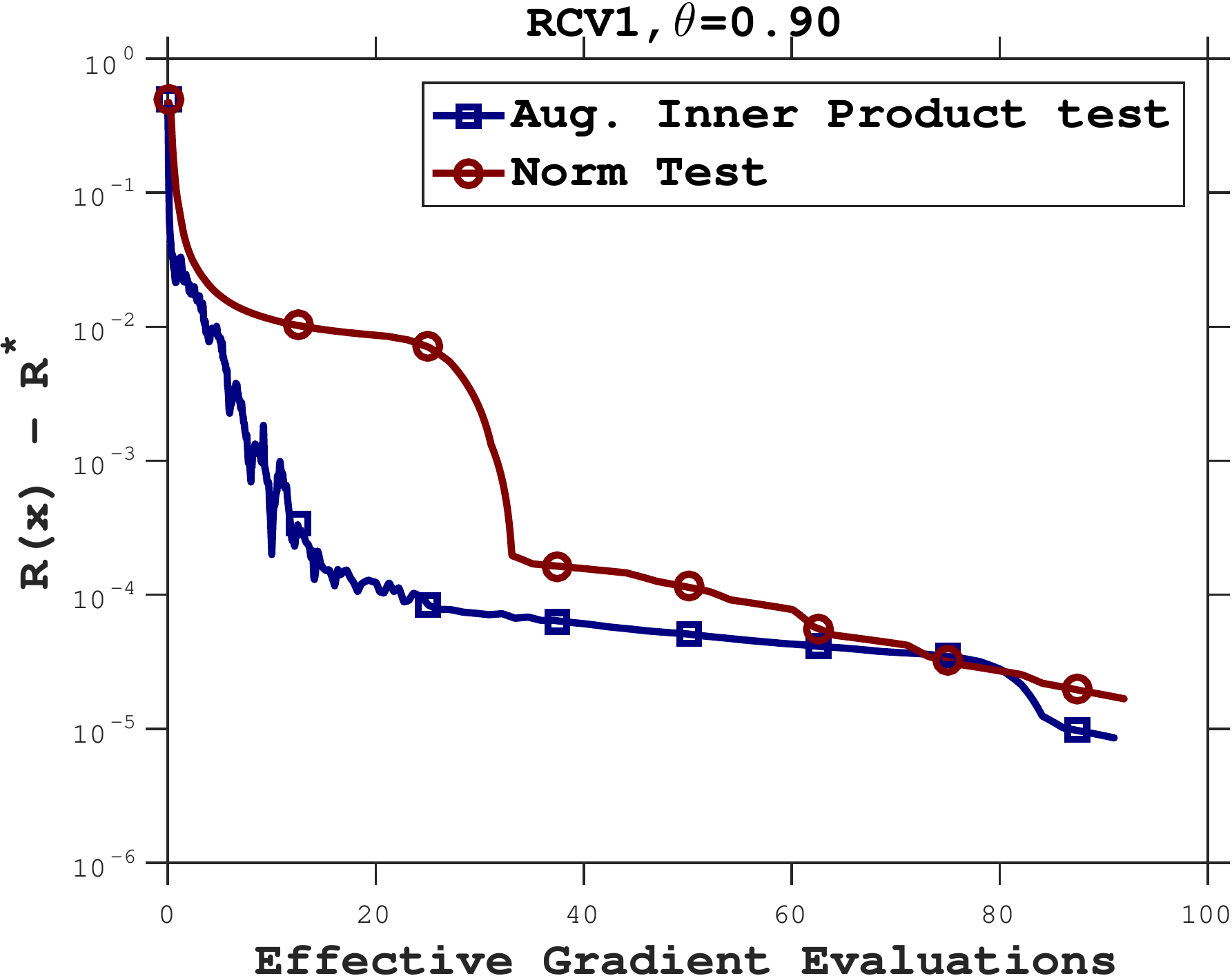}
		\includegraphics[width=0.3\linewidth]{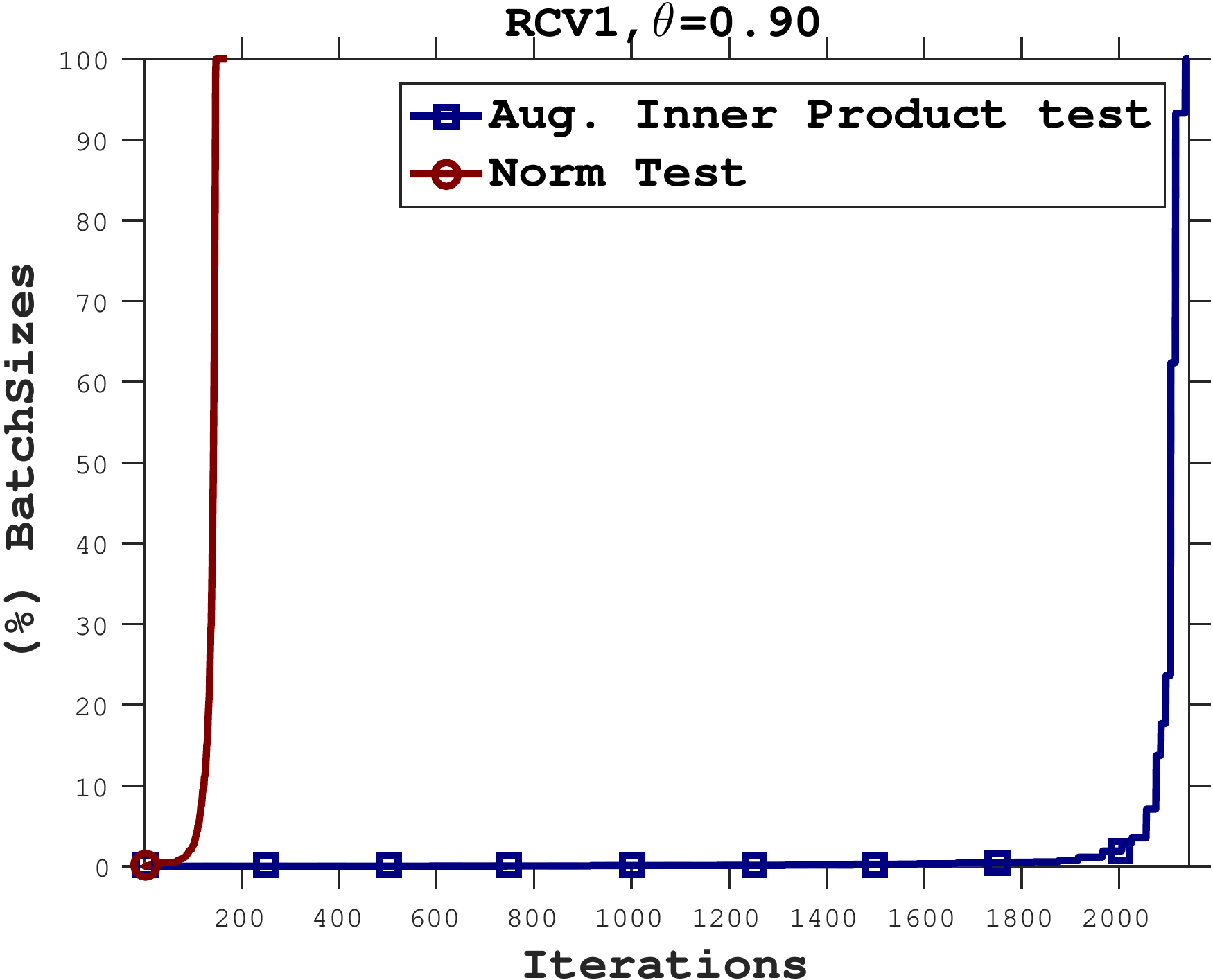}
		\includegraphics[width=0.3\linewidth]{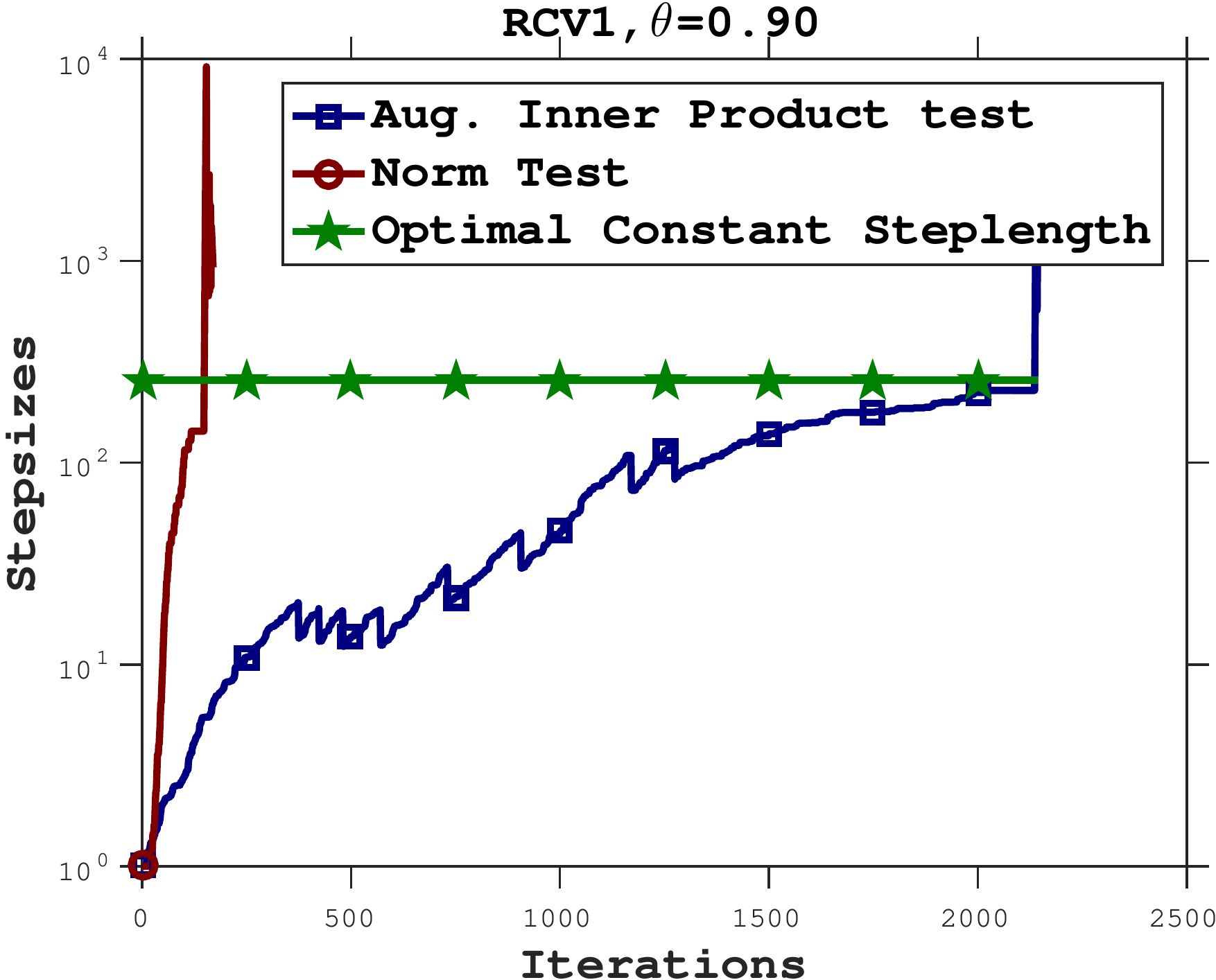}		
		\par\end{centering}
	\caption{ {\tt RCV1} dataset: Performance of Algorithm~\ref{alg:complete}  using the inner product test and using the norm test.  Left: Function error vs. effective gradient evaluations; Middle: Batch size $|S_k|$ vs. iterations; Right: Stepsize vs. iterations.} 
	\label{RCV1-expmnt2} 
\end{figure}
\begin{figure}[H]
	\begin{centering}
		\includegraphics[width=0.3\linewidth]{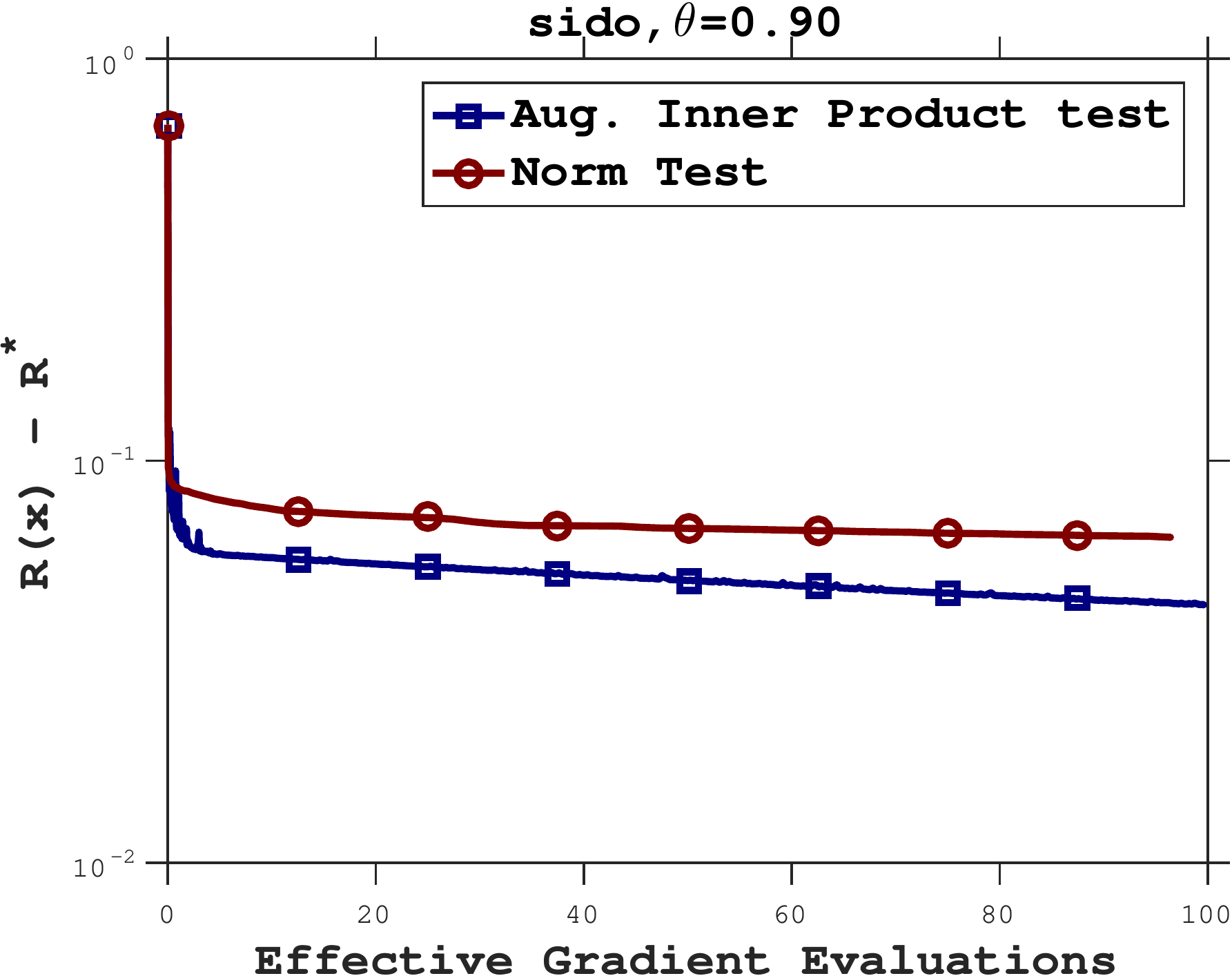}
		\includegraphics[width=0.3\linewidth]{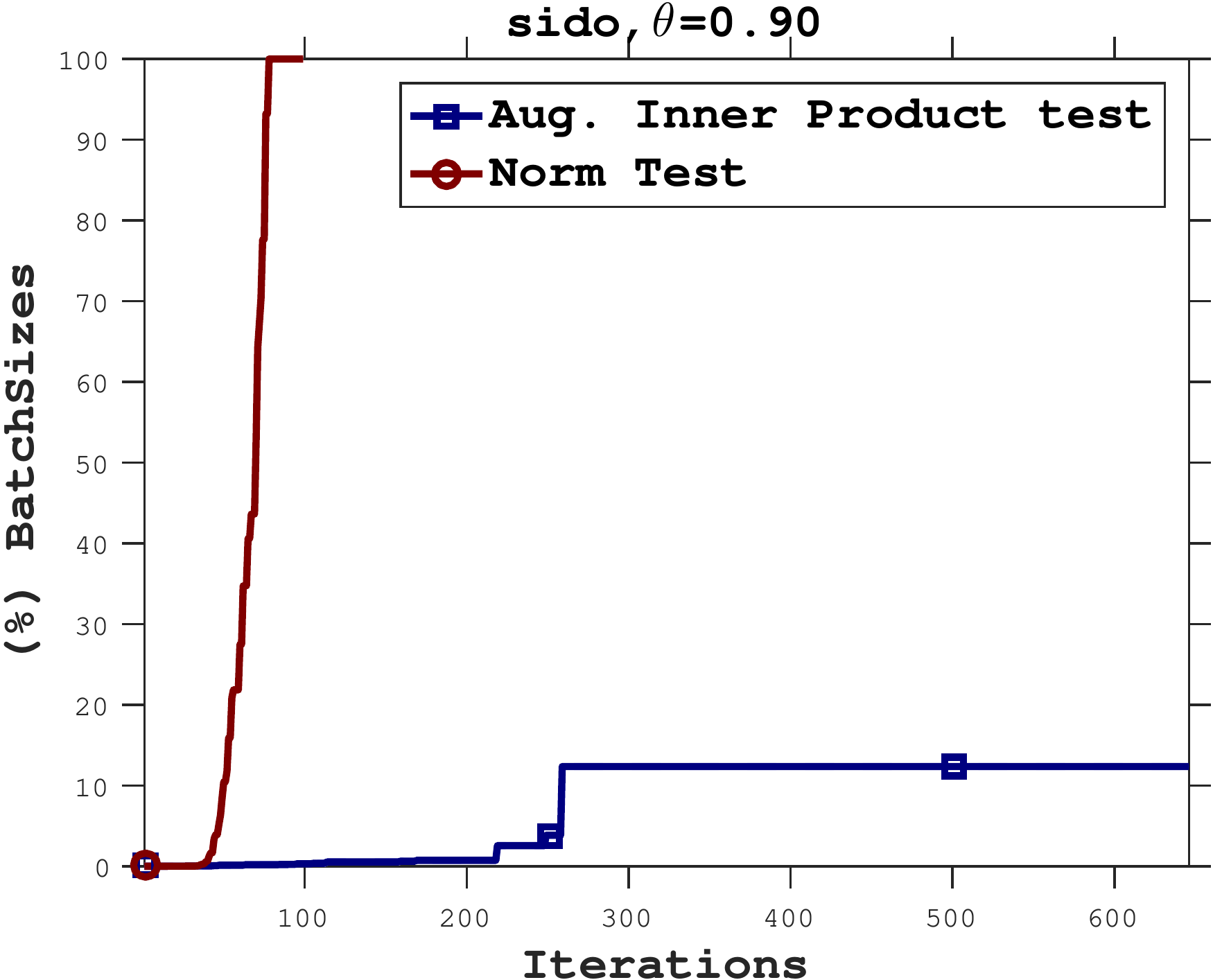}
		\includegraphics[width=0.3\linewidth]{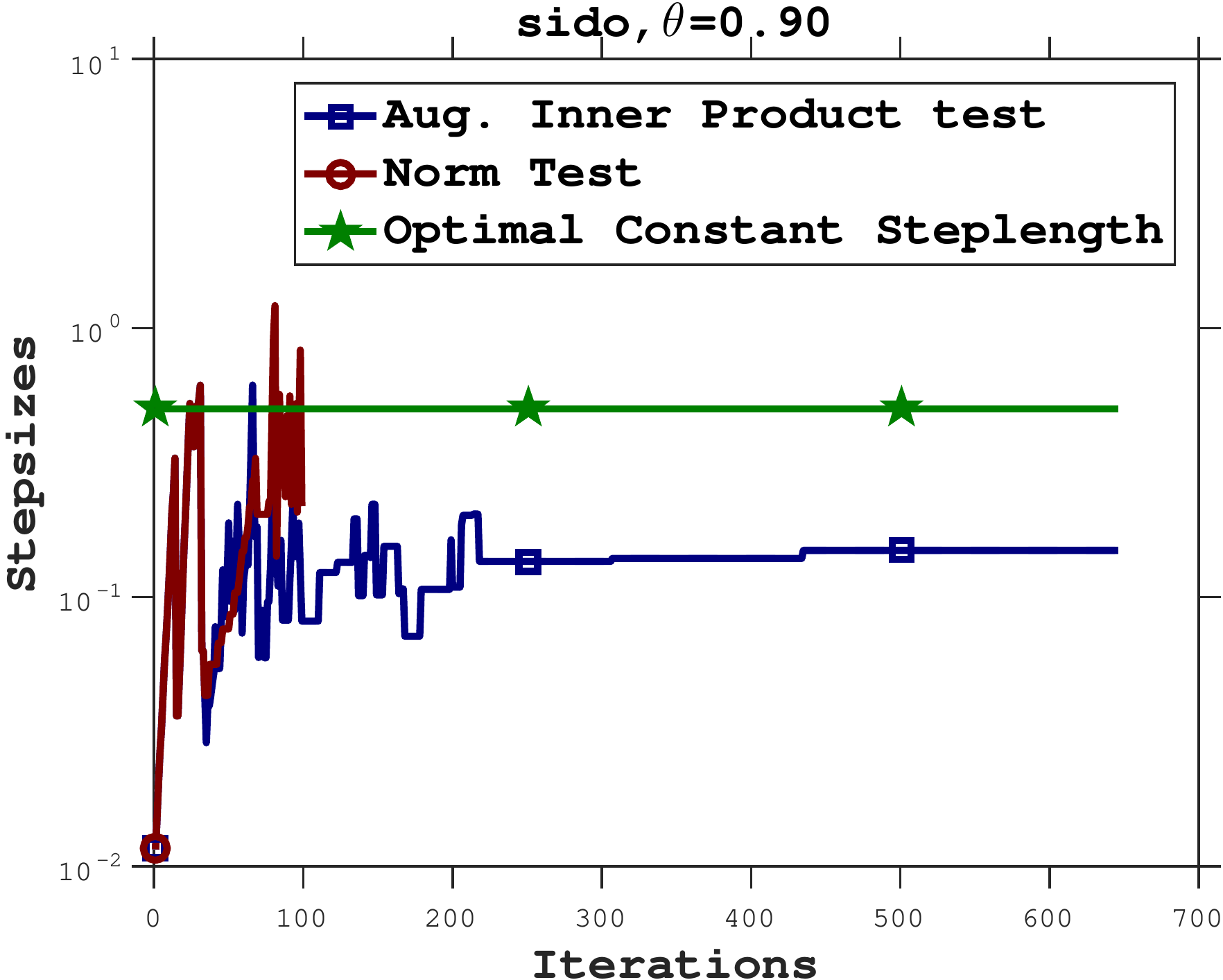}		
		\par\end{centering}
	\caption{ {\tt sido} dataset: Performance of Algorithm~\ref{alg:complete} algorithm using the inner product test and using the norm test. Left: Function error vs. effective gradient evaluations; Middle: Batch size $|S_k|$ vs. iterations; Right: Stepsize vs. iterations.} 
	\label{sido-expmnt2} 
\end{figure}
\begin{figure}[H]
	\begin{centering}
		\includegraphics[width=0.3\linewidth]{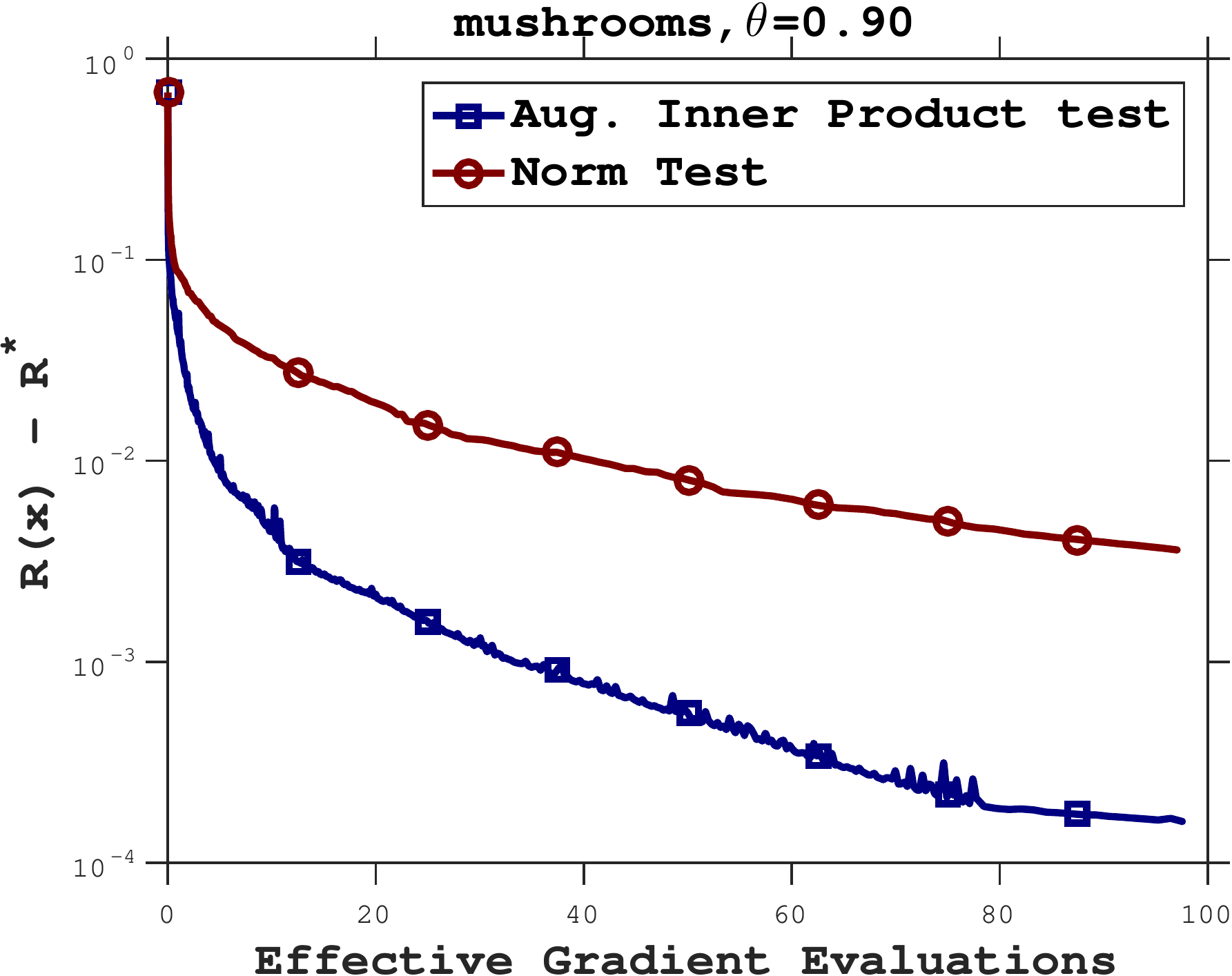}
		\includegraphics[width=0.3\linewidth]{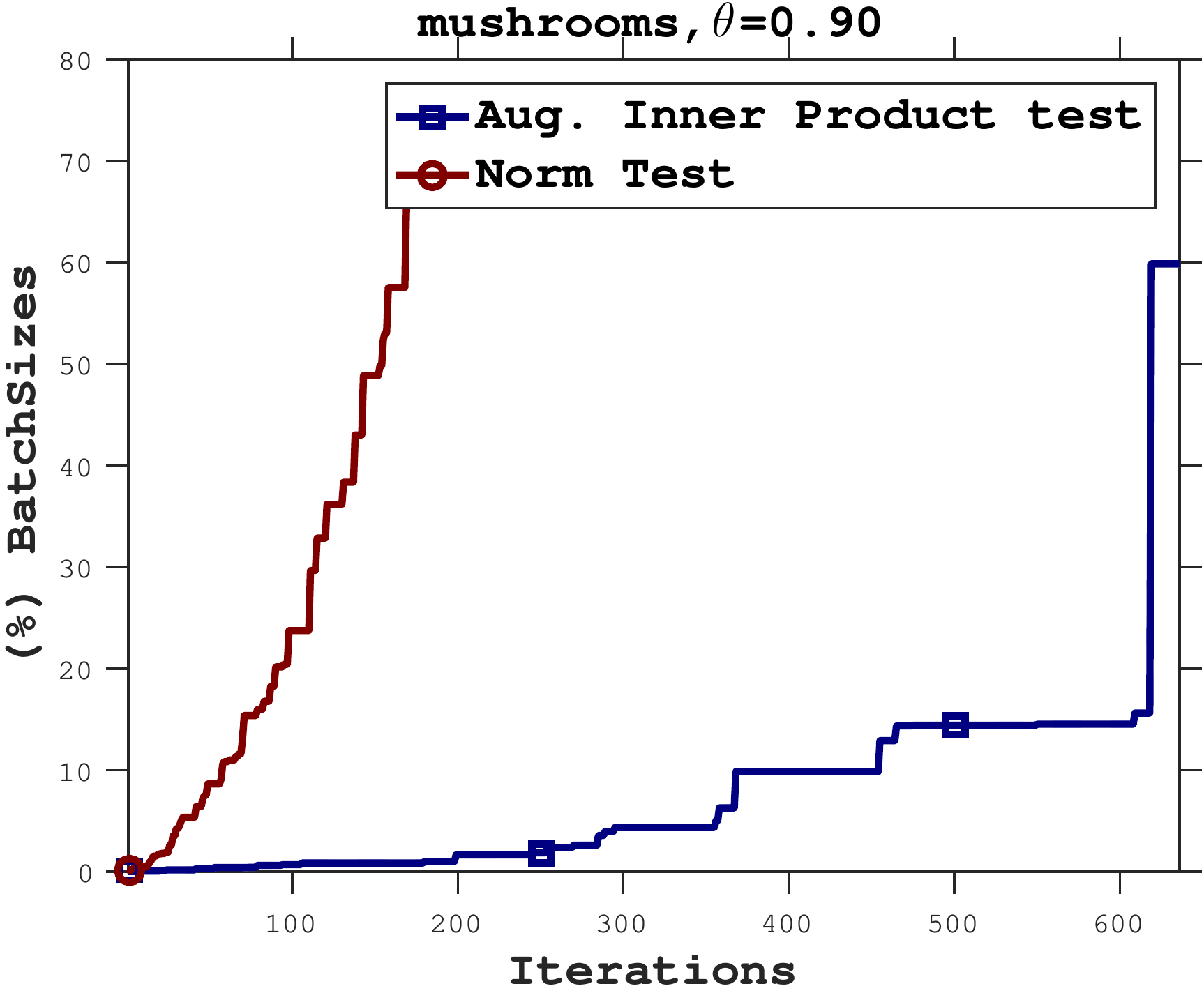}
		\includegraphics[width=0.3\linewidth]{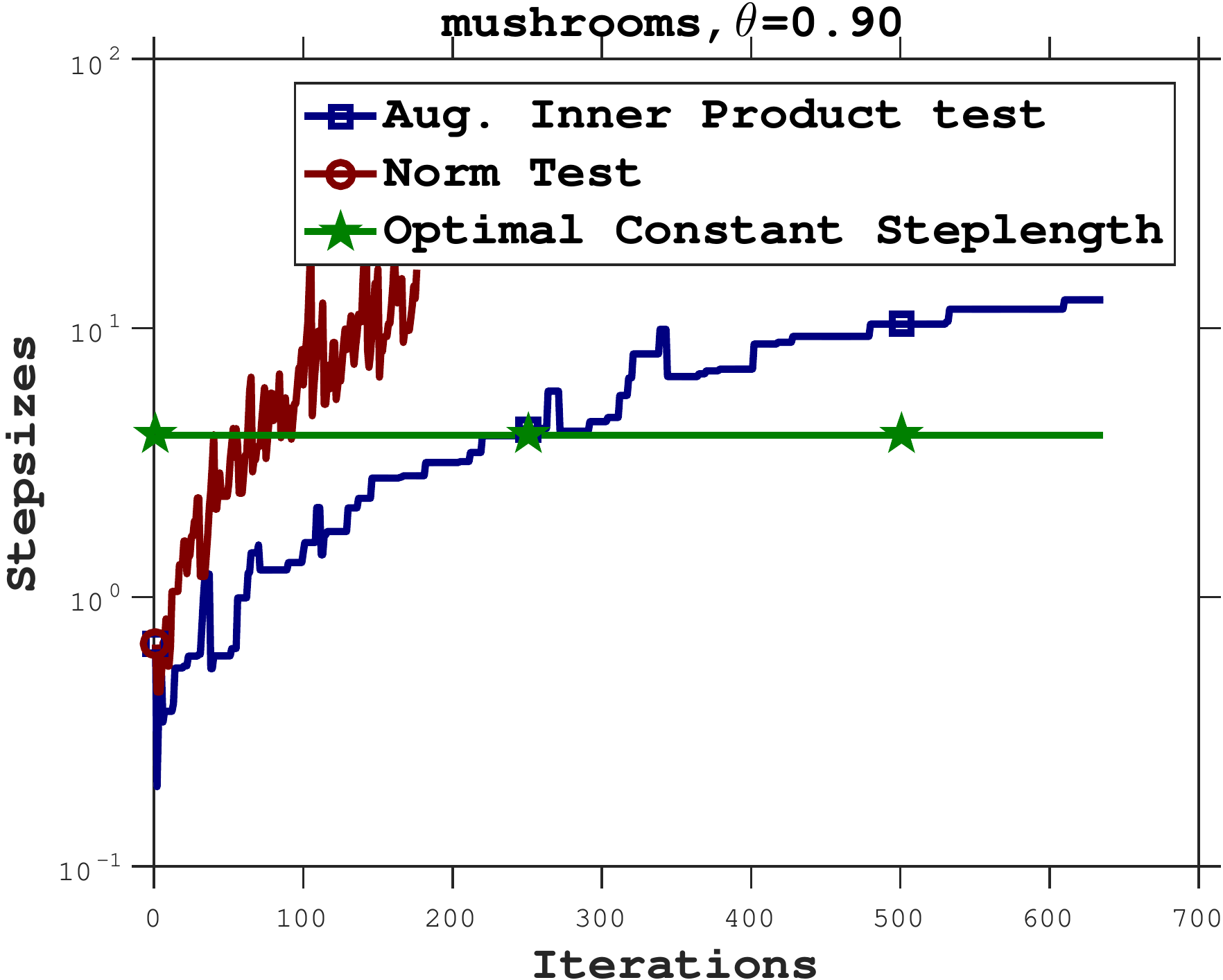}		
		\par\end{centering}
	\caption{ {\tt mushrooms} dataset: Performance of Algorithm~\ref{alg:complete}  using the inner product test and using the norm test.  Left: Function error vs. effective gradient evaluations; Middle: Batch size $|S_k|$ vs. iterations; Right: Stepsize vs. iterations.} 
	\label{mushrooms-expmnt2} 
\end{figure}
\begin{figure}[H]
	\begin{centering}
		\includegraphics[width=0.3\linewidth]{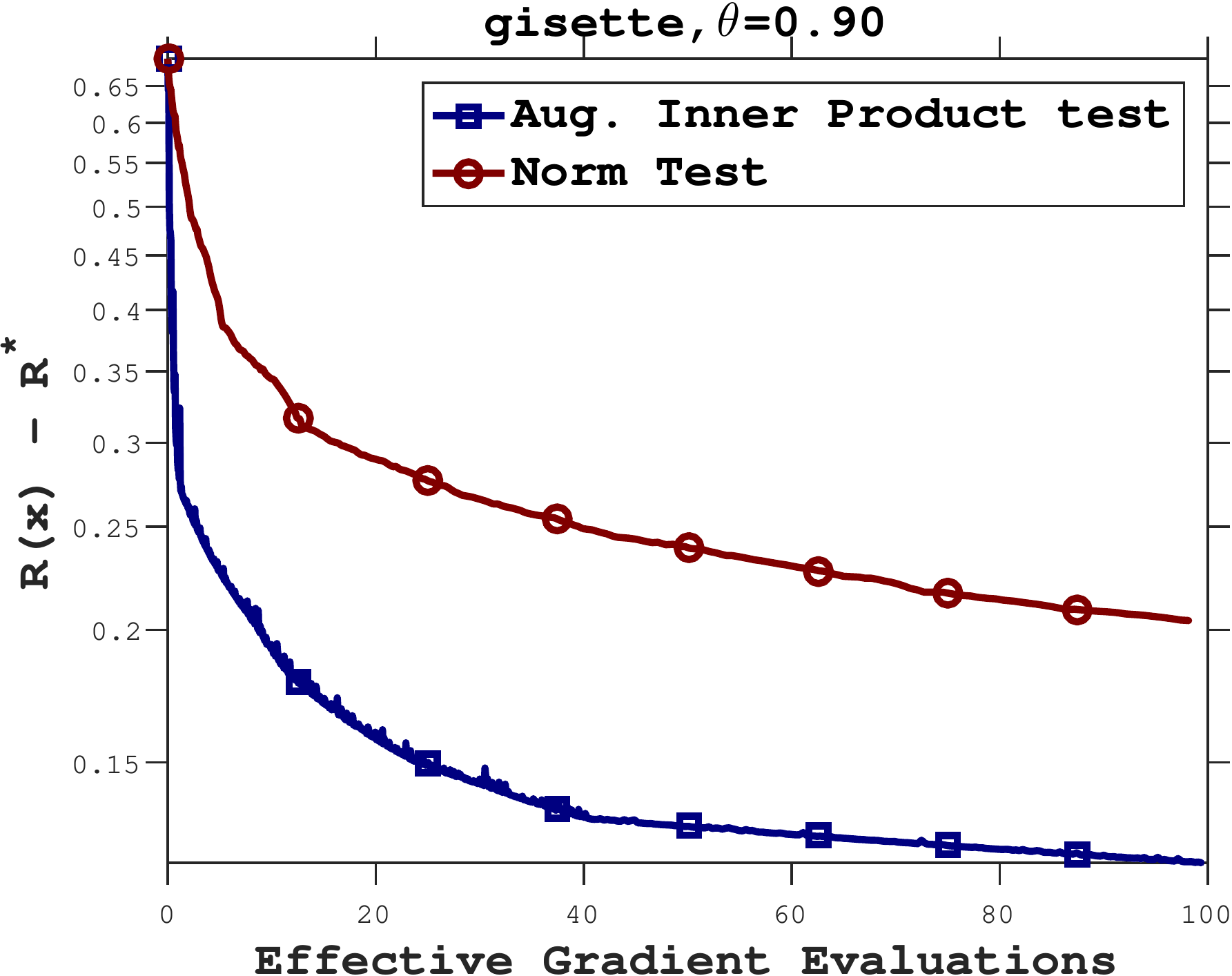}
		\includegraphics[width=0.3\linewidth]{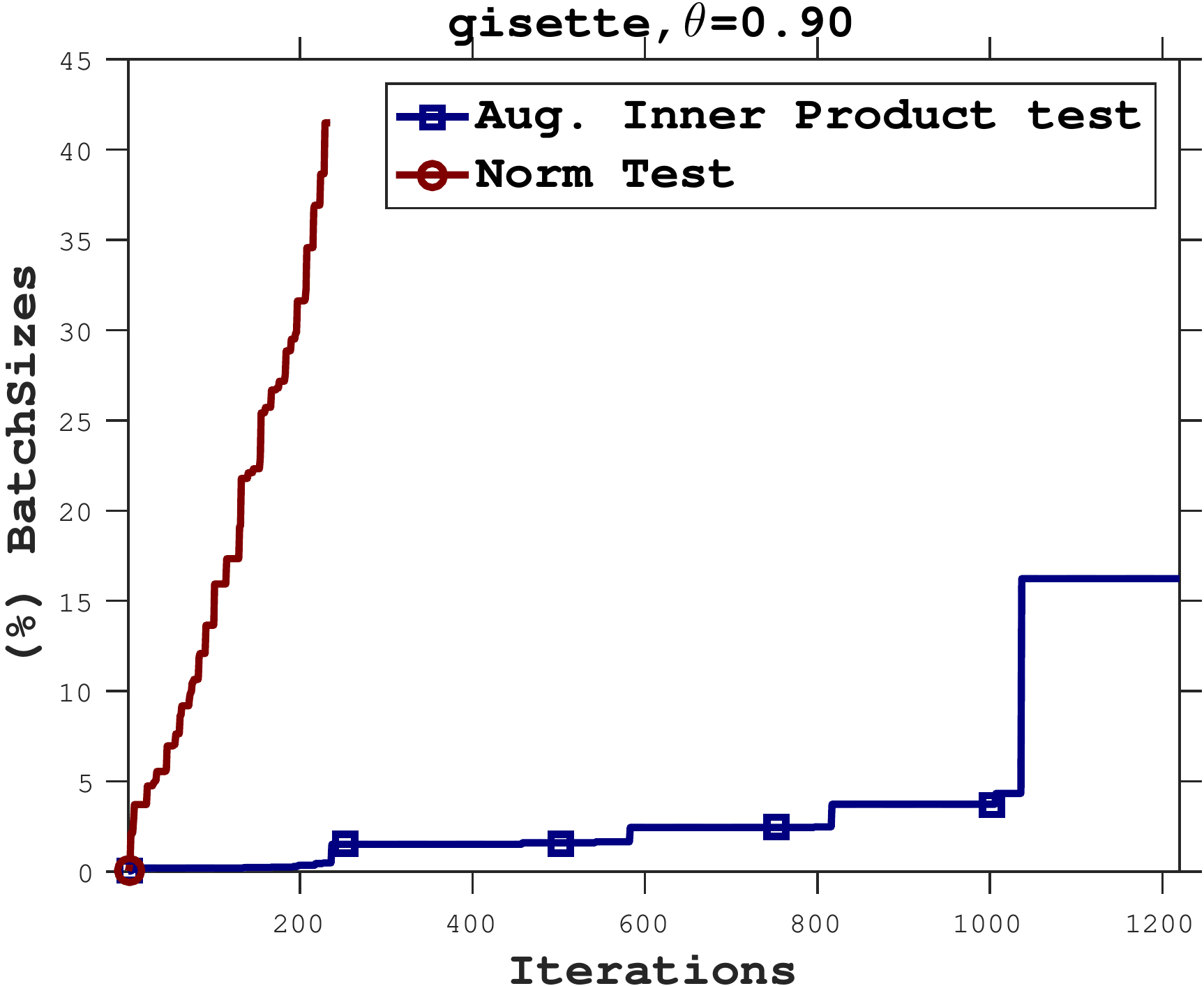}
		\includegraphics[width=0.3\linewidth]{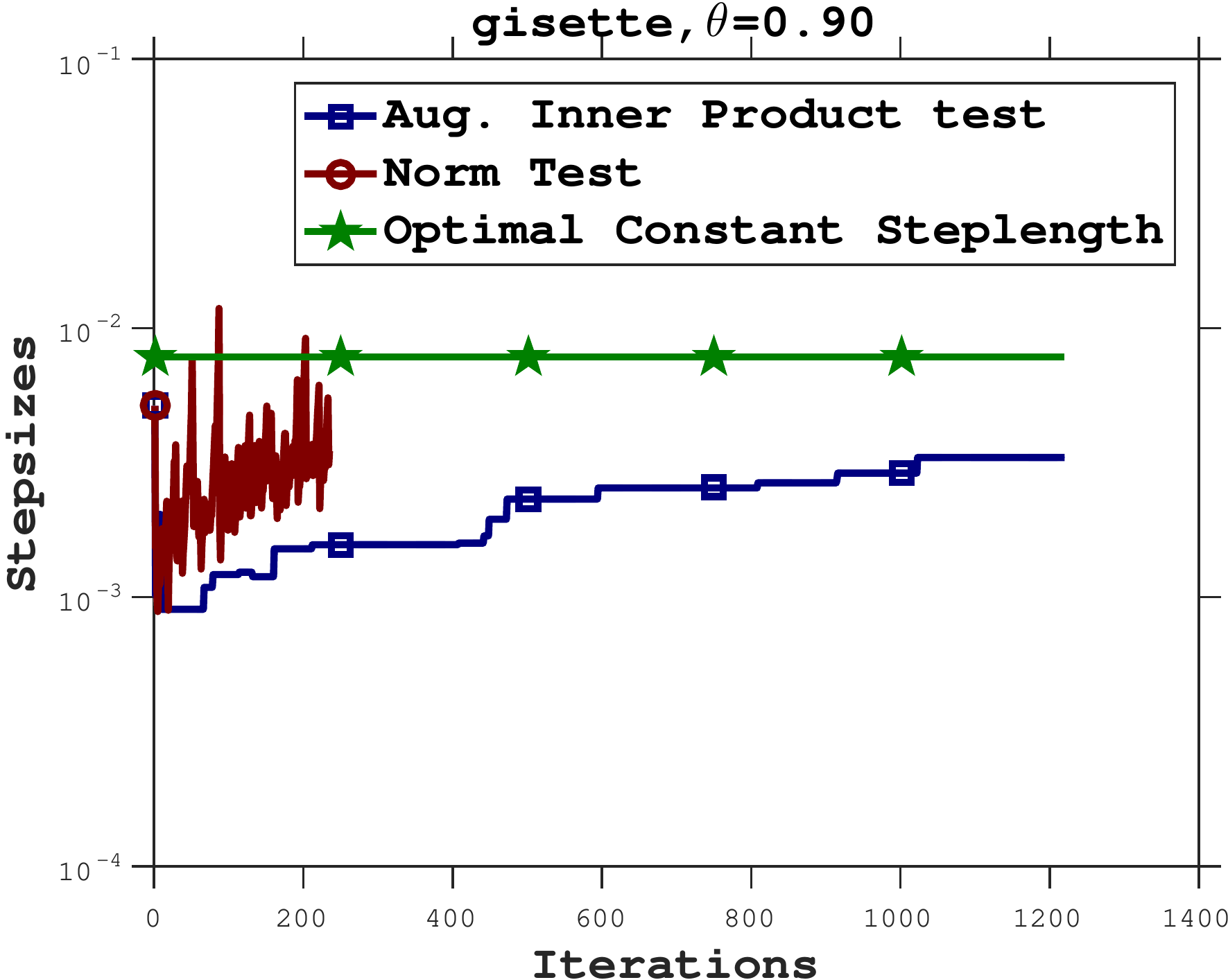}		
		\par\end{centering}
	\caption{ {\tt Gisette} dataset: Performance of Algorithm~\ref{alg:complete} using the inner product test and using the norm test. Left: Function error vs. effective gradient evaluations; Middle: Batch size $|S_k|$ vs. iterations; Right: Stepsize vs. iterations. } 
	\label{Gisette-expmnt2} 
\end{figure}
%\subsection{Experiment 2}
%\begin{figure}[H]
%	\begin{centering}
%		\includegraphics[width=0.45\linewidth]{Plots/Experiment2_Paper/gisette/gisette_Experimen2_SGD1_100_Func_Epochs_90.pdf}		
%		\includegraphics[width=0.45\linewidth]{Plots/Experiment2_Paper/gisette/gisette_Experiment2_SGD1_100_Func_Iter_90.pdf}		
%		\par\end{centering}
%	\caption{ {\bf Gisette Dataset}:  Performance of different first-order methods ($\theta = 0.9$, $\alpha_{sg} =2^{-6}$, $\alpha_{svrg} =2^{-5}$,$\alpha_{scvr} =2^{-6}$ ). } 
%	\label{synthetic-expmnt3-firstordercomparison} 
%\end{figure}
%
%\begin{figure}[H]
%	\begin{centering}
%		\includegraphics[width=0.45\linewidth]{Plots/Experiment2_Paper/mushrooms/mushrooms_Experimen2_SGD1_100_Func_Epochs_90.pdf}		
%		\includegraphics[width=0.45\linewidth]{Plots/Experiment2_Paper/mushrooms/mushrooms_Experiment2_SGD1_100_Func_Iter_90.pdf}		
%		\par\end{centering}
%	\caption{ {\bf Synthetic Dataset}:  Performance of different first-order methods ($\theta = 0.9$, $\alpha_{sg} =2^{-6}$, $\alpha_{svrg} =2^{-5}$,$\alpha_{scvr} =2^{-6}$ ). } 
%	\label{synthetic-expmnt3-firstordercomparison} 
%\end{figure}

\end{document}